\documentclass[11pt,reqno]{amsart}

\usepackage[leqno]{amsmath}
\usepackage{amsthm}
\usepackage{amsfonts, tikz-cd}
\usepackage{amssymb}
\usepackage{eucal}
\usepackage[all]{xy}
\usepackage{mathtools}
\usepackage{stmaryrd}
\usepackage{csquotes}
\usepackage{enumitem}
\usepackage{upgreek}

\setenumerate{itemsep=2pt,parsep=2pt,before={\parskip=2pt}}

\usepackage[colorlinks=true,hyperindex, linkcolor=magenta, pagebackref=false, citecolor=cyan,pdfpagelabels]{hyperref}

\usepackage{xspace}
\usepackage{epsfig,epic,eepic,latexsym,color}
\usepackage[all]{xy}
\usepackage{comment}

\usepackage[english]{babel}
\usepackage{latexsym}

\newcommand{\nc}{\newcommand}
\nc{\rnc}{\renewcommand}
\rnc{\P}{\mathbf P}
\nc{\R}{\mathbf R}
\rnc{\rm}{\mathrm}

\nc{\C}{\mathbf C}
\nc{\Q}{\mathbf Q}
\nc{\Z}{\mathbf Z}
\nc{\N}{\mathbf N}
\nc{\A}{\mathbf A}

\nc{\an}{\operatorname{an}}
\nc{\red}{\operatorname{red}}
\nc{\coker}{\operatorname{coker}}

\nc{\et}{\text{\'et}}

\nc{\htt}{\operatorname{ht}}
\nc{\Nm}{\operatorname{Nm}}
\nc{\Ker}{\operatorname{Ker}}
\nc{\mmod}{\operatorname{mod}}
\nc{\End}{\operatorname{End}}
\nc{\Aut}{\operatorname{Aut}}
\nc{\cont}{\text{cont}}
\nc{\sep}{\text{sep}}
\nc{\Hom}{\mathrm{Hom}}
\nc{\Gal}{\mathrm{Gal}}
\nc{\Spec}{\text{Spec}\,}
\nc{\RZ}{\operatorname{RZ}}

\rnc{\t}{\tau}
\nc{\mm}{\pmb{\mu}}
\rnc{\a}{\alpha}
\nc{\n}{\mathfrak n}
\nc{\m}{\mathfrak m}
\nc{\mfs}{\mathfrak s}

\nc{\p}{\mathfrak p}
\nc{\q}{\mathfrak q}

\nc{\Sym}{\operatorname{Sym}}
\nc{\codim}{\operatorname{codim}}
\nc{\rk}{\operatorname{rk}}
\nc{\GL}{\operatorname{GL}}
\nc{\SL}{\operatorname{SL}}
\nc{\Lie}{\operatorname{Lie}}
\nc{\Ind}{\operatorname{Ind}}
\nc{\Div}{\underline{Div}}
\nc{\Pic}{\mathbf{Pic}}
\nc{\uPic}{\underline{ \mathbf{Pic}}}
\nc{\rH}{\mathrm{H}}
\nc{\Spf}{\operatorname{Spf}}
\nc{\Frac}{\operatorname{Frac}}
\nc{\colim}{\operatorname{colim}}
\nc{\Spa}{\operatorname{Spa}}

\rnc{\an}{\operatorname{an}}

\nc{\xr}{\xrightarrow}

\nc{\eps}{\epsilon}
\nc{\ov}{\overline}
\nc{\ud}{\underline}
\nc{\wdh}{\widehat}

\nc{\F}{\mathcal F}
\nc{\G}{\mathcal G}
\nc{\E}{\mathcal E}
\nc{\M}{\mathcal M}

\nc{\X}{\mathfrak X}
\nc{\Y}{\mathfrak Y}
\nc{\T}{\mathfrak T}

\nc{\LL}{\mathcal{L}}
\rnc{\S}{\mathcal S}

\nc{\ra}{\rangle}

\nc{\os}{\overset}

\rnc{\O}{\mathcal O}
\nc{\J}{\mathcal J}
\theoremstyle{definition}

\newtheorem{thm1}{Theorem}[subsection]
\newtheorem{lemma1}[thm1]{Lemma}
\newtheorem{defn1}[thm1]{Definition}

\newtheorem{prop1}[thm1]{Proposition}

\newtheorem{facts1}[thm1]{Facts}
\newtheorem{rmk1}[thm1]{Remark}
\newtheorem{exmpl1}[thm1]{Example}
\newtheorem{cor1}[thm1]{Corollary}

\newtheorem{thm}{Theorem}[section]
\newtheorem{lemma}[thm]{Lemma}
\newtheorem{defn}[thm]{Definition}

\newtheorem{prop}[thm]{Proposition}
\newtheorem{setup}[thm]{Setup}

\newtheorem{conj}[thm]{Conjecture}
\newtheorem{rmk}[thm]{Remark}

\newtheorem{cor}[thm]{Corollary}

\begin{document}
\bibliographystyle{halpha-abbrv}
\title{Altered Local Uniformization of Rigid-Analytic Spaces}
\author{Bogdan Zavyalov}
\maketitle

\begin{abstract}
We prove a version of Temkin's local altered uniformization theorem. We show that for any rig-smooth, quasi-compact and quasi-separated admissible formal $\O_K$-model $\X$, there is a finite extension $K'/K$ such that $\X_{\O_{K'}}$ locally admits a rig-\'etale morphism $g\colon \X' \to \X_{\O_{K'}}$ and a rig-isomorphism $h\colon \X'' \to \X'$ with $\X'$ being a successive semi-stable curve fibration over $\O_{K'}$ and $\X''$ being a polystable formal $\O_{K'}$-scheme. Moreover, $\X'$ admits an action of a finite group $G$ such that $g\colon \X' \to \X_{\O_{K'}}$ is $G$-invariant, and the adic generic fiber $\X'_{K'}$ becomes a $G$-torsor over its quasi-compact open image $U=g_{K'}(\X'_{K'})$. Also, we study properties of the quotient map $\X'/G \to \X_{\O_{K'}}$ and show that it can be obtained as a composition of open immersions and rig-isomorphisms.
\end{abstract}
\tableofcontents

\section{Introduction}
\subsection{Historical Overview}

The stable modification theorem \cite{DM} of P. Deligne and D. Mumford  says that for a discrete valuation ring $R$ and a smooth proper curve $X$ over the fraction field $K$, there is a finite separable extension $K \subset K'$ such that the base change $X_{K'}$ extends to a semi-stable curve over the integral closure of $R$ in $K'$. This theorem is crucial for many geometric and arithmetic applications since it allows one to reduce many questions about curves over the fraction field $K$ of $R$ to questions about {\it semi-stable curves} over a finite separable extension of $K$. For example, this theorem plays an important role in the proof of Mordell's Conjecture by G. Faltings (look at \cite{Mor} for a discussion of this proof). \smallskip

Over the years people tried to generalize this statement to a more general set-up. For example, a ``Lemma of Gabber''~\cite{Del} roughly says that the same result may be achieved over a quasi-compact and quasi-separated base\footnote{The result is stated in a weaker form, but using the ``spreading out'' techniques developed in \cite[IV\textsubscript{3}]{EGA}
 and \cite{ThTr} the case of a qcqs base can be reduced to \cite[Lemma 1.6]{Del}.} $S$ after a base change along some proper and surjective morphism $S' \to S$. Another generalization was obtained by J.\ de Jong in \cite{DJ2}: he shows that any projective relative curve $f\colon X \to S$ can be made semi-stable by means of a generically \'etale alteration $S'\to S$ and a further modification $X' \to X_{S'}$. This theorem has many important consequences in algebraic geometry (e.g. resolution of singularities up to an alteration), but it has a disadvantage that one can not control the \'etale locus of the morphism $S' \to S$ (and $X' \to X_{S'}$). For example, suppose that $X$ is smooth (or just semi-stable) over an open $U\subset S$. Then de Jong's approach is not robust enough to allow one to choose an $U$-\'etale alteration $S' \to S$ such that $X_{S'}$ admits a semi-stable modification. \smallskip

This difficulty was recently overcome by M. Temkin in \cite{T2}. He shows there that given a relative curve (see Definition~\ref{def-curve}) $f\colon X \to S$ that is semi-stable over an (quasi-compact and schematically dense) open subset $U \subset S$, one can find an ``$U$-\'etale covering'' (see Definition~\ref{U-etale cover}) $S' \to S$ such that $X_{S'}$ admits a semi-stable modification. This proof uses completely new ideas (compared to all older proofs) related to the non-archimedean geometry. Roughly, Temkin uses the notion of a Riemann-Zariski space to reduce the case of an arbitrary quasi-compact and quasi-separated base $S$ to the case of a spectrum of a complete rank-$1$ (possibly not discrete) valuation ring. He then treats this case using rigid-analytic techniques.
\smallskip

The case of higher dimensional families is much harder, and not much is known besides the characteristic $0$ case. It is shown in~\cite{TorEmb} that given a characteristic $0$ field $k$, a $k$-curve $C$, and a finite type morphism $X \to C$ that is smooth over the generic point of $C$, there is a finite morphism $C' \to C$ such that $X_{C'}$ admits a semi-stable modification. The main techniques in the proof are resolution of singularities and toroidal geometry. The case of higher dimensional base (and higher-dimensional fibers) was recently solved by K. Adiprasito, G. Liu and M. Temkin in \cite{AdipLiuTem} using techniques based on the previous work of D. Abramovich and K. Karu \cite{AbrKaru}. \smallskip

To the best of our knowledge, not much is known in the case of finite or mixed characteristic schemes. However, there are some positive results in this direction. M. Temkin proved in \cite{T3} that for any valuation ring $R$ of rank-$1$\footnote{The case of a discrete valuation ring $R$ was obtained before in the work of U. Hartl \cite{Hartl}.} and a finite type flat $R$-scheme $X$ with a smooth fiber over the fraction field $K$, there is a finite extension of valued fields $K \subset K'$ with ring of integers $R'$ such that the base change $X_{R'}$ admits a $K'$-\'etale covering (see Definition~\ref{U-etale cover} and Remark~\ref{rmk:K=U}) $X' \to X_{R'}$ with a strictly semi-stable $R'$-scheme $X'$. The proof is based on his previous work on the Stable Modification Theorem and Relative Riemann-Zariski spaces in \cite{T2}. We note that this is pretty far from the case of the actual Semi-Stable Reduction Theorem since the argument is local on $X$, so it can not control properness of $X' \to X_{R'}$. Also, usually the algorithm produces $X'$ that is not birational to $X_{R'}$. \smallskip

It is important to note that  Temkin's result actually shows more. Namely, his theorem applies not only to usual $R$-schemes, but also to formal $R$-schemes (in the case of a complete $R$). In particular, it is shown in \cite[Theorem 3.3.1]{T3} that given a complete valuation ring $R$ of rank-$1$ and a rig-smooth admissible formal $R$-scheme $\X$, there is a finite valued field extension $K \subset K'$ with the valuation ring $R'$ such that $\X_{R'}$ admits a rig-\'etale covering (see Definition~\ref{different-C-mod}) $\X' \to \X_{R'}$ with a strictly semi-stable admissible formal $R'$-scheme $\X'$.

\subsection{Our results}

In this paper, we prove a version of the local altered uniformization theorem for smooth rigid-analytic varieties following the ideas from \cite{T3}. We start with proving a slightly more refined version of 
Temkin's Stable Modification Theorem~\cite[Theorem 2.3.3]{T2} that will be very useful for our later purposes:

\begin{thm}\label{example-1} (Theorem~\ref{stable-modification-general}) Let $U\subset S$ be a schematically dense quasi-compact open subset of a quasi-compact and quasi-separated scheme $S$, and let $f\colon X \to S$ be an $S$-curve that is semi-stable over $U$ (see Definition~\ref{def-ss}). Then there exist 
\begin{itemize}\itemsep0.5em
	\item A projective $U$-modification (see Definition~\ref{U-mod}) $h\colon S' \to S$ with a finite open Zariski covering $\cup_{i=1}^n V'_i = S'$ by quasi-compact opens $V'_i \subset S'$.
	\item A finite group $G_i$ and a finite, finitely presented, faithfully flat, and $U$-\'etale $G_i$-invariant morphism $t_i\colon W'_i \to V'_i$ for each $i\leq n$. In particular, the morphism $t\colon W'=\sqcup_{i=1}^n W'_i \to S$ is a $U$-\'etale covering (see Definition~\ref{U-etale cover}). 
\end{itemize}

satisfying the following properties:

\begin{enumerate}
	\item The induced morphisms $t_{i,U}\colon W'_{i, U} \to V'_{i, U}$ are $G_i$-torsors.
	\item Each $X_{W'_i}$ admits a $W'_i$-stable $U$-modification (see Definition~\ref{U-stable}) $g_i\colon X'_i \to X_{W'_i}$.
\end{enumerate}

\end{thm}

\begin{rmk} We actually prove a slightly more general result. Namely, Theorem~\ref{example-1} holds under fewer assumptions on $U$ (see Theorem~\ref{stable-modification-general} and Setup~\ref{setup5}). We only mention here that, in particular, Theorem~\ref{example-1} holds if $U$ is the generic point of a quasi-compact, quasi-separated integral scheme $S$. 
\end{rmk}

The main difference between our theorem and \cite[Theorem 2.3.3]{T2} is that we gain a better control over the $U$-\'etale morphism (see Definition~\ref{U-etale}) $W'\to S$\footnote{$W'$ is denoted as $S'$ in \cite[Theorem 2.3.3]{T2}.} such that $X_{W'}$ admits a stable $U$-modification. Namely, our $W'\to S$ has the form 
\[
    W'\coloneqq\bigsqcup_{i=1}^n W'_i \to \bigsqcup_{i=1}^n V'_i \to S
\] 
that it is not merely an abstract $U$-\'etale covering but is a disjoint union of $G_i$-invariant finite, finitely presented and faithfully flat morphisms over the Zariski open covering of a projective $U$-modification $\bigcup_{i=1}^n V'_i=S' \to S$. Moreover, the $U$-restriction of each $W'_i\to S$ becomes a $G_i$-torsor over its open image. We later use these further properties in a crucial way. \smallskip

Theorem~\ref{example-1} allows us to prove a version of Temkin's local altered uniformization Theorem \cite[Theorem 3.3.1]{T3}. We now recall its statement. For any smooth, quasi-compact and quasi-separated rigid-analytic space $X$ over a complete non-archimedean field $K$ and an admissible formal model $\X$ over $\Spf \O_K$, there is a finite separable field extension $K\subset K'$ and a rig-\'etale map $g\colon \X' \to \X_{\O_{K'}}$ such that $\X'$ is strictly semi-stable over $\Spf \O_{K'}$. In this article, we weaken the condition on $\X'$, but we get some control over the structure of the morphism $g\colon  \X' \to \X_{\O_{K'}}$ instead. Namely, in our version of Temkin's local altered uniformization result, $\X'$ will not be (strictly) semi-stable over $\Spf \O_{K'}$, but it rather will be only polystable over $\Spf \O_{K'}$. In exchange, we will get a better control over the generic fiber of $\X'$ (in terms of an action of a finite group).

\begin{thm}\label{example-2} (Theorem~\ref{main-main}) Let $X$ be a quasi-compact and quasi-separated smooth rigid-analytic space over $\Spa(K, \O_K)$ with a given admissible quasi-compact formal model $\X$. Then there is a finite Galois extension $K \subset K'$, a finite extension $K' \subset K''$, a finite number of morphisms of admissible formal $\O_{K'}$-schemes (resp. $\O_{K''}$-schemes ) $g_i\colon \X'_i \to \X_{\O_{K'}}$ and $h_i\colon \X''_{i} \to \X'_{i, \O_{K''}}$, such that
\begin{itemize}\itemsep0.5em
\item Each $\X'_i$ admits an action of a finite group $G_i$ such that $g_i\colon \X'_i \to \X_{\O_{K'}}$ is $G_i$-invariant for each $i$.
\item The morphism $g\colon \X'\coloneqq  \sqcup_i \X'_i \to \X_{\O_{K'}}$ is a rig-\'etale covering (see Definition~\ref{rig-etale-covering}). 
\item On the generic fiber, each $\X'_i$ becomes a $G_i$-torsor over its (quasi-compact) open image in the adic generic fiber $X_{K'}=\X_{K'}$.
\item Each $\X'_i$ is formally quasi-projective over $\O_{K'}$ (see Definition~\ref{quasiproj-formal}) and has a structure of a successive formal semi-stable rig-smooth curve fibration (see Definition~\ref{def-ss-formal}).
\item Each $h_i\colon \X''_i \to \X'_{i, \O_{K''}}$ is a rig-isomorphism, and $\X''_i$ is rig-smooth, polystable formal $\O_{K''}$-scheme (see Definition~\ref{poly-def-formal}).
\item If $\O_K$ is discretely valued, one can choose $\X_i''$ to be strictly semistable (see Definition~\ref{semi-def-formal}).
\end{itemize}
\end{thm}

Let us now explain the main advantage of our version of the local uniformization theorem. One of the disadvantages of \cite[Theorem 3.3.1]{T3} is that there the admissible formal $\O_{K'}$-scheme $\X'$ does not give a model of the rigid-analytic space $X_{K'}$. Instead, what one gets is only that the morphism $\X'_{K'} \to X_{K'}$ is an \'etale covering. While it is good enough to study questions that are local on $X$, it is usually not strong enough for question that are only Zariski local on $\X$. Theorem~\ref{example-2} is better suited for these type of questions. In order to see this, we need to exploit the group action achieved in Theorem~\ref{example-2}. Namely, it allows us to define quotients $\X_i\coloneqq \X'_i/G_i$ that come with natural maps $\varphi_i\colon \X_i \to \X_{\O_{K'}}$. Then the assumptions on $\X_i$ in Theorem~\ref{example-2} imply that the $\X_i$ are admissible formal $\O_{K'}$-models of an open covering of $X_{K'}$. It turns out that the maps $\varphi_i$ can be controlled in a reasonable way {\it integrally}, this plays an important role in  our paper \cite{Zav4} where we give a proof of Poincar\'e Duality for \'etale $\mathbf{F}_p$ cohomology groups on a smooth and proper $p$-adic rigid space.

\begin{thm}\label{example-3} (Theorem~\ref{thm:final}) Let $\X$ be an admissible, quasi-compact and quasi-separated formal $\O_K$-scheme with the smooth generic fiber $\X_K$. Then there is a finite Galois extension $K\subset K'$, and a finite extension $K' \subset K''$, a finite set $(\X_i, \varphi_i)_{i\in I}$ of quasi-compact, quasi-separeted admissible formal $\O_{K'}$-schemes with morphisms $\varphi_i\colon \X_i \to \X_{\O_{K'}}$ such that
\begin{itemize}\itemsep0.5em
\item The set $(\X_i, \varphi_i)$ can be obtained from $\X_{\O_{K'}}$ as a ``composition of open Zariski coverings and rig-isomorphisms'' (see Definition~\ref{comp-formal}).
\item Each $\X_i$ is a geometric quotient (see Remark~\ref{rmk:about-quotient}) of an admissible formal $\O_{K'}$-scheme $\X'_i$ by an action of a finite group $G_i$ such that the map $p_{i,K'}\colon \X'_{i,K'} \to \X_{i,K'}$ is a $G_i$-torsor.
\item Each $\X'_i$ has a structure of a formal semi-stable, rig-smooth successive curve fibration over $\Spf \O_{K'}$.
\item Each $\X'_{i, \O_{K''}}$ admits a rig-isomorphism $h_i\colon \X''_i \to \X'_{i, \O_{K''}}$ with a rig-smooth, polystable formal $\O_{K''}$-scheme $\X''_i$.
\item If $\O_K$ is discretely valued, one can choose $\X_i''$ to be strictly semi-stable.
\end{itemize}
\end{thm}

\begin{rmk}\label{rmk:about-quotient}
The formalism of (geometric) quotients in the set-up of admissible formal schemes and (strongly noetherian) adic spaces is developed from scratch in \cite{Z2}. We briefly give the definition of geometric quotient in Section~\ref{section-final} (see Definition~\ref{geom-quot}) and refer to \cite{Z2} for a detailed discussion of the subject. 
\end{rmk}

Let us say a few words how one can use Theorem~\ref{example-3}, in particular clarify its role in \cite{Zav4}. Suppose we want to show some statement on a(ny) admissible formal $\O_C$-model $\X$ of a smooth, quasi-compact, quasi-separated rigid space $X$ over an algebraically closed non-archimedean field $C$. Then Theorem~\ref{example-3} implies that it suffices to show that this statement is Zariski-local on $\X$, descends along rig-isomorphisms and geometric quotients by an action of a finite group $G$ acting freely on generic fiber, and prove the claim for {\it polystable} $\X$. This strategy is used in \cite{Zav4} to show that particularly defined Faltings' trace map defines an almost (in the technical) perfect pairing on any admissible formal $\O_C$-model of a smooth $X$ by reducing to the case of a polystable model where one can verify it by hands. This, together with the construction of the Faltings' trace map, are the two key parts in the proof of Poincar\'e Duality. \smallskip 

One can think of Theorem~\ref{example-3} as a useful local substitute for a much stronger polystable reduction conjecture:

\begin{conj}\cite[Rigid Version of Conjecture 1.2.6]{AdipLiuPakTemkin}\label{conj} Let $\O$ be a complete rank-$1$ valuation ring with the algebraically closed fraction field $K$, and let $\X$ be a flat, topologically finitely presented formal $\O$-scheme whose adic generic fiber $\X_K$ is smooth. Then there is a rig-isomorphism $\X' \to \X$ such that $\X'$ is polystable. 
\end{conj}
 
\begin{rmk} The conjecture is formulated in the most optimistic way. For instance, it does not even assume that $X$ is separated. This generality is probably motivated by Temkin's Stable Modification Theorem~\cite[Theorem 2.3.3]{T2}, where also no separatedness assumptions are required. 
\end{rmk}


\subsection{Ideas of the main proofs} We now explain the plan of our proof that we follow in the paper. Basically, we follow the ideas of \cite{T1}, \cite{T2}, and \cite{T3}, but unfortunately, we do not know how to use the results of those papers directly to achieve our form of the local uniformization result. Instead we slightly modify all the mains proofs in \cite{T1}, \cite{T2} and \cite{T3}. At the end we use the main result of \cite{AdipLiuPakTemkin} as an important input.  Namely, we use the following strategy:

\begin{itemize}\itemsep0.5em

\item We prove a (slightly) refined version of a stable modification theorem for relative curves over a finite rank valuation rings. The main input here is \cite[Proposition 4.5.1]{T1}.

\item We generalize this refined version of a stable modification theorem for relative curves over semi-valuation rings of finite rank. Here the main input is \cite[Section 2]{T2} and the previous bullet point.

\item We finally prove Theorem~\ref{example-1}. We follow the ideas of \cite[Theorem 2.3.3, Step 2]{T2} but we use the result from the previous bullet point as our input over a semi-valuation base.

\item We prove a schematic version of our local altered uniformization result. This is the most difficult step, we use the ideas from the proof of \cite[Corollary 2.4.2]{T3}. Namely, we present our $S$-scheme $X$ (after some $U$-admissible blow up) as a successive curve fibration over a base. Then we apply our version of the stable modification theorem on each level of this successive curve fibration, and finally we use \cite[Theorem 5.2.16]{AdipLiuPakTemkin} to find a polystable model. In the case of a discretely valued $\O_K$, we use results from \cite{TorEmb} instead. \smallskip

The key points here are flattening techniques from \cite{RG}, our formulation of relative stable modification theorem for curves (which we establish in the previous bullet point), and \cite[Theorem 5.2.16]{AdipLiuPakTemkin} (resp. \cite{TorEmb} in discretely valued case).

\item We reduce Theorem~\ref{example-2} to the schematic local altered uniformization theorem by means of \cite{Elkik} and \cite[Proposition 3.3.2]{T0}.

\item To prove Theorem~\ref{example-3}, we study the construction of the local altered uniformization in more detail. The main input here is the observation that all the constructions above are quite explicit, so we can understand the quotients $\X'_i/G_i$ ``by hands''. We also use the results on geometric quotients from \cite{Z2} to ensure certain properties of these quotients. 

\end{itemize}

The last thing that we want to emphasize is that the main reason behind the fact that we can rather `easily' implement group actions in Temkin's local uniformization result is the uniqueness of stable modification over normal bases \cite[Theorem 1.2]{T1}. This makes certain constructions automatically $G$-equivariant and simplifies the exposition significantly.

\subsection*{Acknowledgements}
The author warmly thanks M. Temkin for enlightening discussions. This paper clearly owes a huge intellectual debt to \cite{T1}, \cite{T2} and \cite{T3}. We are also very grateful to B.\,Bhatt, B.\,Conrad,  A.\,Landesman, S.\,Li, D.\,B.\,Lim, S.\,Petrov, and R.\,Taylor for many fruitful discussions. We are grateful to L.\,Illusie for his help with Appendix~\ref{log}. We express additional gratitude to B.\,Conrad, N.\,Kurnosov, and D.\,B.\,Lim for reading the first draft of this paper and making lots of suggestions on how to improve the exposition of this paper. We would also like to thank Z.\,Zhang for helpful comments on the previous version of the paper.  
\section{Terminology}

We will freely use a lot of definitions related to the notion of a ``(relative) curve''. It seems that different people denote slightly different things by this name. For example, there is no consensus whether a (relative) curve should have geometrically connected fibers. Since it is actually very important for our purposes that we allow curve to be disconnected (and also be non-proper) we decide to explain here some relevant notations. 

Another thing that we recall here is a number of definitions from \cite{T3} that also seem to vary from author to author. So, in order to get rid of any ambiguities in the definitions we decided to spell them out in this section.

\begin{defn}\label{def-curve} We say that a morphism of scheme $f\colon  X \to S$ is a {\it relative curve} (or {\it curve fibration}), if it is finitely presented\footnote{We recall that a finitely presented morphism is required to be quasi-separated.} flat morphism and all non-empty fibers $X_s$ are of pure relative dimension $1$.
\end{defn}

\begin{rmk} As mentioned above, we do not impose any properness assumption on $f$. Moreover, we do not require fibers of $f$ to be (geometrically) connected. This is an important technical point that will help us a lot later. Also, we do not even require $f$ to be separated, but this generality will be of no real interest to us in this paper.
\end{rmk}

\begin{defn}\label{def-ss-field} We say that a morphism $f\colon  X \to \Spec k$ is a {\it semi-stable curve}, if it is a curve over $\Spec k$ (in the sense of Definition~\ref{def-curve}) and for each closed point $x\in X_{\bar k}$ the completed local ring $\widehat{\O_{X_{\bar k}, x}}$ is isomorphic to $\bar{k} \llbracket T \rrbracket$ or $\bar{k}\llbracket
 U,V \rrbracket/(UV)$.
\end{defn}

\begin{defn}\label{def-ss} We say that a morphism $f\colon  X \to S$ is a {\it semi-stable relative curve} (or {\it semi-stable curve fibration}), if $f$ is a relative curve (in the sense of Definition~\ref{def-curve}) and for all $s\in S$ the fiber $f_s\colon  X_s \to \Spec k(s)$ is a semi-stable curve (in the sense of Definition~\ref{def-ss-field}).
\end{defn}

\begin{lemma}\label{ss-alternative} Let $f\colon X\to S$ be a semi-stable relative curve, then for any point $x\in X$ either $f$ is smooth at $x$, or there is an open affine neighborhood $\Spec B \subset X$ containing $x$ and an open affine neighborhood $\Spec A \subset S$ containing $y=f(x)$, such that there exists a diagram of pointed schemes
\[
\begin{tikzcd}
& \left(\Spec C, s\right)  \arrow[swap]{ld}{g} \arrow{rd}{h}\\
\left(\Spec B,x\right) &  & \left(\Spec \frac{A[U,V]}{\left(UV-a\right)}, \left\{y,0,0\right\}\right),
\end{tikzcd}
\]
where $g$ and $h$ are \'etale and $a$ lies in the ideal of $y$.  
\end{lemma}
\begin{proof}
The proof is standard, we only give a reference to \cite[Ch. III, Proposition 2.7]{FK}. 
\end{proof}
\medskip 

For the rest of the definitions in this section, we fix a quasi-compact quasi-separated scheme $S$ and a subset $|U|\subset |S|$ which is schematically dense, quasi-compact, closed under generalization, and {\it admits a schematic structure}. What we mean by this is that there is a quasi-compact scheme $U$ and a schematically dominant monomorphism $i\colon U \to S$ that can be topologically identified with the (topological) embedding $|U|\subset |S|$. 
\smallskip

Two main types of examples of such $U$ are quasi-compact schematically dense open subschemes of $S$ and the generic points of an integral scheme $S$. 

\begin{defn}\label{U-adm} We say that an $S$-scheme $X$ is {\it $U$-admissible} if a scheme $X_U$ is schematically dense in $X$.
\end{defn}

\begin{rmk} We note that that any finitely presented, flat $S$-scheme is $U$-admissible. 
\end{rmk}

\begin{defn}\label{U-mod} We say that a map of $S$-schemes $f\colon X' \to X$ is a {\it $U$-modification}, if $X'$ is $U$-admissible, and $f$ is a proper morphism such that the base change $f_U\colon  X'_U \to X_U$ is an isomorphism. 
\end{defn}

\begin{rmk} For technical reasons, it will be more convenient not to assume that $X$ or $X'$ are finitely presented in Definition~\ref{U-mod} and Definition~\ref{U-adm}. However, the main examples of interest will be finitely presented. 
\end{rmk}

\begin{defn}\label{U-stable} Let $h\colon S'\to S$ be a map of quasi-compact, quasi-separated schemes. We say that a morphism $f\colon X' \to X$ of relative $S'$-curves is an {\it $S'$-stable $U$-modification} if it is a $h^{-1}(U)$-modification, $X'$ is a semi-stable relative curve over $S'$, and for all geometric points $\ov{s}'\in S'$ all contracted irreducible components $Z\subset X'_{\ov{s}'}$ that are isomorphic to $\P^1_{k(\ov{s}')}$ intersect the singular locus $X'_{\ov{s}'}$ in at least $3$ points.
\end{defn}

\begin{rmk} If $S=S'$, we will call $S$-stable $U$-modifications simply as stable $U$-modifications.
\end{rmk}


\begin{defn}\label{U-etale} We say that a map of two $S$-schemes $f\colon X' \to X$ is a {\it $U$-\'etale morphism} (resp. {\it $U$-smooth morphism}), if it is a separated, finitely presented morphism such that the base change $f_U\colon  X'_U \to X_U$ is \'etale (resp. smooth) and $X'$ is $U$-admissible.
\end{defn}

\begin{defn}\label{U-etale cover} We say that a map of two $S$-schemes $f\colon X' \to X$ is a {\it $U$-\'etale covering}, if it is $U$-\'etale and for any valuation ring $R$ any morphism $\phi\colon  \Spec R \to X$ taking the generic point to $X_U$ lifts to a morphism $\phi'\colon  \Spec R' \to X'$ where $R'$ is some valuation ring such that $\Frac(R')/\Frac(R)$ is finite.
\end{defn}

The two main examples of $U$-\'etale coverings are $U$-modifications and faithfully flat $U$-\'etale morphisms. 

\begin{lemma}\label{example-U-covering} Let $S, U$ be as above, and let $h\colon S' \to S$ be either a $U$-modification or a faithfully flat $U$-\'etale morphism. Then $h$ is a $U$-\'etale covering (in the sense of Definition~\ref{U-etale cover}).
\end{lemma}
\begin{proof}
Let us firstly consider the case of a $U$-modification morphism $h\colon S' \to S$. Then it is clearly $U$-\'etale and $S'$ is an $U$-admissible scheme by the very definition of an $U$-modification. The condition about lifting maps from valuation rings is trivially satisfied by the valuative criterion of properness. \medskip

Now let us deal with the case of faithfully flat $U$-\'etale morphisms. Again, the only condition we need to check is the condition about lifting maps from valuation rings. Suppose that we are given a morphism $\phi\colon \Spec R \to S$ such that $R$ is a valuation ring and $\phi$ sends the generic point of $\Spec R$ into $U$. We want to lift it to a morphism $\phi'\colon \Spec R' \to S'$ in a commutative diagram 
\[
\begin{tikzcd}
\Spec R' \arrow{d}  \arrow{r}{\phi'} & S' \arrow{d}{h}\\
\Spec R \arrow{r}{\phi} & S.
\end{tikzcd}
\]
\noindent such that $R'$ is also a valuation ring and the extension of fraction fields $\Frac(R')/\Frac(R)$ is finite. Since $\phi$ sends the generic point of $\Spec R$ into $U$, we conclude that $U\times_S \Spec R$ is schematically dense in $\Spec R$ as $R$ is reduced. Note that surjectivity and flatness are preserved by arbitrary base change. Therefore, it is sufficient to assume that $S=\Spec R$ for the purposes of the proof. \medskip

Consider the closed point $s\in \Spec R$ and choose an arbitrary point $s'\in S'$ in a fiber over $s$. We find an open affine subscheme $V'\subset S'$ that contains the point $s'$. The Going-Down Theorem (\cite[Theorem 9.5]{M1}) shows that there is a point $\eta' \in V'$ that maps to the generic point $\eta\in \Spec R=S$. Moreover, this point can be chosen to be a generic point of $V'$. 
\smallskip

We claim that $\O_{S',\eta'}$ is a field. The only thing we need to check is that $S'$ is reduced at $\eta'$, but this is automatic because $\Spec R$ is reduced and $h\colon S' \to \Spec R$ is \'etale over $\eta$. Let us denote a prime ideal corresponding to $s'$ by $\p$. Then $U$-admissibility of $S'$ implies that $\O_{S',s'}$ is a subring of a field $\O_{S', \eta'}$, and \cite[Theorem 10.2]{M1} guarantees that there is a valuation ring $R'\subset \O_{S',\eta'}$ that dominates $\O_{S', s'}$. It defines a morphism $\phi'\colon \Spec R' \to S'$ as a composition \[\Spec R' \to \Spec \O_{S',s'} \to S'.\] Then it is clear that the diagram
\[
\begin{tikzcd}
\Spec R' \arrow{d}  \arrow{r}{\phi'} & S' \arrow{d}{f}\\
\Spec R \arrow{r}{\phi} & S
\end{tikzcd}
\]
is commutative. The finiteness $\Frac(R')/\Frac(R)=\O_{S', \eta'}/\Frac(R)$ follows from quasi-finiteness (over $U$) of the morphism $f\colon S' \to \Spec R$.
\end{proof}

\begin{rmk}\label{rmk:K=U} In the case $S=\Spec \O_K$ for some valuation ring $K$ and $U=\Spec K$, we will usually say $K$-admissible (resp. $K$-modification, resp. $K$-\'etale, resp. $K$-smooth) instead of $U$-admissible (resp. $U$-modification, resp. $U$-\'etale, resp. $U$-smooth).
\end{rmk}

We recall another two technical definitions that will play an important role in this paper:

\begin{defn}\label{defn:strictly-semistable} Let $\O_K$ be a discrete valuation ring. A finitely presented $\O_K$-scheme $X$ is called {\it strictly semi-stable} if Zariski-locally it admits an \'etale morphism 
\[
U \to \Spec \frac{\O_K[t_0, \dots, t_l]}{(t_{0}\cdots t_{m}-\pi)}
\]
for some integers $m\leq l$, and a uniformizer $\pi \in \m_K\setminus \m_K^2$. 
\end{defn}

\begin{defn}\label{snc} Let $X$ be a regular scheme. A {\it strict normal crossing divisor} (or an {\it snc divisor})
 in $X$ is a closed subscheme $Y$ such that the intersection of every set of irreducible components of $Y$ is also regular.
\end{defn}

\begin{rmk} If the residue field of $\O_K$ is perfect, a flat finitely presented $\O_K$-scheme $X$ is strictly semi-stable if it the generic fiber $X_K$ is $K$-smooth, $X$ is regular, and the special fiber $X_k \subset X$ is an snc divisor in $X$. See \cite[ page 59, 2.16]{DJ1} for more detail. 
\end{rmk}

\begin{defn}\label{polystable}\cite{AdipLiuPakTemkin} Let $\O_K$ be a rank-$1$ valuation ring. A finitely presented $\O_K$-scheme $X$ is called {\it ALPT-polystable} if \'etale locally it admits an \'etale morphism 
\[
U \to \Spec \frac{\O_K[u_1^{\pm 1}, \dots, u_m^{\pm 1}, t_{0,0},\dots, t_{0, n_0}, \dots, t_{l,0}, \dots, t_{l, n_l}]}{(t_{1,0}\cdots t_{1, n_1}-\pi_1, \dots, t_{l,0}\cdots t_{l, n_l}-\pi_l)}
\]
for some $\pi_i\in \m_K$. The $\O_K$-scheme 
\[
\Spec \frac{\O_K[u_1^{\pm 1}, \dots, u_m^{\pm 1}, t_{0,0},\dots, t_{0, n_0}, \dots, t_{l,0}, \dots, t_{l, n_l}]}{(t_{1,0}\cdots t_{1, n_1}-\pi_1, \dots, t_{l,0}\cdots t_{l, n_l}-\pi_l)}
\]
is called a {\it model polystable} $\O_K$-scheme. 
\end{defn}

\begin{rmk} We note that this definition allows $\pi_i=0$. In particular, the generic fiber of an ALPT-polystable $\O_K$-scheme is not necessarily smooth. 
\end{rmk}

\begin{rmk} This definition has an advantage over the notion of semi-stability is that any product of ALPT-polystable $\O_K$-schemes is ALPT-polystable. The analogous statement is false for semi-stable schemes.
\end{rmk}

\begin{rmk}\label{ALPT-polystable-flat} Any ALPT-polystable $\O_K$-scheme is flat over $\O_K$. Indeed, it suffices to show that a model polystable $\O_K$-scheme is flat, but its function algebra clearly does not have $\O_K$-torsion hence $\O_K$-flat.
\end{rmk}

We will slightly modify the notion of poly-stability to make it more useful for our purposes later. We start with a few lemmas.

\begin{lemma}\label{origin} Let $\O_K$ be a rank-$1$ valuation ring with algebraically closed residue field $k$, let $X$ be an ALPT-polystable $\O_K$-scheme, and let $x\in X$ be a closed point in the closed fiber. Then there is an \'etale neighborhood $f\colon (U, u) \to (X, x)$ of $x$ such that $f(u)=x$ and $U$ admits an \'etale morphism 
\[
g\colon U \to \Spec \frac{\O_K[u_1^{\pm 1}, \dots, u_m^{\pm 1}, t_{0,0},\dots, t_{0, n_0}, \dots, t_{l,0}, \dots, t_{l, n_l}]}{(t_{1,0}\cdots t_{1, n_1}-\pi_1, \dots, t_{l,0}\cdots t_{l, n_l}-\pi_l)}
\] 
such that $t_{0,j}(g(u))=1$ for all $j\leq n_0$, and $t_{i, j}(g(u))=0$ for all $1\leq i\leq l$, $j\leq n_l$.
\end{lemma}
\begin{proof}
{\it Step 1. Reduce to the model case}: As $X$ is ALPT-polystable, we can find an \'etale neighborhood $U$ of $x$ with an \'etale morphism $g$ to some model polystable $\O_K$-scheme $Z$. Pick $u\in U$ to be any point over $x$, and then it suffices to prove the claim for the pair $(Z, z)\coloneqq (Z, g(u))$, i.e. we can assume that \[Z=\Spec \frac{\O_K[u_1^{\pm 1}, \dots, u_m^{\pm 1}, t_{0,0},\dots, t_{0, n_0}, \dots, t_{l,0}, \dots, t_{l, n_l}]}{(t_{1,0}\cdots t_{1, n_1}-\pi_1, \dots, t_{l,0}\cdots t_{l, n_l}-\pi_l)}.\] 

{\it Step 2. Reduce to the semi-stable model example}: As the product of polystable models is polystable, and the product of \'etale morphisms is \'etale, it suffices to treat separately the cases of $Z=\Spec \O_K[t]$ and $Z=\Spec \O_K[t_0, \dots, t_n]/(t_0\cdots t_n -\pi)$ for some element $\pi\in \m_K$. \medskip

{\it Step 3. Finish the proof in the semi-stable}: The semi-stable case was already treated in \cite[Lemma 2.3.5]{T3}, but the proof has a mathematical typo carried all over the proof. So we have decided to write a complete proof here. \smallskip

The case $Z=\Spec \O_K[t]$ is easy. Consider $z\in Z$ as a closed point of the special fiber $Z_s$ that is a variety over an algebraically closed field $k$. Thus it makes sense to evaluate $t$ at $z$ and get an element in $k$. Suppose that $t(z)=\ov{c}\in k$, lift it to some element $c\in \O_K$. Then the desired \'etale morphism is just the shift by $c-1$ on $Z=\Spec \O_K[t]$. \smallskip 

Now we deal with the case $Z=\Spec \O_K[t_0, \dots, t_n]/(t_0\cdots t_n -\pi)$ for some $\pi \in \m_K$. Then we rename $t_0, \dots, t_n$ so that $t_0(z), \dots, t_m(z)=0\in k$ but $t_{m+1}(z), \dots, t_n(z) \neq 0$. We consider $t_i(z)=:\ov{c_i}\in k$ and let $c_i$ be any lift of $\ov{c_i}$ to $\O_K$. Finally, we define $t'_i\coloneqq t_i$ for $0\leq i\leq m-1$, $t'_m\coloneqq t_m\cdots t_n$, and $t'_{i}=t_i-c_i+1$ for $i\geq m+1$. This defines the natural morphism
\[
\frac{\O_K[t'_0, \dots, t'_n]}{(t'_0\cdots t'_m -\pi)} \to \frac{\O_K[t_0, \dots, t_n]}{(t_0\cdots t_n -\pi)}
\]
On the scheme level, it defines the morphism
\[
g\colon Z \to \Spec \O_K[t'_0, \dots, t'_n]/(t'_0\cdots t'_m -\pi)
\]
that is easily seen to be \'etale at $z\in Z$. The target is ALPT-polystable, and $t'_i(g(z))=0\in k$ for $i\leq m$, $t'_i(g(z))=1\in k$ for $i>m$.
\end{proof}

\begin{cor}\label{k-smooth-poly} Let $\O_K$ be a rank-$1$ valuation ring with algebraically closed residue field $k$ and fraction field $K$, let $X$ be a $K$-smooth ALPT-polystable $\O_K$-scheme, and let $x\in X$ be any point. Then there is an \'etale neighborhood $f\colon (U, u) \to (X, x)$ of $x$ such that $f(u)=x$ and $U$ admits an \'etale morphism 
\[
g\colon U \to \Spec \frac{\O_K[u_1^{\pm 1}, \dots, u_m^{\pm 1}, t_{1,0}, \dots, t_{1, n_1}, \dots, t_{l,0}, \dots, t_{l, n_l}]}{(t_{1,0}\cdots t_{1, n_1}-\pi_1, \dots, t_{l,0}\cdots t_{l, n_l}-\pi_l)}
\]
for some {\it non-zero} $\pi_i\in \m_K \setminus 0$.
\end{cor}
\begin{proof}
If $x\in X$ lies in the generic fiber, the claim is easy. Indeed, $X \to \Spec \O_K$ is smooth at $x$, so there is a neighborhood $x\in U\subset X$ with an \'etale morphism $U \to \Spec \O_K[u_0, \dots, u_n]$. Then we can shift the image similarly to what we have done in Step~$2$ of Lemma~\ref{origin} to assume that we get an \'etale map of the form $U \to \Spec \O_C[u_0^{\pm 1}, \dots, u_n^{\pm 1}]$. \medskip

The case of $x\in X$ lying in the special fiber essentially follows from Lemma~\ref{origin}. Namely, it suffices to prove the claim for closed points $x$ since \'etale morphisms are open. In this case, we choose an \'etale neighborhood $U$ of $x$ with an \'etale morphism 
\[
g\colon U \to \Spec \frac{\O_K[u_1^{\pm 1}, \dots, u_m^{\pm 1}, t_{0,0},\dots, t_{0, n_0}, \dots, t_{l,0}, \dots, t_{l, n_l}]}{(t_{1,0}\cdots t_{1, n_1}-\pi_1, \dots, t_{l,0}\cdots t_{l, n_l}-\pi_l)}
\]
such that $t_{0,j}(g(u))=1$ for all $j\leq n_0$, and $t_{i, j}(g(u))=0$ for all $1\leq i\leq l$, $j\leq n_l$. The first condition ensures that the map factors as 
\[
g\colon U \to \Spec \frac{\O_K[u_1^{\pm 1}, \dots, u_m^{\pm 1}, t_{0,0}^{\pm 1},\dots, t_{0, n_0}^{\pm 1}, \dots, t_{l,0}, \dots, t_{l, n_l}]}{(t_{1,0}\cdots t_{1, n_1}-\pi_1, \dots, t_{l,0}\cdots t_{l, n_l}-\pi_l)} \ .
\]
Now note that $U$ is $C$-smooth as it is \'etale over $X$. Since $g$ is \'etale, we conclude that the open subscheme $\rm{Im}(g)$ is $K$-smooth. Now we recall that $t_{i, j}(g(u))=0$ for all $1\leq i\leq l$, $j\leq n_l$, so the composition of $g$ with the projection onto the last factors defines a map
\[
h\colon U \to M\coloneqq \Spec \frac{\O_K[t_{1,0},\dots, t_{1, n_1}, \dots, t_{l,0}, \dots, t_{l, n_l}]}{(t_{1,0}\cdots t_{1, n_1}-\pi_1, \dots, t_{l,0}\cdots t_{l, n_l}-\pi_l)}
\]
that is smooth, and $h(U)$ contains the origin $O$ in the special fiber of $M$. Again, since $h$ is smooth, we conclude that $M$ is $K$-smooth at the image of $h$. In particular, $\O_{M, O}\otimes_{\O_K} K$ is $K$-(ind)smooth as $h(U)$ contains $O$. Now Remark~\ref{ALPT-polystable-flat} implies that each factor \[M_i\coloneqq \Spec \frac{\O_K[t_{i,0},\dots, t_{i, n_i}]}{(t_{i,0}\cdots t_{i, n_i}-\pi_i)}\] is $\O_K$-flat. Thus, \cite[\href{https://stacks.math.columbia.edu/tag/02VL}{Tag 02VL}]{stacks-project} implies that each factor the rings $\O_{M_i, O_i}\otimes_{\O_K} K$ is $K$-(ind)smooth, where $O_i$ is the origin in the special fiber of $M_i$. Now we use the Jacobian criterion to ensure that 
the $K$-(ind)smoothness of $\O_{M_i, O_i}\otimes_{\O_K} K$ implies that $\pi_i\neq 0$.
\end{proof}

\begin{defn}\label{poly-def} Let $\O_K$ be a rank-$1$ valuation ring. A flat, finitely presented, $K$-smooth $\O_K$-scheme $X$ is called {\it polystable} if \'etale locally it admits an \'etale morphism 
\[
U \to \Spec  \frac{\O_K[t_{1,0},\dots, t_{1, n_1}, \dots, t_{l,0}, \dots, t_{l, n_l}]}{(t_{1,0}\cdots t_{1, n_1}-\pi_1, \dots, t_{l,0}\cdots t_{l, n_l}-\pi_l)}
\]
for some $\pi_i\in \O_K \setminus \{0\}$. The $\O_K$-scheme 
\[
\Spec \frac{\O_K[t_{1,0},\dots, t_{1, n_1}, \dots, t_{l,0}, \dots, t_{l, n_l}]}{(t_{1,0}\cdots t_{1, n_1}-\pi_1, \dots, t_{l,0}\cdots t_{l, n_l}-\pi_l)}
\]
is called a {\it model polystable} $\O_K$-scheme. 
\end{defn}

\begin{rmk} Corollary~\ref{k-smooth-poly} ensures that any $K$-smooth ALPT-polystable $\O_K$-scheme is polystable if the residue field $k$ of $\O_K$ is algebraically closed. Indeed, it guarantees that we can choose non-zero $\pi_i\in \m_K$ for all singular factors $\Spec \frac{\O_K[t_{i,0},\dots, t_{i, n_i}]}{(t_{i,0}\cdots t_{i, n_i}-\pi_i)}$, and the smooth factors $\Spec \O_K[u^{\pm 1}]$ can be rewritten as $\Spec \O_K[t_0, t_1]/(t_0t_1-1)$. 

One can show that these two notions of polystability are actually the same provided that $X$ is $K$-smooth and $k$ is algebraically closed. More generally, these two notions are the same if $k$ is perfect. One can show this by passing to the strict henselization $\O_K^{\mathrm{sh}}$ and using \cite[\href{https://stacks.math.columbia.edu/tag/0ASJ}{Tag 0ASJ}]{stacks-project}. We will never use any of these results, so we do not give a proof. We do not know if these two notions are the same for a general $\O_K$. 
\end{rmk}

\section{Stable Modification Theorem for Curves Over Finite Rank Valuation Rings}

In this section we prove a version of Temkin's Stable Modification Theorem over a finite rank valuation ring $R$. We fix some notations for this section:  $S\coloneqq \Spec R$, and $U\coloneqq \eta$ be the generic point of $\Spec R$. There are different ways to define a rank of a valuation; we recall both of them below:

\begin{defn}\label{rank} We say that a valuation ring $R$ is of {\it finite rank}, if $\Spec R$ is a finite set. In this case, the rank $r$ is defined as $ |\Spec R| -1$ (number of points in $\Spec R$ minus $1$). It coincides with a rank (or height) of its valuation group $\Gamma$ by \cite[Ch. VI, \textsection 4, n. 4, Proposition 5]{Bou}. We denote rank of $R$ by $\rk R$ or $\rk \Gamma$. 
\end{defn} 
\begin{defn}\label{rat-rank} We say that a valuation ring $R$ with a value group $\Gamma$ has {\it rational rank} $r$, if $\dim_{\Q} \Gamma \otimes_{\Z} \Q=r$ (it is infinite if dimension of $\Gamma \otimes_{\Z} \Q$ is infinite). We denote it by $\rk_{\Q} \Gamma$ or $\rk_{\Q} R$. 
\end{defn}
\smallskip

Recall that for any valuation ring $R$ we have inequality $\rk_{\Q} R \geq \rk R$ by \cite[Ch. 6, \textsection 10, n.2, Proposition 3]{Bou}. Definition~\ref{rat-rank} is useful to verify that certain valuation rings have finite rank. \medskip

Now we recall the statement of Temkin's Stable Modification Theorem:  

\begin{thm}\cite[Proposition 4.5.1]{T1}\label{st-mod-val-T}  Let $R$ be a finite rank valuation ring with a separably closed field of fractions $K$, and let $X \to \Spec R$ be a relative curve such that the generic fibre is semi-stable. There is a stable $U$-modification $f\colon X' \to X$ (in particular, $X'$ is a semi-stable curve over $\Spec R$).
\end{thm}


We want to show a (slightly) refined version of this theorem for an arbitrary finite rank valuation ring. Firstly, we recall some facts about valuation rings. 

\begin{lemma}\label{ext-val-1} Let $R$ be a valuation ring with a field of fractions $K$, let $K'$ be an algebraic extension of $K$ and denote by $R'$ the normalization of $R$ in $K'$. Then the natural morphism $R' \otimes_R K \to K'$ is an isomorphism. Also, if $R_1$ is any valuation ring in $K'$ that dominates $R$, then the natural map $R_1\otimes_R K \to K'$ is an isomorphism.
\end{lemma}
\begin{proof}
First if all, note that $R' \otimes_R K$ is a localization of the domain $R'$. Therefore, the natural map $R' \otimes_R K \to K'$ is injective, in particular it is a domain. In order to show that it is an isomorphism, it suffices to show that $R'\otimes_R K$ is a field. But $R'\otimes_R K$ is a domain integral over $K$, thus it must be a field.


As for the second claim, we note that $R_1$ is integrally closed in its fraction field $K'$, so it contains $R'$. This easily implies that $R_1 \otimes_R K$ surjects onto $K'$. Injectivity follows easily from $R$-flatness of $R_1$.
\end{proof}
\smallskip

\begin{lemma}\label{val-ext-2} Let $R$ be a valuation ring with a field of fractions $K$, and let $K'$ be a finite Galois extension of $K$ with Galois group $G$. Then the group $G$ acts transitively on the set of maximal ideals in the normalization $R'$ of $R$ (in $K'$), and all these maximal ideals contract to the unique maximal ideal in $R$.
\end{lemma}
\begin{proof}
Note that the question is well-posed because $G$ acts on a normalization of $R$, and therefore acts on its set of maximal ideals. Since the extension $R \subset R'$ is integral, we see that any maximal ideal in $R'$ contracts to the maximal ideal in $R$ by \cite[Section 9, Lemma 2]{M1}, and \cite[Theorem 9.3]{M1} proves that the Galois group $G$ acts transitively on these ideals.
\end{proof}
\smallskip

\begin{lemma}\label{open-1} Let $R$ be a valuation ring of finite rank with field of fractions $K$. Then the subset $\Spec K \subset \Spec R$ is open.
\end{lemma}
\begin{proof}
We note that $K=S^{-1}R$ is a localization of a ring $R$ at the multiplicative set $S\coloneqq  R\setminus \{0\}$. We have an equality $K=\colim_{x\in S} R_x$, so 
\[
\Spec K=\lim_{x\in S} \Spec R_x \cong \bigcap_{x\in S} \Spec R_x \subset \Spec R.
\]
Since $\Spec R$ is has a finite underlying topological space, there are only finitely many options for $\Spec R_x \subset \Spec R$. Therefore, there exists some $x\in S$ such that $\Spec K = \Spec R_x$, so $\Spec K$ is open in $\Spec R$.
\end{proof}

\begin{thm}\label{stable-modification-valuation} Let $S= \Spec R$ be a spectrum of a valuation ring $R$ with fraction field $K$, and let $f\colon X \to S$ be a relative curve such that the generic fibre is semi-stable. Then there is a finite Galois extension $K'/K$ with $R'$ the integral closure of $R$ in $K'$ such that $X_{S'}=X\times_S S'$ admits an $S'$-stable $U$-modification $g\colon X' \to X_{S'}$ with $S'= \Spec R'$. Moreover, an action of $G$ on $S'$ over $S$ lifts to an action on $X'$ over $X$.
\end{thm}
\begin{proof}
{\it Step 0}: We use \cite[Corollary 2.1.3]{T0} to write $R=\colim_i R_i$ as a filtered colimit of its valuation subrings of finite height. Then a standard approximation argument allows to assume that $R=R_i$ is of finite rank. \\

{\it Step 1}: Choose a separable closure $K\subset K^{\operatorname{sep}}$, then \cite[Theorem 10.2]{M1} guarantees the existence of a valuation ring $R'$ in $K^{\operatorname{sep}}$ that dominates $R$. Now by \cite[Ch. 6, \textsection 8, n.1, Corollary 1]{Bou} we get that $\rk(R')=\rk(R)$, so $R'$ is a valuation ring of finite rank. This means that we can apply Theorem~\ref{st-mod-val-T} to get an $S'$-stable $U$-modification $g\colon X' \to X_{R'}$ with $S'=\Spec R'$ (note that $X_{R'} \times_{S} U=X_{K'}$ by Lemma~\ref{ext-val-1}). \\

{\it Step 2}: We write $K^{\operatorname{sep}}= \colim_{i\in I} K_i$ as a direct colimit of finite Galois extensions $K_i$ of $K$. Note that each $R_i\coloneqq  K_i \cap R'$ is again a valuation ring, $R'=\colim_{i\in I} R_i$, and for each $i$ we have $R_i \otimes_R K \cong K_i$. We define $S_i\coloneqq \Spec R_i$. Then a standard approximation argument allows to spread out the $S'$-stable $U$-modification $X' \to X_{S'}$ to an $S_i$-stable $U$-modification $g_i\colon X'_i \to X_{R_i}$. However, we are not done yet since $R_i$ is generally not the integral closure of $R$ in $K_i$ (as this integral closure is usually just semi-local rather than local). \\


{\it Step 3}: Now we change notations and denote $K_i$ by $K'$, $R_i$ by $R_1$, and the stable $U$-modification $g_i\colon X'_i \to X_{R_i}$ by $g\colon X'_1 \to X_{R_1}$. Also we denote by $R'$ the integral closure of $R$ in $K'$. Now note that \cite[Ch. 6, \textsection 8, n.6, Proposition 6]{Bou} implies that $R_1$ is given by a localization of $R'$ at some maximal ideal $\m$, and Lemma~\ref{val-ext-2} says that the Galois group $G\coloneqq \Gal(L/K)$ acts transitively on the set of maximal ideals in $R'$. For all $\sigma \in G$ the pullback $\sigma^*(g)\colon \sigma^*(X'_1) \to X_{(R_1)_{\sigma(\m)}}$ is an $R_1$-stable $U$-modification, so there exists a stable modification of $X_{R'}$ after a localization at each maximal ideal of $R'$. Then we use the standard approximation techniques 
once again to spread out this $R_1$-stable $U$-modification to a $V_{\n}$-stable $U$-modification $g_{\n}\colon X'_{\n} \to X_{V_{\n}}$ over some open $V_{\n} \subset \Spec R'$ around each maximal ideal $\n \subset R'$. Since $\Spec R'$ is normal, an $R'$-stable $U$-modification of $X_{R'}$ is unique up to a unique isomorphism by \cite[Corollary 1.3]{T1} and the same holds over all opens of $\Spec R'$. This uniqueness assertion allows us to glue various $X'_{\n}$ to a unique $R'$-stable $U$-modification $g\colon X' \to X_{R'}$. \\ 


{\it Step 4:} The $R'$-stable $U$-modification $X' \to X_{R'}$ automatically admits a unique action of $G$ extending that an action on its $K'$-fiber by \cite[Corollary 1.6]{T1} and the morphism $g\colon  X' \to X_{R'}$ is $G$-equivariant.
\end{proof}

\section{Stable Modification Theorem for Curves Over Finite Rank Semi-Valuation Rings}

Let us recall a definition of a semi-valuation ring from \cite[\textsection 2]{T2}. 

\begin{defn} A {\it valuation} on a ring $A$ is commutative ordered group $\Gamma$ with a multiplicative map $|\ |\colon A \to \Gamma \cup \{0\}$ which satisfied the inequality
\[
|x+y| \leq \operatorname{max}(|x|, |y|)
\]
and sends $1$ to $1$.
\end{defn}

\begin{defn}\label{defn-semivaluation} A {\it semi-valuation ring} is a ring $\O$ with a valuation $|\ |\colon  \O \to \Gamma \cup \{0\}$ such that all zero-divisors of $\O$ lie in the prime kernel $\m=\ker(|\ |)$ and for any pair $g,h\in \O$ with $|g|\leq |h| \neq 0$ one has $h|g$ in $\O$. We say that $A\coloneqq \O_{\m}$ is the {\it ring of semi-fractions} of $\O$.
\end{defn}

In \cite[\textsection 2]{T2} Temkin gives a very useful description of all semi-valuation rings. Namely, they all come as a ``composition'' of a local ring $A$ and a valuation ring $R$. Let us explain what this means. Given a semi-valuation ring $\O$ we see that the natural map $\O \to \O_{\m}=\colon A$ is an injective morphism with a maximal ideal of $A$ equal to $\m A=\m$ (exercise!), and the image $R=\O/\m$ of $\O$ under the projection map $p\colon A \to A/\m A$ is a valuation ring with fraction field $A/\m A$. Indeed, the fraction field of $R$ is naturally identified with $\O_{\m}/\m\O_{\m}=A/\m A$ and the valuation $|\ |\colon  \O \to \Gamma \cup \{0\}$ induces a valuation on $A/\m A$ with a ring of valuation being equal to $R$. Moreover one can easily show that $\m A=\m$ using that the ideal $\m$ is $\O\setminus \m$-divisible. This implies that $\O=p^{-1}(R)$, and gives an example of a ``composition'' of a local ring $A$ and a valuation ring $R$: 

\begin{defn} The {\it composition} of a local ring $(A, \m_A)$ and a valuation ring $R\subset A/\m A$ (with fraction field $A/\m A$) is the ring $\O\coloneqq  p^{-1}(R)$, where $p\colon A \to A/\m A$ is the natural projection map. It is easy to see that any such composition $\O$ is a semi-valuation ring (in the sense of Definition~\ref{defn-semivaluation}), and any semi-valuation ring $\O$ is a composition of $A=\O_{\m}$ and $R=\O/\m \subset K=A/\m A$.
\end{defn}

Actually this structure result will be more useful for us than the original definition, and from now on by a ``semi-valuation ring'' we will mean a pair $(\O, \m)$ consisting of a ring $\O$ and a prime ideal $\m$ such that $\O$ injects into $A\coloneqq \O_{\m}$ and the image of $\O/\m$ in $A/\m A$ is a valuation ring with fraction field $A/\m A$. Given any semi-valuation ring $(\O, \m)$ we will always denote by $A$ its localization $\O_{\m}$, by $K$ the residue field of $A/\m A$ and by $R \subset K$ the valuation ring $\O/ \m$. Also, we reserve some notation for this section:  we will always denote $\Spec \O$ by $S$, $\Spec A$ by $U$ (so $U \to S$ is schematically dominant, monic and stable under generalization), $\Spec R$ by $T$, and $\Spec K$ by $\eta$ (we follow notations from \cite{T2}). Note that the natural commutative diagram 

\[
\begin{tikzcd}
\eta \arrow{d}  \arrow{r} & U \arrow{d}\\
T \arrow{r} & S
\end{tikzcd}
\]
is bi-Cartesian (easy to check from the very definition of semi-valuation ring). 

\begin{defn} We say that a semi-valuation ring $(\O, \m)$ is of {\it finite rank}, if the valuation ring $R$ associated with this pair is of finite rank (in the sense of Definition~\ref{rank}).
\end{defn}

\begin{lemma}\label{open-2} Let $\O$ be a finite rank semi-valuation ring with ring of semi-fractions $A$. Then the subset $\Spec A \subset \O$ is open.
\end{lemma}
\begin{proof}
We note that $A=S^{-1}\O$ is a localization of a ring $\O$ at a multiplicative set $S\coloneqq \O \setminus \m$. So we have an equality $A=\colim_{x\in S} \O_x$, thus 
\[
\Spec A=\lim_{x\in S} \Spec \O_x \cong \bigcap_{x\in S} \Spec \O_x \subset \Spec \O.
\]
Since the spectrum $\Spec R$ of the associated valuation ring $R$ has a finite underlying topological space, there are only finitely many options for subsets of $\Spec \O$ containing $\Spec A$. For any $x\in S$ the subset $\Spec \O_x$ contains $\Spec A$, so we conclude that exists some $x\in S$ such that $\Spec A = \Spec \O_x$, so $\Spec A$ is open in $\Spec \O$.
\end{proof} 
\smallskip

Before going to the proof of Stable Modification Theorem over semi-valuation bases, we need to show two basic lemmas. 

The following lemma is well-known and a form of it appears in \cite{SGA1}, but there is a connectedness assumption on $X$ that we would like to avoid. So we write down a proof following the ideas from \cite{SGA1}.

\begin{lemma}\label{dominate-torsor} Let $Y$ be a connected scheme and $f\colon X \to Y$ be a finite \'etale morphism. Then there is a finite group $G$ (actually it can be chosen to be a symmetric group $S_n$) and a $G$-torsor $g\colon  X' \to Y$ such that $g$ dominates $f$, i.e. there is a factorization 

$$
X' \xr{h} X \xr{f} Y.
$$
\end{lemma}
\begin{proof}
Note that the morphism $f\colon X \to Y$ is finitely presented (part of the definition of an \'etale morphism) and finite, so the degree function $\deg(X/Y)$ is locally constant on $Y$. Since $Y$ is connected; $f$ is of degree $d$ for some integer $d$. Then we consider $d$-th fiber power 
\[
X^{d/Y}=X\times_Y X \times_Y \dots \times_Y X.
\]

The \'etaleness of the morphism $f$ implies that the ``big diagonal'' $\Delta \subset X^{d/Y}$ is open (by the ``big diagonal'' we mean a union of all natural ``diagonals'' $X^{(d-1)/Y} \to X^{d/Y}$). But the separatedness of the morphism $f$ implies that the ``big diagonal'' is also closed in $X^{d/Y}$. This means that $X'\coloneqq  X^{d/Y}\setminus \Delta$ is open and closed in $X^{d/Y}$, so the natural morphism $g\colon  X' \to Y $ is finite (as closed in a finite $Y$-scheme) and \'etale (as open inside an \'etale $Y$-scheme). Moreover, there is the natural action of the symmetric group $S_d$ on $X'$ because functorially it represents a functor of $d$ fiberwise distinct points of $X$ on the category of $Y$-schemes. So, in order to check that $X'$ is an $S_d$-torsor, we need to check that the natural morphism of $Y$-schemes
\[
p\colon  G\times X' \to X'\times_Y X'
\]
is an isomorphism. Since both schemes are finite \'etale over $Y$, we can check this on geometric fibers. But this claim is absolutely clear on geometric fibers from the functorial description of $X'$. Thus $g\colon X' \to Y$ is indeed an $S_d$-torsor. 

Now we define a morphism $h\colon X' \to X$ as a map induced from the projection onto the first coordinate $p_1\colon X^{d/Y} \to X$. We need to check that $h$ is surjective, it can be done on geometric fibers. And again on geometric fibers it is clear from the functorial description of $X'$ and the definition of $d$.
\end{proof}

The next lemma is also well-known and we include it only for our convenience. 

\begin{lemma}\label{lift-fet} Let $(A, \m_A)$ be a local ring with a residue field $K$, and let $K'/K$ be a finite separable extension. Then we can lift it to a finite \'etale morphism $A \to A'$, such that $A'/\m_A A'=K'$. In particular, $A'$ is local. 
\end{lemma}
\begin{proof}
By the Primitive Element Theorem we know that $K' \cong K[T]/(\ov{P})$, where $\ov{P}\in K[T]$ is a monic separable polynomial. Lift this polynomial to an arbitrary {\it monic} polynomial $P\in A[T]$ and then $A'\coloneqq  A[T]/(P)$ does the job. 
\end{proof}
\smallskip

\begin{thm}\label{stable-modification-semivaluation} Let $(\O, \m)$ be a semi-valuation ring of a finite rank and let $f\colon  X \to S$ be a relative curve such that the base change $f_U\colon  X_U \to U$ is a semi-stable relative curve. There is a finite group $G$ and a faithfully flat finite and finitely presented morphism $\O \to \O'$ such that $\O'$ admits an action of $G$ with the following properties: 
	\begin{itemize}\itemsep0.5em
	\item The natural morphism $f\colon S' \to S$ is $G$-invariant with $S'=\Spec \O'$.
	\item The restriction of $f$ over $U$ defines a $G$-torsor $f_U\colon  S'_U \to U$ (in particular, it is finite etale).
	\item The base change $X_{S'}$ admits an $S'$-stable $U$-modification $g\colon  X' \to X_{S'}$.
	\end{itemize} 
\end{thm}
\begin{proof}
We start by considering the bi-Cartesian square 
\[
\begin{tikzcd}
\Spec K \arrow{d}  \arrow{r} & \Spec A \arrow{d}\\
\Spec R \arrow{r} & \Spec \O.
\end{tikzcd}
\]
We know by Theorem~\ref{stable-modification-valuation} that we can find a finite Galois extension $K \subset K'$ with $R'$ being the integral closure of $R$ in $K'$ such that the base change $X_{R'}$ admits a $R'$-stable $U$-modification $g\colon X' \to X_{R'}$. We use Lemma~\ref{lift-fet} to lift Galois extension $K \subset K'$ to a finite \'etale map $A \to A'$. And, moreover, we use Lemma~\ref{dominate-torsor} to find a finite \'etale $A$-algebra $A''$ such that $\Spec A'' \to \Spec A$ dominates $\Spec A' \to \Spec A$ and $\Spec A'' \to \Spec A$ is a torsor for some finite group $G$. Let $K''$ be $A''/\m_A A''$; this is not necessarilly a field, but still a finite \'etale $K$-algebra (so, it is a product of finite separable field extensions of $K$). Now let us denote by $R''$ the normalization of $R$ in $K''$; note that the morphism
\[
g_{R''}\colon  X'_{R''} \to X_{R''}
\]
is a $R''$-stable $U$-modification (because stability is a condition on geometric fibers).  \smallskip

Let us now denote $K''$ by $K'$, $R''$ by $R'$ and $A''$ by $A'$. So what we get from the previous discussion is that we have a $G$-torsor $\a\colon \Spec A' \to \Spec A$ for some finite group $G$ such that $X_{R'}$ admits a $R'$-stable $U$-modification $g\colon X' \to X_{R'}$, where $R'$ is the normalization of $R$ in the special fibre of $\a$. Note that by construction $G$ acts on $\Spec K'$ and $\Spec R'$ over $\Spec R$. \smallskip

Since $K'$ is a finite \'etale $K$-algebra and $R'\otimes_R K =K'$ (apply Lemma~\ref{ext-val-1} to factor fields of $K'$), we can find a finite number of elements $a_j \in R'$ such that $\{a_j\}$ is a basis of $K'$ as a vector space over $K$. Write $R'=\colim_{i\in I} R_i$ as a filtered colimit of $R$-finite type subalgebras $R_i$ of $R'$ containing all the elements $a_j$. Since they are $R$-torsion free modules (as the subrings of $R$-torsion free $R'$) they are actually flat over $R$. Then it is easy to see that $R_i \otimes_R K \to K'$ is an isomorphism (it is injective by $R$-flatness of $R_i$ and it is surjective since $R_i$ contains all $K$-basis elements $a_j$ of $K'$). Moreover, since $R_i$ are finite type and integral they are really finite. Then \cite[Theorem 7.10]{M1} implies that each $R_i$ is actually finitely presented as an $R$-module, and therefore also finitely presented as an $R$-algebra by \cite[IV\textsubscript{1}, 1.4.7]{EGA}. Now since $R'$ has an action of $G$ by $R$-algebra automorphisms, we can use Lemma~\ref{spread-group} (it is easy to see that all conditions of this Lemma are satisfied for a filtered system $\{R_i\}_{i\in I}$) to say that we can choose a subsystem $R_j$ such that each $R_j$ is $G$-stable subalgebra of $R$, and $R=\colim_{j\in J} R_j$. So a standard approximation argument allows to spread out the $R'$-stable $U$-modification $X' \to X_{R'}$ to an $R_j$-stable $U$-modification $g_j\colon X'_j \to X_{R_j}$.  \smallskip

Now we use \cite[Proposition 1.4.4.1(ii), (iii)]{CCO} to ``glue'' $\Spec R_j$ and $\Spec A'$ ``along'' $\Spec K'$ to achieve a faithfully flat, finite
, finitely presented $\O$-algebra $\O'$ (explicitly, $\O'$ is equal to $A'\times_{K'} R_j$) with an $\O$-algebra action of $G$. Let $S'=\Spec \O'$ and we claim that $X_{S'}$ admits an $S'$-stable $U$-modification. \smallskip

We know that $X_{R_j}$ has a $R_j$-stable $U$-modification $g_j\colon  X'_j \to X_{R_j}$, and $X_U$ is a relative semi-stable $U$-curve by the assumption. Therefore, $X_{K'}$ is also a relative semi-stable $K'$-curve, and $(X'_j)_U$ is canonically isomorphic to $X_{K'}$. Now we can use \cite[Proposition 1.4.4.1(iii)]{CCO} to ``glue'' $X'_j$ and $X_{A'}$ to an $S'$-stable $U$-modification $X'' \to X_{\O'}$\footnote{Alternatively, one can use \cite[Theorem 1.1]{T1} and \cite[Lemma 2.3.1 (ii)]{T2}.}. The last thing to observe is that the $U$-restriction of $\Spec \O' \to \Spec \O$ is a $G$-torsor because it coincides (by construction) with the $G$-torsor $\Spec A' \to \Spec A$.
\end{proof}

\section{Stable Modification Theorem for Curves Over General Bases}

Let us set up some notations for this section. 

\begin{setup}\label{setup5} We fix a quasi-compact quasi-separated scheme $S$, a stable under generalization quasi-compact subset $U$ that admits a scheme structure such that $U \to S$ is schematically dominant, and a relative $S$-curve  $f\colon X \to S$ (in the sense of Definition~\ref{def-curve}).
\end{setup}

\begin{rmk}\label{sch-dense-good} Any quasi-compact open subscheme $U'$ containing the underlying topological space of $U$ is schematically dense in $S$. Indeed, the natural map $U \to S$ clearly factors through the inclusion $U' \to S$, so the inclusion $U' \to S$ must be schematically dominant as $U \to S$ is.  
\end{rmk}

Before we start the discussion of the Stable Modification Theorem over $S$, we need to recall some facts about relative Riemann-Zariski spaces from \cite{T1} and \cite{T2}. 
\smallskip 

We consider the system of all $U$-modifications of $S$. This is a filtered system because any two $U$-modification $S_i$ and $S_j$ can be dominated by a $U$-modification obtained via taking a schematic closure of $U$ in the product $S_i \times_S S_j$. So we take a limit over this filtered system in the category of locally ringed spaces; we call this limit as a {\it relative Riemann-Zariski space} $\RZ_U(S)$. We will usually denote it just by $\S$ (it will not cause any confusion in what follows). It has a natural ``structure sheaf'' obtained as 
\[
\O_{\S}=\colim \pi_i^{-1}\O_{S_i}, \text{ where } \pi_i \text{ is the natural projection } \pi_i\colon  \S \to S_i.
\]

Note that although it is proven in \cite[IV\textsubscript{3}, \textsection 8]{EGA} that a filtered limit of schemes with affine transition maps exists as a scheme, it is extremely false without the affineness assumption. For example, one can show that the functor $\prod_{i=1}^{\infty} \mathbf P^1$ is not even an algebraic space. Thus there is no reason for $\S$ to be a scheme, but is rather just some abstract locally ringed space. It is difficult to work with this space, unless one proves results about the structure of this mysterious object. We summarize some of them below.

\begin{lemma}\label{temkin-RZ}\cite[Proposition 3.3.1]{T1}\cite[Proposition 2.2.1]{T2} Fix $U,S, \S$ as in the Setup~\ref{setup5}. Then
\begin{enumerate}
\item The underlying topological space of $\S$ is quasi-compact, quasi-separated, and surjects onto $S'$ for any $U$-modification $S' \to S$.
\item For each point $\mathfrak{s}\in \S$ the local ring $\O_{\S, \mathfrak{s}}$ has a natural structure of a semi-valuation ring with an isomorphism $\Spec A \cong U\times_S \Spec \O_{\S,\mfs}$, where $A$ is the ring of semi-fractions of $\O_{\S,\mfs}$.
\end{enumerate}
\end{lemma}
\begin{proof}
We set up dictionary between our notations and the notations in \cite{T2}. What we call $U$ (resp. $S$, resp. $\S$) is denoted by $Y$ (resp. $X$, resp $\X$) in \cite{T2}. Then the first claim is shown in \cite[Prop. 3.3.1]{T1} and \cite[Prop. 2.2.1]{T2}. And the second claim is also implicitly obtained in {\it the proof} of \cite[Prop. 2.2.1]{T2} (and \cite[Prop. 2.2.2]{T2}). \\

We note that the proof of \cite[Proposition 2.2.1]{T2} works under the assumption that $f\colon Y \to X$ is decomposable (a composition of an affine morphism followed by a proper one). We can use \cite[Theorem 1.1.3]{T2} to see that our morphism $U \to S$ is automatically decomposable. But in the case of $U \to S$ being an open immersion (the only case we will really use in this paper\footnote{Even though we allow more general $U$ in the formulation of Theorem~\ref{stable-modification-general}, the first step there is reduction to the case of an open $U$.}) it is much easier to show. Namely, we consider the complement $Z\coloneqq S\setminus U$, a closed subset of $S$. Since $S$ is quasi-compact quasi-separated and $U \to S$ is quasi-compact, \cite[Lemma 1.3]{C2007} says that there exists a finite type quasi-coherent ideal sheaf $\mathcal I$ with $\operatorname{V}(\mathcal I)=Z$. Then we define $S'\coloneqq {\bf{Proj}}_S \bigoplus_{i\in \N} \mathcal I^n$ to be the blow-up of $S$ along $\mathcal I$. Since $\mathcal I$ is finite type we conclude that $p\colon S' \to S$ is projective by \cite[II, 5.5.1]{EGA} (in particular, it is proper) . Since $\mathcal I|_U \cong \O_U$ we see that the $U$-restriction of $p\colon S' \to S$ is an isomorphism. Therefore, we have a commutative diagram: 
\[
\begin{tikzcd}
& S' \arrow{d}{p} \\
U \arrow{ur}{j'} \arrow{r}{j} & S
\end{tikzcd}
\]
\noindent where $p$ is proper and $j'$ is an affine open immersion since its complement $S' \setminus j'(U)$ is a Cartier divisor.
\end{proof}

This Lemma explains the significance of semi-valuation rings for the purposes of proving the Stable Modification Theorem:  we use various approximation techniques to reduce the case of a general $S$ to the case of curves over local rings of a relative Riemann-Zariski spaces $\S$, which in turn are semi-valuation rings. 

We will freely use the notion of ``$U$-admissible blow-up'' from \cite[Section 5.1]{RG}. We now provide the reader with the definition and main properties of this construction

\begin{defn}\label{adm-blowup} For an open immersion $U\to S$, a morphism $p\colon S' \to S$ is called a {\it $U$-admissible blow-up} if $p$ is equal to the morphism ${\bf{Proj}}_S \bigoplus_n \mathcal I^n \to S$, where $\mathcal I$ is a finitely generated, quasi-coherent sheaf of ideals on $S$ such that $\mathrm{V}(\mathcal{I})\cap U = \varnothing$. In particular, $p$ is of finite type. 
\end{defn}

We summarize the main properties of such blow-ups in the following lemma:

\begin{lemma}\label{cofinal} Let $S$ be a quasi-compact quasi-separated scheme, and let $U\subset S$ be a schematically dense quasi-compact open subscheme. Then any $U$-admissible blow-up $p\colon S' \to S$ is a projective admissible $U$-modification. Moreover, $U$-admissible blow-ups are cofinal among all $U$-modifications of $S$.
\end{lemma}
\begin{proof}
Projectivity follows the direct description of $p$ as the Proj of a finitely generated quasi-coherent sheaf of ideals, and \cite[II, 5.5.1]{EGA}. The fact that $p$ is an isomorphism can be easily seen from the fact that $\mathcal I|_U =\O_U$. Finally, admissible blow-ups are cofinal due to \cite[Corollaire 5.7.12]{RG} (or \cite[\href{https://stacks.math.columbia.edu/tag/081T}{Tag 081T}]{stacks-project} for a proof that does not mention algebraic spaces).
\end{proof}

\begin{thm}\label{stable-modification-general} Let $S, U$ be as in the Setup~\ref{setup5}, and let $f\colon X \to S$ be an $S$-curve that is semi-stable over $U$. Then there exist 
\begin{itemize}\itemsep0.5em
	\item A projective $U$-modification $h\colon S' \to S$ with a finite open Zariski covering $\cup_{i=1}^n V'_i = S'$ by quasi-compact opens $V'_i \subset S'$.
	\item A finite group $G_i$ and a finite, finitely presented, faithfully flat, and $U$-\'etale $G_i$-invariant morphism $t_i\colon W'_i \to V'_i$ for each $i\leq n$. In particular, the morphism $t\colon W'=\sqcup_{i=1}^n W'_i \to S$ is a $U$-\'etale covering. 
\end{itemize}

satisfying the following properties:

\begin{enumerate}
	\item The induced morphisms $t_{i,U}\colon W'_{i, U} \to V'_{i, U}$ are $G_i$-torsors.
	\item Each $X_{W'_i}$ admits a $W'_i$-stable $U$-modification $g_i\colon X'_i \to X_{W'_i}$.
\end{enumerate}
\end{thm}

\begin{rmk}\label{stable-modification-general-U-covering} Note that  the morphism \[t\colon \bigsqcup_{i=1}^n W'_i \to S\] from Theorem~\ref{stable-modification-general} is a quasi-projective $U$-\'etale covering. Indeed, it is a composition of the faithfully flat quasi-projective $U$-\'etale morphism \[\bigsqcup W'_{i} \to S'\] and the quasi-projective $U$-modification \[S' \to S.\] Both of those are $U$-\'etale coverings by Lemma~\ref{example-U-covering}, so their composition is also a quasi-projective $U$-\'etale covering.

Moreover, we recall that \cite[Theorem 1.1]{T1} shows that the $W'_i$-stable $U$-modifications $g_i\colon X'_i \to X_{W'_i}$ are also projective.
\end{rmk}

\begin{proof}
{\it Step 0. Approximation to the case of  finite type  $\Z$-schemes:} Note that \cite[Lemma 5.1]{T1} guarantees that any $x\in U$ has an open quasi-compact neighborhood $U_x\subset S$ such that $X$ has semi-stable fibers over $U_x$. Since $U$ is quasi-compact we can cover $U$ by finitely many of those to find some quasi-compact schematically dense open subscheme $U'\subset S$ (see Remark~\ref{sch-dense-good}) such that $X$ has semi-stable fibers over $U'$. We can replace $U$ by $U'$ without loss of generality (any $U'$-modification is also a $U$-modification). So, from now on, we may (and do) assume that $U$ is open in $S$. 

Now a standard approximation argument ensures that the triple $(S, U, X)$ comes a base change from some finite type $\Z$-scheme $S'$, an open schematically dense subscheme $U'$, and a relative curve $f'\colon X' \to S'$. Therefore, without loss of generality we can assume that $S$ is finite type over $\Spec \Z$ and $U$ is a schematically dense open subscheme of $S$. \\

{\it Step 1. We solve the problem over local rings of $\O_{\S}$}: Lemma~\ref{temkin-RZ} guarantees that all local rings of $\O_{\S}$ are semi-valuation rings. By Theorem~\ref{stable-modification-semivaluation} the only thing that we need to check is that under the assumption that $S$ is finite type over $\Z$, all these semi-valuation rings are of finite rank. In order to do this we need to observe that the proof of \cite[Proposition 2.2.1]{T2} actually shows a bit more than stated there explicitly. Namely, choose an arbitrary point $\mathfrak{s}\in \S$, then Lemma~\ref{temkin-RZ} guarantees that there is a natural structure of a semi-valuation ring on a ring $\O\coloneqq \O_{\S,\mfs}$. Denote the kernel of corresponding valuation by $\m$ (and let $y$ be the corresponding point of $\Spec \O$), and the corresponding ring semi-fractions $\O_{\m}$ by $A$. By the construction of relative Riemann-Zariski spaces we have a natural morphism of schemes $\pi_{\mfs}\colon \Spec \O \to S$; consider the image $y'\coloneqq \pi_{\mfs}(y)\in U$. Then the proof of \cite[Proposition 2.2.1]{T2} really proves that $A\cong \O_{U, y'}$. In particular, this is a ring essentially of finite type over $\Z$. Thus $K\coloneqq A/\m A$ is a field of finite transcendence degree over the prime field $\mathbf F$. So we can use Abhyankar's inequality \cite[Lemma 2.1.2]{T1} to conclude that $R\subset K$ must be of finite rank. 

Use Theorem~\ref{stable-modification-semivaluation} to get a finite group $G_{\mfs}$ and a faithfully flat finite finitely presented $G_{\mfs}$-invariant morphism $h_{\mfs}\colon \Spec \O'_{\mfs} \to \Spec \O_{\S, \mfs}$ such that $X_{\O'_{\mfs}}$ admits a $\O'_{\mfs}$-stable $U$-modification $g'_{\mfs}\colon X'_{\mfs} \to X_{\O'_{\mfs}}$ and $(h_{\mfs})_U\colon (\Spec \O'_{\mfs})_U \to (\Spec \O_{\S, \mfs})_U$ becomes a $G_{\mfs}$-torsor. \\

{\it Step 2. We spread out this stable modification over a local ring of some $U$-modification of $S$:} Let $\pi_i\colon \S \to S_i$ be the natural projection from a relative Riemann-Zariski space of $S$ onto some $U$-modification $S_i$ and denote by $s_i=\pi_i(\mfs)$ the image of $\mfs$ in $S_i$. Then since filtered colimits commute with each other we see that 
\[
\O_{\S, \mfs}=\colim_{i} \O_{S_i, s_i}.
\]
We have a faithfully flat finite finitely presented $G_{\mfs}$-invariant morphism $\Spec \O'_{\mfs} \to \Spec \O_{\S, \mfs}$ which is a $G_{\mfs}$-torsor over $U$. Therefore, we can use the standard approximation techniques to 
find a faithfully flat finite finitely presented $G_{\mfs}$-invariant morphism $\Spec \O'_{s_i} \to \Spec \O_{S_i, s_i}$ which becomes a $G_{\mfs}$-torsor over $U$.
Moreover, we can use standard approximation techniques to get that (possibly after enlarging $i$) we can assume that $X_{\O'_{s_i}}$ admits a $\O'_{s_i}$-stable $U$-modification. 

Let us summarize what we have achieved so far. For each point $\mfs \in \S$ there exists a big enough $i$ such that for $s_i\coloneqq \pi_i(s)$ there is a $U$-modification $S_i \to S$ and a faithfully flat finite finitely presented $G_{\mfs}$-invariant (for some finite group $G_{\mfs}$) morphism $\Spec \O'_{s_i} \to \Spec \O_{S_i, s_i}$ such that it induces a $G_{\mfs}$-torsor over $U$ and $X_{\O'_{s_i}}$ admits a $\O'_{s_i}$-stable $U$-modification. \\

{\it Step 3. We spread out this stable modification over some Zariski open subspace of the point $s_i$:} Basically, we repeat the same argument as above, but now we use that \[ \O_{S_i, s_i}=\colim_{s_i \in V \text{-affine}} \O_{S_i}(V)\] This means that we can again run the same approximation argument to find some open affine $V_{i} \subset S_i$ that contains $s_i$ and a faithfully flat finite finitely presented $G_{\mfs}$-invariant morphism $W_i \to V_i$ such that it induces $G_{\mfs}$-torsor over $U$ and $X_{W_i}$ admits a $W_i$-stable $U$-modification. Note that by the very construction $V_{\mfs}\coloneqq \pi_i^{-1}(V_i)$ is an open subset of $\S$ containing $\mfs$. \\

{\it Step 4. We use quasi-compactness of $\S$ to finish the argument:} Since $\mfs$ was an arbitrary point of $\S$, the previous construction provides us with a covering of $\S$ by open subsets $V_{\mfs}$. Lemma~\ref{temkin-RZ} says that $\S$ has quasi-compact underlying topological space. Therefore we can choose a finite covering of $\S$ by means of such opens $V_{\mfs}$. Denote the corresponding $\mfs$'s just by $(\mfs_j)_{j\in J}$ for some finite set $J$. Then after changing some notation, the previous discussion gives us a set of finite groups $G_j$, $U$-modifications $h_j\colon S_j \to S$, open affine subsets $V_j \subset S_j$ containing a point $\pi_j(\mfs_j)$, and finite, finitely presented and finitely presented $G_j$-invariant morphisms $W_j \to V_j$ which become $G_j$-torsors over $U$, with the condition that $X_{W_j}$ admits a $W_j$-stable $U$-modification for all $j$. 
\smallskip 

Recall that a system of all $U$-modifications is filtered, so we can find an $U$-modification $S' \to S$ that dominates all of the $S_j$. Then the data above induces open quasi-compact subsets \[V'_j\coloneqq V_j\times_{S_j}S' \subset S'\] and finite, finitely presented and faithfully flat $G_j$-invariants morphisms \[W'_j\coloneqq W_j\times_{S_j} S' \to V'_j,\] such that they induce $G_j$-torsors over $U$ and $X_{W'_j}$ admits a $W'_j$-stable $U$-modification for all $j$.

We claim that open quasi-compact $V'_j$ cover $S'$. To show this we note that under the natural projections $\pi_j\colon \S \to S_j$ and $\pi\colon \S \to S'$ we have an equality 
\[
V_{s_j}\coloneqq \pi_j^{-1}(V_j)=\pi^{-1}(V'_j).
\]
This means that $\pi^{-1}(V'_j)$ form a covering $\S$. Therefore, we conclude that the $V'_j$ cover $S'$ because the projection $\S \to S'$ is surjective by Lemma~\ref{temkin-RZ}.

The fact that we can choose $h\colon S' \to S$ to be a projective $U$-modification follows from Lemma~\ref{cofinal}.
\end{proof}


\section{Schematic Version of Local Uniformization}

In what follows, we will freely use the following notations

\begin{defn}\label{G-inv} We say that a morphism of $S$-schemes $f\colon X \to Y$ is {\it $G$-invariant} for some group $G$, if $X$ admits an action of $G$ by $S$-automorphisms, and $f\circ g=f$ for all $g\in G$.
\end{defn}

Whenever we say that a morphism is $G$-invariant, we implicitly assume that there is a given action of $G$ on $X$ over $S$.

\begin{defn}\label{defn-U-torsor} A morphism $f\colon  X\to Y$ is called a {\it $G$-torsor over an open subset $V \subset Y$}, if $f$ is a $G$-invariant morphism that factors through $V$, and 
\[
f_V\colon X \to V
\]
is a $G$-torsor over $V$ (i.e. it is flat finitely presented and the natural action morphism $X\times G \to X\times_V X$ is an isomorphism). In particular, $V$ must be equal to the image $f(X)\subset Y$.
\end{defn}

\begin{lemma}\label{U-torsor} Let $G, H$ be two finite groups, and $f\colon X' \to X$ be a $G$-invariant map of schemes that is a $G$-torsor over its open image $V\coloneqq f(X')$ (in the sense of Definition~\ref{defn-U-torsor}). Let $g\colon X' \to Y$ be any $G$-invariant morphism, and let $h\colon Z \to Y$ be an $H$-invariant morphism that is an $H$-torsor over its open image $W\coloneqq h(Z)\subset Y$. Consider the base change morphism $h'\colon X''\coloneqq Z\times_Y X' \to X'$ and the composition $t\colon X'' \to X$ in the commutative diagram: 
\[
\begin{tikzcd}
X & X' \arrow{d}{g} \arrow{l}{f} & X'' \arrow{d} \arrow{l}{h'}\\
& Y  & Z \arrow{l}{h}
\end{tikzcd}
\]
The morphism $t$ is $G\times H$-invariant, and it is a $G\times H$-torsor over the open $f(g^{-1}(W))$.
\end{lemma}
\begin{proof}
Firstly, we note that the subset $f(g^{-1}(W))$ is indeed open since $f$ is flat and locally finitely presented by definition. Now we observe that since $g\colon X' \to Y$ is $G$-invariant, so $g^{-1}(W)$ is a $G$-stable open subset of $X'$, the map $p\colon g^{-1}(W) \to X$ is a $G$-torsor over $f(g^{-1}(W))$. Therefore we can replace $Y$ with $W$ (so, we also replace $X'$ with $X'\times_Y W$, $Z$ with $Z\times_Y W$, and $X''$ with $X''\times_Y W$) without affecting any assumptions of the lemma. Hence we can assume $h\colon Z \to Y$ is actually an $H$-torsor.

Now observe that we can replace $X$ with $V'\coloneqq f(X')$ without affecting any assumptions of the lemma, so we can assume that $f\colon X' \to X$ is a $G$-torsor.  Then we have a diagram

\[
\begin{tikzcd}
X & X' \arrow{d}{g} \arrow{l}{f} & X'' \arrow{d} \arrow{l}{h'}\\
& Y  & Z \arrow{l}{h},
\end{tikzcd}
\]
where $f$ is a $G$-torsor and $h'$ is an $H$-torsor (as the base change of the $H$-torsor $h\colon Z\to Y$), and we need to show that $t=f\circ h'\colon X'' \to X'$ is a $G\times H$-torsor. We see that there is a natural action of the group $G\times H$ on $X''$ since $g$ and $h$ are, respectively, $G$ and $H$-invariant morphisms. Moreover, the morphism $h'$ is $H$-invariant and $G$-equivariant. Finally, we note that $t$ is a finite \'etale morphism as a composition of two finite \'etale morphism. 

It now suffices to check that the natural action $X''\times (G\times H) \to X'' \times_{X} X''$ is an isomorphism. We can check this on geometric points (because $t$ is finite \'etale). So we can reduce to the situation $X=\Spec k$ for an algebraically closed field $k$. Then $X'$ is a disjoint union of $\Spec k$, and the $G$-invariance of $g\colon X' \to Y$ implies that $g$ factors through a point in $Y$. This implies that we can also assume that $Y$ is just a point $\Spec k$. In this case it is easy to see that $X''$ is just a disjoint union of $\#(G\times H)$ copies of $\Spec k$ with a simply transitive action of $G\times H$. Therefore $X'' \times (G\times H) \to X'' \times_X X''$ is an isomorphism in this case. 
\end{proof}


We recall another technical definition that we will need to formulate the main theorem of this section. 

\begin{defn}\label{defn-henselian-field} A valuation field $K$ is {\it henselian} if its valuation ring $\O_K$ is a henselian local ring.
\end{defn}

We summarize the main facts about henselian valuation rings in the following lemmas: 



\begin{lemma}\label{prop-hens} Let $K$ be a rank-$1$ valuation field with the valuation ring $\O_K$. 
\begin{itemize} 
\item Suppose that $\O_K$ is complete with respect to its valuation topology. Then $K$ is a henselian field.
\item Suppose $K$ be henselian, and let $L$ be an algebraic extension of $K$. Then $L$ has a unique structure of a henselian field compatible with $K$. Moreover, the associated valuation ring $\O_L$ is equal to the normalization of $\O_K$ in $L$.
\end{itemize}
\end{lemma}
\begin{proof}
We start with the proof of the first statement. \cite[Proposition 2.4.3]{Ber} implies that $K$ is henselian if and only if $K$ is quasi-complete in the sense of \cite[Definition 2.3.1]{Ber}. Now quasi-completeness of $\O_K$ follows from \cite[Theorem 3.2.4/2]{BGR}. \medskip

Now we prove the second part. We note that \cite[Proposition 2.4.3]{Ber} implies that $K$ is quasi-complete. In particular, there is a unique extension of the norm on $K$ to a norm on $L$. Now \cite[Theorem 10.4]{M1} implies that $\O_{L}$ is equal to the normalization of $\O_K$ in $L$. Finally, it is easy to see from \cite[Definition 2.3.1]{Ber} that $L$ is quasi-complete. Therefore, $L$ is henselian by \cite[Proposition 2.4.3]{Ber}. 
\end{proof}


The main goal of this section is to prove the following result.

\begin{thm}\label{schematic-uniformization} Let $K$ be a henselian rank-$1$ valuation field with a valuation ring $\O_K$, and let $X \to \Spec \O_K$ be a flat, finitely presented morphism that is smooth over $\Spec K$. Then there is finite Galois extension $K\subset K'$, a finite extension $K' \subset K''$, and a finite set $I$ of triples $\{g_i\colon  X'_i \to X_{\O_{K'}}, h_i\colon X''_i \to X'_{i, \O_{K''}}, G_i, V_i\}_{i\in I}$ with the following properties: 
\begin{enumerate}
\item The morphism $g\colon \bigsqcup_{i\in I} X'_i \to X_{\O_{K'}}$ is a quasi-projective, $K'$-\'etale covering with each $X'_i$ being a finitely presented, $K'$-admissible $\O_{K'}$-scheme.
\item For all $i\in I$, $G_i$ is a finite group, the morphism $g_i\colon  X'_i \to X_{\O_{K'}}$ is $G_i$-invariant, and its base change $g_{i,K'}\colon X'_{i, K'} \to X_{K'}$ is a $G_i$-torsor over its image $V_i\coloneqq g_i(X'_{i,K'}) \subset X_{K'}$.
\item For all $i\in I$, $X'_i$ admits a structure of a successive semi-stable $K'$-smooth curve fibration over $\Spec \O_{K'}$.
\item For all $i\in I$, $h_i\colon X''_i \to X'_{i, \O_{K''}}$ is a $K''$-modification with $X''_i$ as $K''$-smooth, polystable $\O_{K''}$-scheme. 
\item If $K$ is discretely valued, one can choose $X''_i$ to be strictly semistable over $K''$. 
\end{enumerate}
\end{thm}

\begin{rmk} Lemma~\ref{prop-hens} ensures that $K'$ (resp. $K''$) has the unique structure of a rank-$1$ valuation field compatible with that structure on $K$. In particular, $\O_{K'}$ (resp. $\O_{K''}$) makes sense as the valuation ring associated to that valuation. 
\end{rmk}

\begin{rmk} Lemma~\ref{prop-hens} also implies that Theorem~\ref{schematic-uniformization} holds for a complete rank-$1$ valuation field $K$. 
\end{rmk}


The idea of our proof is rather easy in principle:  we first reduce to the case when $X$ has a structure of a successive curve fibration over a base $S$, then we ``resolve'' each curve fibration by a semi-stable one using Theorem~\ref{stable-modification-general}. The last step is to apply \cite[Theorem 5.2.16]{AdipLiuPakTemkin} to ``resolve'' each $X'_i$ by a polystable scheme. There two main difficulties in this approach. \smallskip

The first issue is to control the torsor condition on $K$-fibers. The problem already appears in the case of finite field extensions: a tower of two finite Galois extensions is not necessarily Galois. To resolve this issue we use (a trivial observation) that a composite of two Galois extension is Galois. So, if we try to resolve our tower of curve fibrations step by step from top to bottom, we can hope to preserve the $G$-torsor condition at each step of our inductive argument. This is exactly (with some extra technical complications) what we do in the proof below. \smallskip

The other problem is to verify the assumptions of \cite[Theorem 5.2.16]{AdipLiuPakTemkin}, i.e. we need to construct a structure of a log smooth log variety on a successsive semi-stable $K$-smooth curve fibration over $\Spec \O_K$. This issue is resolved in Appendix~\ref{log} by a spreading out argument to the well-known case.  \smallskip

We prefer firstly to deal with the case of curve fibrations (over a general base) and then use some other techniques to reduce the general case to the case of a successive curve fibration. \smallskip

The next proposition is rather lengthy to formulate; unfortunately, we do not know how to avoid  the heavy notation.

\begin{lemma}\label{main-lemma} Let $\O_K$ be a henselian valuation ring of rank-$1$ with fraction field $K$, and $S=\Spec \O_K$. Suppose that $f\colon X \to S$ is a morphism that can be written as a composition
\[
X=X_n \xr{f_n} X_{n-1} \xr{f_{n-1}}X_{n-2} \xr{f_{n-2}} \dots \xr{f_{2}}X_1 \xr{f_1}X_0=S
\]
such that each morphism $f_i$ is a relative curve (in the sense of Definition~\ref{def-curve}) that is $K$-smooth. Then there is a finite Galois extension $K\subset K'$ with $\O_{K'}$ the ring of integers in $K'$, a finite non-empty set $J$, a set of finite groups $\left(G_{j}\right)_{j\in J}$ indexed by $J$, and commutative diagrams: 

\[
\begin{tikzcd}[column sep=13ex]
X_{n, \O_{K'}} \arrow{d}{f_{n, \O_{K'}}}  & X'_{n,j}  \arrow{l}{g_{n,j}} \arrow{d}{h_{n,j}}\\
X_{n-1, \O_{K'}} \arrow{d}{f_{n-1, \O_{K'}}} & X'_{n-1, j} \arrow{l}{g_{n-1,j}} \arrow{d}{h_{n-1, j}}\\
\vdots \arrow{d}{f_{2, \O_{K'}}} & \vdots \arrow{l}{g_{k,j}} \arrow{d}{h_{2, j}}\\
X_{1, \O_{K'}}\arrow{d}{f_{1, \O_{K'}}} & X'_{1, j} \arrow{ld}{h_{1, j}} \arrow{l}{g_{1,j}} \\
\Spec {\O_{K'}} &  
\end{tikzcd}
\]
of flat, finitely presented $S'\coloneqq \Spec \O_{K'}$-schemes such that 
\begin{enumerate}
\item $X'_{k, j}$ admits an $\O_{K'}$-action of the group $G_j$ for any $k\geq 1$ and $j\in J$. Moreover, $g_{k,j}$ is $G_j$-invariant and $h_{k,j}$ is $G_j$-equivariant for any $j\in J, k=1, \dots, n$.
\item The $K'$-restriction $(g_{k, j})_{K'}\colon (X'_{k, j})_{K'} \to (X_{k})_{K'}$ is a $G_j$-torsor over its (open) image $V_{k,j}$ for $k\geq 1$.
\item $h_{k, j}$ a relative semi-stable $K'$-smooth curve\footnote{This means that $h_{k, j}$ is a relative semi-stable curve fibration that is smooth on generic fibers.} for any $j\in J$ and any $k\geq 1$.
\item $g_{k,j}$ is quasi-projective for any $k\geq 1$ and $j\in J$. Moreover, the map $g_k\colon \sqcup_{j\in J}X'_{k,j} \to X_{k, R'}$ is a $K'$-\'etale covering for any $k\geq 1$.
\item If $f_i$ are quasi-projective for every $i\geq 1$. Then the geometric quotient $X'_{n,j}/G_j$ exists as a flat, finitely presented $S'$-scheme for any $j\in J$. Moreover, in this case one can choose $X'_{n,j}$ so that the set 
\[
\left(X'_{n,j}/G_{j}, \ \ov{g_{n,j}}\right)_{j\in J}
\]
can be obtained from $X_{n, \O_{K'}}$ as a composition of Zariski open coverings and $K$-modifications (in the sense of Definition~\ref{comp-alg}). 
\end{enumerate}
\end{lemma}
\begin{proof}
We only explain the main idea of the proof here and refer to Lemma~\ref{main-lemma-4} for a detailed proof of this Lemma. 

The key idea is to argue by descending induction on $n$ to construct the desired tower of semi-stable curve fibrations step by step. At each induction step we use Theorem~\ref{stable-modification-general} to ``resolve'' a new layer of our tower (and possibly increasing the set $J$). Then there are two issues: the first is to lift a group action on a semi-stable modification, and the second is to control the $G$-torsor condition. We deal with the first issue by invoking the Temkin's uniqueness result for {\it stable} modification over a normal base (\cite[Theorem 1.2]{T1}). And Lemma~\ref{U-torsor} allows us to effectively control the torsor condition. Details of this argument are carried over in Section~\ref{pizdec}.
\end{proof}

A delicate aspect of Theorem~\ref{schematic-uniformization} is the finite presentation condition. Since the ring $\O_K$ is usually not noetherian, we need to be able to control that all our constructions preserve finite presentation condition. This can be done due to a (well-known) lemma below: 

\begin{lemma}\label{admissible} Let $R$ be a valuation ring with the fraction field $K$, and let $X$ be a flat $R$-scheme locally of finite type. Then
\begin{itemize}\itemsep0.5em
\item $X$ is locally finitely presented over $R$.
\item Any $K$-admissible blow-up $X'$ of $X$ is a flat and locally finitely presented $R$-scheme. In particular, the blow-up morphism $p\colon X' \to X$ is locally of finite presentation.
\end{itemize}
\end{lemma}
\begin{proof}
The first claim follows from \cite[Th\'eor\`eme (3.4.1)]{RG} (an alternative reference is \cite[\href{https\colon //stacks.math.columbia.edu/tag/053E}{Tag 053E}]{stacks-project}) and faithfully flat descent from the henselisation of $R$. As for the second claim we recall that 
$X$ is $R$-flat if and only if $\O_X$ is $R$-torsion free. Then any $K$-admissible blow-up $X'\coloneqq {\bf{Proj}}_X \bigoplus_{i\in \N} \mathcal I^n$ of $X$ is again $R$-flat and locally of finite type over $R$. Hence, it is locally of finite presentation by the first part of the proof.
\end{proof}

\begin{rmk} Before we start the proof of Theorem~\ref{schematic-uniformization}, we want to point out that it will be crucial for the argument that we do not make any properness or connectedness assumptions in Theorem~\ref{stable-modification-general}. The proof starts by choosing an open cover of $X_K$ by $U_i$ with  \'etale maps to $\mathbf{A}^n_K$. This defines relative curves $U_{i, K} \to \A^{n-1}_K$ and then after suitable modification we end up in the situation where Lemma~\ref{main-lemma} is applicable. But in order for this to work,   we cannot impose any properness or geometric connectivity assumptions in Theorem~\ref{stable-modification-general}.
\end{rmk}

Now we are finally ready to prove Theorem~\ref{schematic-uniformization}.

\begin{proof}[Proof of Theorem~\ref{schematic-uniformization}]
We start the proof by noting that if $X$ has a structure of a successive curve fibration over $\O_K$, then Lemma~\ref{main-lemma} constructs $K'/K$ and $g_i\colon X'_i \to X_{\O_{K'}}$ satisfying properties $(1)$, $(2)$, and $(3)$. So we want to reduce the question to the situation when $X$ has a structure of a successive curve fibration over $\O_K$ so that Lemma~\ref{main-lemma} is applicable. Then we separately construct $K''/K$ and $h_i\colon X''_i \to X'_{i, \O_{K''}}$ using  \cite[Theorem 5.2.16 and Corollary 5.1.14]{AdipLiuPakTemkin}. 

We note that Lemma~\ref{admissible} guarantees that for the purpose of the proof we can pass to a finite covering of $X$ by quasi-compact open subsets and $K$-admissible blow-ups of those. Since the generic fibre $X_{K}$ is smooth and quasi-compact over $\Spec K$, we can choose a finite covering of $X_K$ by quasi-compact (affine) open subschemes $U_i$ with \'etale $K$-morphisms $U_i \to \mathbf A^{n_i}_K$ for some ${n_i}$. Then \cite[Lemma 2.5.1]{T3} says that there is a $K$-admissible blow-up $\pi\colon X' \to X$ and a finite open quasi-compact covering $\cup_{i=1}^m X'_i = X'$ such that $X'_{i, K} \cap X_K =U_i$. We use the observation from the first sentence of this paragraph to replace $X$ with $X'_i$, so we can assume that the generic fibre of $X$ admits an \'etale morphism to some $\mathbf A^d_K$. 

Now we are in the situation that $X$ is finitely presented and flat over $\Spec \O_K$ and its generic fiber admits an \'etale morphism $f''\colon X_K \to \mathbf A^{d}_{K}$. Compose it with projection on the first $d-1$ coordinates to get a morphism $f''_d\colon X_K \to \mathbf A^{d-1}_K$ that is a smooth relative curve! For technical reasons it will be simpler to consider it as a morphism $f''_d\colon X_K \to (\P^1_K)^{d-1}$, which is also, certainly, a smooth relative curve. The reason is that we want to use \cite[Theorem 2.4 and Remark 2.5]{C2007} to find a $K$-admissible blow-up $\pi\colon X' \to X$ and a morphism $f'_d\colon X' \to (\P^1_{\O_K})^{d-1}$ extending $f''_d\circ \pi_K$. And we can use the above observation to replace $X$ with $X'$, so we assume that $X$ admits a map to $(\P^1_K)^{d-1}$ that is a smooth relative curve over the generic fibre. 

Under the assumption as above, we almost have a structure of a successive $K$-smooth curve fibration on $X$. Indeed, we have a map $f'_d\colon X \to (\P^1_{\O_K})^{d-1}$ over $\O_K$ that over $K$ is a smooth relative curve, and also we have projection morphisms $p_i\colon  (\P^1_{\O_K})^i \to (\P^1_{\O_K})^{i-1}$ that are $K$-smooth curve fibrations. The only issue is that $f'_d$ is not necessary flat, but we deal with that issue by invoking \cite[Corollaire 5.7.10]{RG} (or \cite[\href{https://stacks.math.columbia.edu/tag/081R}{Tag 081R}]{stacks-project}) to find a $K$-admissible blow up $r_{d-1}\colon  X_{d-1} \to (\P^1_{\O_K})^{d-1}$ such that the strict transform $f_d\colon  X' \to X_{d-1}$ is flat. Then we claim that $f_d$ is automatically a relative curve. Indeed, it is flat and \cite[\href{https://stacks.math.columbia.edu/tag/0D4H}{Tag 0D4H}]{stacks-project} implies that the relative dimension of all fibers is $\leq 1$. Now \cite[\href{https://stacks.math.columbia.edu/tag/02FZ}{Tag 02FZ}]{stacks-project} implies that the dimension of all non-empty fibers is $\geq 1$ as $X'_K$ is dense in $X'$. Thus, all non-empty fibers are of pure dimension $1$. 

Now we consider the map $X_{d-1} \to (\P^1_{\O_K})^{d-2}$. It is not necessarily flat, but we can find those $K$-admissible blow-ups step by step (also affecting previous choices of $f_i$'s via base change) to finally come up with a structure of a successive $K$-smooth curve fibration on a new $X'$. Then we just replace $X$ with $X'$ as we did above and use Lemma~\ref{main-lemma} to find a finite Galois extension $K \subset K'$ and morphisms $g_i\colon X'_i \to X_{\O_{K'}}$ with the source $X'_i$ being a successive semi-stable $K'$-smooth curve fibration over $\Spec \O_{K'}$. Now we say that since $K$-admissible blow-ups are $\O_K$-projective, we get that the final morphisms $X'_i \to X_{\O_{K'}}$ are quasi-projective and they are finitely presented by Lemma~\ref{admissible}.

We show how to construct a finite extension $K' \subset K''$ such that $X'_{i, \O_{K''}}$ admits a $K''$-modification $h_i\colon X_i'' \to X'_{i, \O_{K''}}$ with a $K''$-smooth polystable $X''_i$. The usual approximation type argument allows to assume that $K$ is algebraically closed. In this situation, the value group of $\O_{K}$ is divisible, so the results from \cite{AdipLiuPakTemkin} are applicable. Namely, \cite[Theorem 5.2.16 and Corollary 5.1.14]{AdipLiuPakTemkin} imply that there is a $K$-modification $h_i\colon X_i'' \to X'_{i}$ with the desired property once we know that there is a vertical\footnote{This means that the log structure is trivial on the generic fiber, i.e. $\M_{X}|_{X_{K'}}\simeq \O_{X_{K'}}^\times$.} log structure on each $X'_i$ making it into a log smooth log variety\footnote{Look at Definition~\ref{log-variety-temkin} and Definition~\ref{log-smooth-temkin} for the precise meaning of these words.}. However, this structure is constructed Theorem~\ref{log-str-semist}. This constructs the desired modification. 

The last thing we are left to show is that $X''_i$ can be chosen to be strictly semi-stable over $K''$ if $K$ is discretely valued. Proposition~\ref{noeth-case} ensures that $X'_i$ are log smooth over $\O_{K'}$ (with its standard log structure) if $K'$ is discretely valued. It is a classical fact (that essentially goes back to \cite{TorEmb}) that a log smooth scheme over a discretely valued can be modified into a strictly semi-stable one after a finite extension of $K$. We sketch a proof because we can not find it being explicitly formulated in the literature. 

First, \cite[Theorem 5.10]{Niziol} implies there is a log-blow-up (in particular, it is a $K'$-modification) $Y'_i \to X'_i$ such that $Y'_i$ is classically regular, log-regular $\O_{K'}$-scheme for each $i$. Then \cite[Exp.\ III, Prop. 3.6.5]{deGabber} ensures that, \'etale locally, each $Y'_i$ admits an \'etale morphism to 
\[
\Spec \O_{K'}[T_0,\dots, T_l]/(T_0^{e_0}\dots T_l^{e_l}-\pi)
\]
for some uniformizer $\pi$ and  $e_i\in \Z_{\geq 0}$ with at least one $e_i\neq 0$. Then the discussion before \cite[Theorem 3.1.7]{Temkin-absolute} describes how the combinatorial algorithms from \cite{TorEmb} can be used to find a quasi-finite extension of discretely valued rings $\O_{K'} \subset \O''$ such that $Y'_{i, \O''}$ admits a $K''\coloneqq \rm{Frac}(\O'')$-modification $X''_i \to Y'_{i, \O''}$ that is strictly semistable. The henselian assumption on $K$ implies that $\O''=\O_{K''}$.  
\end{proof}

\section{Analytic Version of Local Uniformization}

The idea for the proof of analytic version of the Local Uniformization result is to reduce the question to the schematic case. We do this by means of local algebraization. For what follows, we fix a complete rank-$1$ valuation ring $K$ with a valuation ring $\O_K$. \medskip

Before starting the proof, we briefly discuss some notions related to formal schemes and their adic generic fibers. We refer the reader to Appendix~\ref{adic} for a more detailed discussion. Here we mention only the main notions.

A {\it formal $\O_K$-scheme} will always mean an $I$-adic formal $\O_K$-scheme for a(ny) ideal of definition $I\subset \O_K$. It is easily seen to be independent of the choice of $I$. And by the {\it completion} $\wdh{X}$ of a finite type $\O_K$-scheme $X$, we always mean the $I$-adic completion. The same terminology is used in \cite{B}. This abuse of notation is common in $p$-adic geometry, but might be non-common in the other areas. 

A {\it rigid-analytic space over $K$} will always mean a locally topologically finite type $(K, \O_K)$-adic space in the sense of \cite{H1}. It is not necessary to use this approach to the foundations of non-archimedean geometry for our purposes, but we find it to be more convenient in our subsequent work. The reader, who is more familiar with the classical Tate approach to rigid-analytic geometry (as one used in \cite{B}), may safely use it in what follows. For a precise relation between those two notions, we refer the reader to \cite[Proposition 4.5]{H1}.

\begin{thm}\label{main-main} Let $X$ be a quasi-compact and quasi-separated smooth rigid-analytic space over $\Spa(K, \O_K)$ with a given admissible quasi-compact formal model $\X$. Then there is a finite Galois extension $K \subset K'$, a finite extension $K' \subset K''$, a finite number of morphisms of admissible formal $\O_{K'}$-schemes (resp. $\O_{K''}$-schemes ) $g_i\colon \X'_i \to \X_{\O_{K'}}$ and $h_i\colon \X''_{i} \to \X'_{i, \O_{K''}}$, such that
\begin{itemize}\itemsep0.5em
\item Each $\X'_i$ admits an action of a finite group $G_i$ such that $g_i\colon \X'_i \to \X_{\O_{K'}}$ is $G_i$-invariant for each $i$.
\item The morphism $g\colon \X'\coloneqq  \sqcup_i \X'_i \to \X_{\O_{K'}}$ is a rig-\'etale covering (in the sense of Definition~\ref{rig-etale-covering}). 
\item On the generic fiber, each $\X'_i$ becomes a $G_i$-torsor over its (quasi-compact) open image in the adic generic fiber $X_{K'}=\X_{K'}$.
\item Each $\X'_i$ is formally quasi-projective over $\O_{K'}$ (in the sense of Definition~\ref{quasiproj-formal}) and has a structure of a successive formal semi-stable rig-smooth curve fibration (in the sense of Definition~\ref{def-ss-formal}).
\item Each $h_i\colon \X''_i \to \X'_{i, \O_{K''}}$ is a rig-isomorphism, and $\X''_i$ is rig-smooth, polystable formal $\O_{K''}$-scheme (in the sense of Definition~\ref{poly-def-formal}).
\item If $K$ is discretely valued, one can choose $\X''_i$ to be strictly semistable over $K''$ (in the sense of Definition~\ref{semi-def-formal}). 

\end{itemize}
\end{thm}
\begin{proof}
We start by noting that the statement is local on $\X$, so we can assume that $\X$ is an affine admissible rig-smooth formal model. We use \cite[Theorem 3.1.3]{T3} (it essentially boils down to \cite[Th\'eorem\`e 7 on page 582 and Remarque 2(c) on p.588]{Elkik} and \cite[Proposition 3.3.2]{T0}) that says that an affine rig-smooth formal scheme $\X$ can be algebraized to an affine flat finitely presented $\O_K$-scheme $Y$ with smooth generic fibre $Y_K$.  We apply Theorem~\ref{schematic-uniformization} to find a finite Galois extension $K'/K$, an extension $K''/K$, and morphisms $g'_i\colon  Y'_i \to Y_{\O_{K'}}$, $h'_i\colon Y''_i \to Y'_{\O_{K''}}$ with all the properties from Theorem~\ref{schematic-uniformization}. Then we pass to the $I$-adic completions of those schemes $\wdh{g'_i}\colon \wdh{Y'_i} \to \wdh{Y_{\O_{K'}}}$, $\wdh{h'_i}\colon \wdh{Y''_i} \to \wdh{Y'_{\O_{K''}}}$ and apply \cite[Lemma 3.2.2]{T3}, Lemma~\ref{completion-ss}, Lemma~\ref{Galois-analytification} and an isomorphism $\wdh{Y_{\O_{K'}}} \cong \X_{\O_{K'}}$ to get morphisms 
\[
g_i\coloneqq \wdh{g'_i}\colon  \wdh{Y'_i} \to \X_{\O_{K'}} \text{ and } h_i\coloneqq \wdh{h'_i}\colon  \wdh{Y''_i} \to \wdh{Y'_{i, \O_{K''}}}=\X_{\O_{K''}}
\] 
which satisfy all the properties from the formulation of the Theorem besides formal quasi-projectivity. But it is actually also formally quasi-projective over $\Spf \O_{K'}$ because  Theorem~\ref{schematic-uniformization} ensures that $Y_i$ is quasi-projective over an affine scheme $Y$ for all $i$. Thus it is also quasi-projective over $\Spec \O_{K'}$. And the completion of a finite type quasi-projective $\O_{K'}$-scheme is a quasi-projective topologically finite type formal $\O_{K'}$-scheme (since the completion along a finitely generated ideal in $\O_{K}$ preserves the special fibre).
\end{proof}

\section{Uniformization by Quotients of Successive Semi-Stable Curve Fibrations}\label{section-final}

In the section we fix a complete non-archemidean field $K$. We show that any admissible quasi-compact and quasi-separated formal model of a smooth rigid-analytic space locally admits an uniformization (in a precise sense explained below) by a ``good'' quotient of a successive semi-stable curve fibration by an action of a finite group $G$. This is morally just a consequence of Theorem~\ref{main-main} (or really its proof); the basic idea is that we just pass from each model $\X_i'$ in the formulation of Theorem~\ref{main-main} to a model $\X'_i/G_i$. Existence of such quotients is systematically worked out in \cite{Z2}. We now briefly review  the definition of such quotients and the necessary existence results:

\begin{defn}\label{geom-quot} Let $G$ be a finite group, $S$ a ringed space (resp. topologically ringed space), and $X$ an $S$-ringed (resp. topologically ringed) spaces with a right $S$-action of a finite group $G$. The {\it geometric quotient} $X/G=(|X/G|, \O_{X/G}, h)$ consists of:
\begin{itemize}\itemsep0.5em
\item the topological space $|X/G|\coloneqq |X|/G$ with the quotient topology. We denote by $\pi:|X| \to |X/G|$ the natural projection,
\item the sheaf of rings (resp. topological rings) $\O_{X/G}\coloneqq (\pi_*\O_X)^G$,
\item the morphism $h:X/G \to S$ defined by the pair $(h, h^{\#})$, where $h:|X|/G \to S$ is the unique morphism induced by $f\colon X \to S$, and $h^{\#}$ is the natural morphism 
\[
\O_{S} \to h_*\left(\O_{X/G}\right)=h_*\left(\left(\pi_*\O_{X}\right)^G\right)=\left(h_*\left(\pi_*\O_{X}\right)\right)^G=\left(f_*\O_X\right)^G
\]
that comes from $G$-invariance of $f$.
\end{itemize}
\end{defn}

Now suppose that $\O_K$ is a complete rank-$1$ valuation ring. Then \cite[Theorem 2.2.6]{Z2} ensures that the geometric quotient of a quasi-projective, flat, locally finite type $\O_K$-scheme $X$ exists as an admissible $\O_K$-scheme and the quotient map $X \to X/G$ is the universal $G$-invariant map in the category of $\O_K$-schemes. Similarly, \cite[Theorem 3.3.4]{Z2} guarantees that the geometric quotient of a formally projective, admissible formal $\O_K$-scheme exists as an admissible formal $\O_K$-scheme with the same universal property.


We now define precisely what we mean a ``uniformization''. Let $\X$ be a locally topologically finite type formal $\O_K$-scheme, and let $\varphi_i\colon \X_i \to \X$ be a finite set of locally topologically finite type $\O_K$-morphisms. 

\begin{defn}\label{comp-formal} We say that a set $(\X_i, \varphi_i)_{i\in I}$ can be obtained as a {\it composition of open Zariski coverings and rig-isomorphisms of $\X$}, if this set can be achieved in a finite number of steps using the following rules:  we start with the set consisting of one morphism $(\X, \operatorname{Id}_\X)$ and at each step we are allowed either to change one element $\varphi_i\colon \X_i \to \X$ by composing with a rig-isomorphism (in the sense of Definition~\ref{C-mod-formal}) $\varphi'_i\colon \X'_i \to \X$  or to replace $\varphi_i\colon X_i \to X$ (and keep other elements the same) by the compositions $\X_{i,j} \to \X_i \xr{\phi_i} \X$ with $\X_i=\cup_{j=1}^n \X_{i,j}$ a Zariski cover.
\end{defn}

It also has its schematic counterpart that will play an intermediate step in our proof of the uniformization result:

\begin{defn}\label{comp-alg} Let $\varphi_i\colon X_i \to X$ be a finite set of morphisms between $\O_K$-schemes. We say that a set $(X_i, \varphi_i)_{i\in I}$ can be obtained as a {\it composition of open Zariski coverings and $K$-modifications of $X$}, if this set can be achieved in a finite number of steps using the following rules:  we start with the set consisting of one morphism $(X, \operatorname{Id}_X)$ and at each step we are allowed either to change one element $\varphi_i\colon X_i \to X$ by composing with a $K$-modification (in the sense of Definition~\ref{U-mod}) $\varphi'_i\colon X'_i \to X$ or to replace $\varphi_i\colon \X_i \to \X$ (and keep other elements the same) by the compositions $X_{i,j} \to X_i \xr{\phi_i} X$ with $X_i=\cup_{j=1}^n X_{i,j}$ a Zariski cover.
 \end{defn}

\begin{lemma}\label{al-an} Let $(X_i, \varphi_i)_{i\in I}$ be a finite set that can be obtained from $X$ as a composition of open Zariski coverings and $K$-modifications. The set of $I$-adic completions $(\wdh{X_i}, \wdh{\varphi_i})_{i \in I}$ is obtained from $\X$ as a composition of open Zariski coverings and rig-isomorphisms of $\wdh{X}$.
\end{lemma}
\begin{proof}
This follows directly from the fact that the completion of a Zariski covering is a Zariski covering, and the completion of a $K$-modification is a rig-isomorphism (Lemma~\ref{different-C-mod}).
\end{proof}

\begin{thm}\label{thm:final} Let $\X$ be an admissible, quasi-compact and quasi-separated formal $\O_K$-scheme with the smooth generic fiber $\X_K$. Then there is a finite Galois extension $K\subset K'$, and a finite extension $K' \subset K''$, a finite set $(\X_i, \varphi_i)_{i\in I}$ of quasi-compact, quasi-separeted admissible formal $\O_{K'}$-schemes with morphisms $\varphi_i\colon \X_i \to \X_{\O_{K'}}$ such that
\begin{itemize}\itemsep0.5em
\item The set $(\X_i, \varphi_i)$ can be obtained from $\X_{\O_{K'}}$ as a composition of open Zariski coverings and rig-isomorphisms.
\item Each $\X_i$ is a geometric quotient of an admissible formal $\O_{K'}$-scheme $\X'_i$ by an action of a finite group $G_i$ such that the generic fiber of the quotient map $p_{i,K'}\colon \X'_{i,K'} \to \X_{i,K'}$ is a $G_i$-torsor.
\item Each of $\X'_i$ has a structure of a semi-stable, rig-smooth successive curve fibration over $\Spf \O_{K'}$.
\item Each $\X'_{i, \O_{K''}}$ admits a rig-isomorphism $h_i\colon \X''_i \to \X'_{i, \O_{K''}}$ with a rig-smooth, polystable formal $\O_{K''}$-scheme $\X''_i$.
\item If $K$ is discretely valued, one can choose $\X''_i$ to be strictly semistable over $K''$.
\end{itemize}
\end{thm}
\begin{proof}
The construction is very easy, we use Theorem~\ref{main-main} to get a finite Galois extension $K\subset K'$, finite extension $K' \subset K''$ (that can be chosen to be Galois if $K$ is perfect), and a finite set of morphisms $g_i\colon \X'_i \to \X_{\O_{K'}}$, $h_i\colon \X''_i \to \X'_{i, \O_{K''}}$ with all the properties from the Theorem~\ref{main-main}. Formal quasi-projectivity of $\X'_i$ over $\Spf \O_{K'}$ shows that we can use \cite[Theorem 3.3.4]{Z2} to see that the geometric quotient $p_i\colon \X'_i \to \X'_i/G_i$ exists as admissible formal $\O_{K'}$-scheme for any $i\in I$. We define \[\X_i\coloneqq \X'_i/G_i\] as the geometric quotient of $\X'_i$ by an action of $G_i$, this naturally comes equipped with a map $\varphi_i\colon \X_i \to \X_{\O_{K'}}$.

Recall that \cite[Proposition 5.2.1]{Z2} implies that each $\X_i$ is an admissible, quasi-compact and quasi-separated formal $\O_{K'}$-scheme. Moreover, \cite[Theorem 4.4.1]{Z2} shows that the geometric quotient of the generic fibre $\X'_{i,K'}/G_i$ exists as adic space topologically finite type over $\Spa(K', \O_{K'})$, and we have a functorial isomorphism $\X'_{i,K'}/G_i \cong (\X'_i/G_i)_{K'}$. Theorem~\ref{main-main} says that the generic fibre $g_{i,K'}\colon \X'_{i,K'} \to \X_{K'}$ induces a $G$-torsor over some open $U_i\subset \X_{K'}$. Therefore 
\[
(\X'_i/G_i)_{K'} \cong \X'_{i,K'}/G_i \cong U_i.
\]
Thus the $U_i$-restriction of the map $p_{i,K'}\colon  \X'_{i,K'} \to \X_{i,K'}$ is a $G_i$-torsor as well. Thus, the only thing we really need to show is that $(\X_i, \varphi_i)_{i\in I}$ is obtained from $\X_{\O_{K'}}$ as a composition of open Zariski coverings and rig-isomorphisms. We do not know how to see this from the first principles, the only way to prove it that we are aware of is to use the explicit construction of each $\X'_i$, as follows. \smallskip

We briefly remind the reader the proof of Theorem~\ref{main-main}. The first step was to reduce the situation to the affine case, then use Elkik's algebraization theorem to reduce to the schematic case. Since passing to affine coverings does no harm for our purpose (proving that $\X_i$ can be obtained as a composition of open Zariski coverings and $C$-modifications) we can assume from the beginning  that $\X$ is affine. Moreover, Lemma~\ref{al-an} and Theorem \cite[Theorem 3.4.1]{Z2} show that it suffices to prove our claim in the schematic case. We change our notations here and denote an algebraization of $\X$ by $X$\footnote{Since we will never use the adic generic fiber of $\X$ in the rest of the proof, the possible confusion of $X$ with the generic fiber $\X_K$ will never occur.}.

The next step in the proof of Theorem~\ref{main-main} was to reduce the affine schematic version to the situation where $X_{K}$ (is affine and) has a structure of a successive $K$-smooth curve fibration. This was done by means of projective $K$-modifications and Zariski open coverings, they also cause no harm for our purpose. So we can assume that $X$ is an affine flat finitely presented $\O_K$-scheme with a structure of a successive $K$-smooth curve fibration
\[
X=X_n \xr{f_n} X_{n-1} \xr{f_{n-1}} \dots \xr{f_2} X_1 \xr{f_1} X_0=\Spec \O_K
\]
such that each morphism $f_i$ is quasi-projective. Then Lemma~\ref{main-lemma-4}(\ref{main-lemma-4-5}) finishes the proof.
\end{proof}

\section{Proof of Lemma~\ref{main-lemma}}\label{pizdec}

We formulate here a more precise version of Lemma~\ref{main-lemma} and give a detailed proof. Even though the idea of the proof is rather easy, it turns out that it requires some patience to give a rigorous proof.

\begin{prop}\label{main-lemma-2} Let $R$ be a valuation ring with fraction field $K$. Let $S$ be the affine scheme $\Spec R$. Suppose that $f\colon X \to S$ is a morphism that can be written as a composition 
\[
X=X_n \xr{f_n} X_{n-1} \xr{f_{n-1}}X_{n-2} \xr{f_{n-2}} \dots \xr{f_{2}}X_1 \xr{f_1}X_0=S
\]
such that each morphism $f_i$ is a relative curve (in the sense of Definition~\ref{def-curve}) that is $K$-smooth. There is a finite set $J$, a set of finite groups $\left(G_{j}\right)_{j\in J}$ indexed by $J$, and commutative diagrams: 

\[
\begin{tikzcd}[column sep=13ex]
X_{n} \arrow{d}{f_n}  & X_{n,j}  \arrow{l}{g_{n,j}} \arrow{d}{h_{n,j}}\\
X_{n-1} \arrow{d}{f_{n-1}} & X_{n-1, j} \arrow{l}{g_{n-1,j}} \arrow{d}{h_{n-1, j}}\\
\vdots \arrow{d}{f_2} & \vdots \arrow{l}{g_{k,j}} \arrow{d}{h_{2, j}}\\
X_1\arrow{d}{f_1} & X_{1, j} \arrow{l}{g_{1,j}} \\
\Spec R &  
\end{tikzcd}
\]
of flat, finitely presented $S\coloneqq \Spec R$-schemes such that 
\begin{enumerate}
\item $X_{k, j}$ admits an $R$-action of the group $G_j$ for any $k\geq 1$ and $j\in J$. Moreover, $g_{k,j}$ is $G_j$-invariant and $h_{k,j}$ is $G_j$-equivariant for any $j\in J, k=1, \dots, n$.
\item The $K$-restriction $(g_{k, j})_{K}\colon (X_{k, j})_{K} \to (X_{k})_{K}$ is a $G_j$-torsor over its (open) image $V_{k,j}$.
\item $h_{k,j}$ is a semi-stable $K$-smooth relative curve for any $j\in J$ and any $k\geq 2$ 
\item $g_{k,j}$ is quasi-projective for any $k\geq 1$ and $j\in J$. Moreover, the map $g_k\colon \sqcup_{j\in J}X_{k,j} \to X_{k}$ is a $K$-\'etale covering\footnote{In particular, each $g_{k,j}$ is finitely presented, see Definition~\ref{U-etale}.} for any $k\geq 1$.
\end{enumerate}
\end{prop}
\begin{proof}

We prove the claim by descending induction on $m\leq n$. Namely, we show that for any $n\geq m \geq 1$ there is a finite set $J_m$, a set of finite groups $\left(G^{(m)}_{j}\right)_{j\in J_m}$ indexed by $J_m$, and commutative diagrams: 
\[
\begin{tikzcd}[column sep=13ex]
X_{n} \arrow{d}{f_n}  & X_{n,j}^{(m)}  \arrow{l}{g_{n,j}^{(m)}} \arrow{d}{h_{n,j}^{(m)}}\\
X_{n-1} \arrow{d}{f_{n-1}} & X_{n-1, j}^{(m)} \arrow{l}{g_{n-1,j}^{(m)}} \arrow{d}{h_{n-1, j}^{(m)}}\\
\vdots \arrow{d}{f_{m+1}} & \vdots \arrow{l}{g_{k+1,j}^{(m)}} \arrow{d}{h_{m+1, j}^{(m)}}\\
X_m &  \arrow{l}{g_{m,j}^{(m)}} X_{m, j}^{(m)}
\end{tikzcd}
\]
of flat, finitely presented $S$-schemes such that 
\begin{itemize}\itemsep0.5em

\item $X^{(m)}_{k, j}$ admits an $R$-action of the group $G^{(m)}_j$ for any $k\geq m$ and $j\in J_m$. Moreover, $g^{(m)}_{k,j}$ is $G^{(m)}_j$-invariant and $h^{(m)}_{k,j}$ is $G^{(m)}_j$-equivariant for any $j\in J_m, k\geq m$.
\item The $K$-restriction $(g^{(m)}_{k, j})_{K}\colon (X^{(m)}_{k, j})_{K} \to (X^{(m)}_{k})_{K}$ is a $G^{(m)}_j$-torsor over its (open) image $V^{(m)}_{k,j}$.
\item $h^{(m)}_{k,j}$ is a semi-stable $K$-smooth relative curve for any $j\in J_m$ and any $k> m$. 
\item $g^{(m)}_{k,j}$ is quasi-projective for any $k\geq m$ and $j\in J_m$. Moreover, the map $g^{(m)}_k\colon \sqcup_{j\in J_m}X^{(m)}_{k,j} \to X_{k}$ is a $K$-\'etale covering for any $k\geq m$.
\end{itemize} 

\begin{rmk} We use the superscript $^{(m)}$ to denote objects that come from the induction hypothesis. We do it to emphasize that the induction argument does change the objects constructed at the previous step. More precisely, given a tower $(X^{(m)}_{k,j})_{k\geq m}$ the resulting tower ``$(X^{(m-1)}_{k, j})_{k\geq m-1}$'' will not be an extension of $(X^{(m)}_{k,j})_{k\geq m}$ one layer lower, but rather some other tower built out of $(X^{(m)}_{k,j})_{k\geq m}$ using Theorem~\ref{stable-modification-general}. In particular, the induction step does ``enlarge'' the sets $J_m$ and groups $G_j^{(m)}$, but it shrinks the opens $V^{(m)}_{k,j}$.
\end{rmk}

The statement is trivial for $m=n$ (we can take a trivial group and the identity morphism $X_n \to X_n$). In particular, the case $n=1$ is settled. Now suppose that $n\geq 2$ and we proved the claim for some $m\geq 2$,  we deduce the claim for $m-1\geq 1$ from that. We divide the proof into several steps for the convenience of the reader. \smallskip 

{\it Step 1}: We first use the induction hypothesis to find a commutative diagram
\[
\begin{tikzcd}[column sep=13ex]
X_{n} \arrow{d}{f_n}  & X_{n,j}^{(m)}  \arrow{l}{g_{n,j}^{(m)}} \arrow{d}{h_{n,j}^{(m)}}\\
X_{n-1} \arrow{d}{f_{n-1}} & X_{n-1, j}^{(m)} \arrow{l}{g_{n-1,j}^{(m)}} \arrow{d}{h_{n-1, j}^{(m)}}\\
\vdots \arrow{d}{f_{m+1}} & \vdots \arrow{l}{g_{k,j}^{(m)}} \arrow{d}{h_{m+1, j}^{(m)}}\\
X_m \arrow{d}{f_m} &  \arrow{l}{g_{m,j}^{(m)}} X_{m, j}^{(m)} \\
X_{m-1} & 
\end{tikzcd}
\]
with all the desired properties of $g_{k, j}^{(m)}$ and $h_{k, j}^{(m)}$. We now work separately with each $j\in J_m$, so we fix one such $j$ for the rest of the induction argument. \\

{\it We ``modify'' $f_m$}. Consider the finitely presented morphism $\a_{m, j}\colon X_{m,j}^{(m)} \to X_{m-1}$ obtained as the composition  $f_m \circ g^{(m)}_{m,j}$, this morphism is not necessarilly flat, but $(\a_{m, j})_K$ is a smooth relative curve. So \cite[Corollaire 5.7.10]{RG} guarantees that there is an $K$-admissible blow-up $r_j\colon X'_{m-1, j} \to X_{m-1}$ such that the strict transform of $\a_{m, j}$ along the map $r_j$ is an $S$-morphism $\a'_{m, j}\colon X'_{m,j} \to X'_{m-1, j}$ that is flat of pure relative dimension $1$ and $K$-smooth. We note that $X'_{m-1, j}$ is flat and finitely presented over $S$ by Lemma~\ref{admissible} (so $\a'_{m,j}$ is finitely presented), and the map $r_j$ is a projective $K$-modification by Lemma~\ref{cofinal}. As a strict transform of a $K$-admissible blow-up is a $K$-admissible blow-up, we conclude that $r_{m, j}\colon X'_{m, j}\to X_{m, j}^{(m)}$ is a finitely presented, projective $K$-modification ($r_{m,j}$ is finitely presented since it is an $S$-map between finitely presented $S$-schemes). \smallskip

We note that each $X_{m,j}^{(m)}$ admits a $G_{j}^{(m)}$-action over $X'_{m-1, j}$. Indeed, the morphism $\a_{m, j}$ is $G_j^{(m)}$-invariant for any $j\in J_m$ since $g_{m, j}^{(m)}$ is. Therefore, the strict transform $X'_{m, j}$ admits a $G_j^{(m)}$-action such that $\a'_{m, j}$ is $G_j^{(m)}$-invariant and $r_{m ,j}$ is $G_j^{(m)}$-equivariant for all $j\in J_m$. \\

{\it We now ``modify'' $f_k$ for $n\geq k\geq m+1$\footnote{This case does not arise if $m=n$, but this case was already considered above.}}. We consider the morphisms 
\[
\a_{k, j}\colon X^{(m)}_{k, j} \to X_{m-1} \text{ defined as } f_m \circ g^{(m)}_{m,j} \circ h^{(m)}_{m+1, j}\circ \dots \circ h^{(m)}_{k, j}
\]
and their strict transforms along the map $r_j\colon X'_{m-1, j} \to X_{m-1}$, denoted as $\a_{k, j}'\colon X'_{k, j} \to X'_{m-1, j}$. We need to note two things. Firstly, we recall that strict transform is defined as the schematic closure of the $K$-fiber 
\[
(X_{k,j}^{(m)}\times_{X_{m-1}} X'_{m-1, j})_K
\]
in the total fibre product $X_{k,j}^{(m)}\times_{X_{m-1}} X'_{m-1, j}$, so the action of $G^{(m)}_{j}$ on $X^{(m)}_{k, j}$ defines the evident action on $X'_{k, j}$. Moreover, this shows that $X'_{k, j}$ is flat and finite type over $R$, so it is finitely presented by Lemma~\ref{admissible}. Secondly, we note that we can consider $X'_{k, j}$ as the strict transform of $X^{(m)}_{k, j}$ along the $K$-admissible blow-up $X'_{m,j} \to X_{m,j}^{(m)}$. Thus all the $S$-maps $X'_{k,j} \to X_{k,j}^{(m)}$ between finitely presented $S$-schemes are $K$-admissible blow-ups; in particular, they are projective, finitely presented $K$-modifications. Moreover, by the assumption all the morphisms $X^{(m)}_{k, j} \to X^{(m)}_{m, j}$ are successive semi-stable curve fibrations, so they are flat. \smallskip

Therefore, we have natural isomorphisms 
\[
X'_{k, j} \cong X^{(m)}_{k, j} \times_{X^{(m)}_{m, j}} X'_{m, j}
\]
for any $k\geq m$. This implies that all the morphisms 
\[
h'_{k, j}\colon X'_{k, j} \to X'_{k-1, j}
\] 
arising from $h_{k,j}^{(m)}$ are semi-stable curve fibrations. As the fiber product, each $X'_{k, j}$ admits an action of $G_{j}^{(m)}$ such that $r_{k, j}$ and $h'_{k, j}$ are $G_j^{(m)}$-equivariant. In particular, the morphisms $g'_{k,j}\coloneqq g^{(m)}_{k,j} \circ r_{k,j}$ are $G^{(m)}_j$-invariant. 
\\




To sum up, we obtain the following commutative diagram: 

 \[
\begin{tikzcd}[column sep=13ex, row sep=large]
X_n \arrow{d}{f_n}  & X^{(m)}_{n,j}  \arrow{l}{g^{(m)}_{n,j}} \arrow{d}{h^{(m)}_{n,j}} &
 X'_{n,j} \arrow{l}{r_{n,j}} \arrow{d}{h'_{n,j}}\\
X_{n-1} \arrow{d}{f_{n-1}} & X^{(m)}_{n-1, j} \arrow{l}{g^{(m)}_{n-1,j}} \arrow{d}{h^{(m)}_{n-1, j}} & 
X'_{n-1, j} \arrow{l}{r_{n-1, j}} \arrow{d}{h'_{n-1, j}}\\
\vdots \arrow{d}{f_{m+1}} & \vdots \arrow{l}{g^{(m)}_{k,j}} \arrow{d}{h^{(m)}_{m+1, j}} & 
\vdots \arrow{l}{r_{k, j}} \arrow{d}{h'_{m+1, j}} \\
X_m \arrow[rd, swap, "f_{m}"] & X^{(m)}_{m,j} \arrow{l}{g^{(m)}_{m,j}}\arrow{d}{\a_{m,j}} &
X'_{m,j} \arrow{l}{r_{m, j}} \arrow{d}{\a'_{m,j}}\\
& X_{m-1} & X'_{m-1, j}\arrow{l}{r_j}
\end{tikzcd}
\]

Here are the relevant properties of this construction:  
\begin{itemize}\itemsep0.5em

\item All the schemes in this diagram are flat and finitely presented over $S$.
\item $X'_{k,j}$ admits an action of $G^{(m)}_{j}$ for any $k\geq m$. All the ``horizontal'' morphisms $r_{k,j}$ are $G^{(m)}_{j}$-equivariant. In particular, the morphisms $g'_{k,j}= g^{(m)}_{k,j} \circ r_{k,j}$ are $G^{(m)}_j$-invariant. 
\item for all $k\geq m+1$, $h'_{k, j}$ is $G_{j}^{(m)}$-equivariant and defines a semi-stable $K$-smooth curve fibration.
\item The $K$-restriction $\left(g'_{k, j}\right)_K\colon  \left(X'_{k,j}\right)_K \to \left(X_{k}\right)_K$ is a $G^{(m)}_{j}$-torsor over its image $V^{(m)}_{k, j}$ for any $k\geq m$. That follows from the fact that $r_{k,j}$ is a $K$-modification for any $k\geq m$.
\item For each $k\geq m$, the morphism 
\[
\bigsqcup_{j\in J_m} g'_{k, j}\colon \bigsqcup_{j \in J_m} X^{'}_{k, j} \to X_k
\] 
is a quasi-projective $K$-\'etale covering (it is a composition of a $K$-\'etale covering and an $K$-admissible blow-up) for all $k\geq m$.
\item The morphism $\a'_{m, j} \colon X'_{m, j} \to X'_{m-1, j}$ is a relative curve fibration that is also $G^{(m)}_{j}$-invariant. 
\end{itemize}  
\medskip

{\it Step 2: We ``resolve'' each $\a_{m,j}'$.} We note that the $K$-restriction of $\a'_{m,j}$ is naturally identified with the $K$-restriction of $\a_{m, j}$ that is, in turn, the composition of $g^{(m)}_{m,j}$ and $f_m$. By assumption $g^{(m)}_{m,j}$ is $K$-\'etale and $f_{m}$ is $K$-smooth, so $\a'_{m, j}$ is $K$-smooth as well. So we can apply Theorem~\ref{stable-modification-general} (together with Remark~\ref{stable-modification-general-U-covering}) to $\a'_{m, j}$ in the role of $f$ there to obtain a non-empty finite set $I$, a set of finite groups $\left(H_i\right)_{i\in I}$ indexed by $I$, and $H_i$-invariant, quasi-projective $K$-\'etale morphisms
\[
\psi_{m-1, j, i}\colon X_{m-1, j, i} \to X'_{m-1, j}\footnote{$X_{m-1, j, i}$ plays the role of $W'_i$ in Theorem~\ref{stable-modification-general} and we suppress $V'_i$ from there.}
\]
and $X_{m-1,j,i}$-stable (projective) $K$-modifications 
\[
\beta_{m, j, i}\colon X_{m, j, i} \to X'_{m, j, i} \coloneqq X'_{m,j}\times_{X'_{m-1, j}}X_{m-1, j, i}
\]
such that the diagram
\begin{equation}\label{eqn:1}
	\begin{tikzcd}[column sep=13ex, row sep=8ex]
	X^{(m)}_{m, j} \arrow{d}{\a_{m, j}}  & X'_{m, j}  \arrow{l}{r_{m, j}} \arrow{d}{\a'_{m,j}} \drar[phantom, "\boxed{}"]  &   X'_{m, j, i} \arrow{l}{\psi_{m, j, i}}\arrow{d} & X_{m,j, i} \arrow{l}{\beta_{m, j, i}} \arrow{ld}{\text{semi-stable curve}}\\
	X_{m-1} &  X'_{m-1, j} \arrow{l}{r_j} & X_{m-1, j, i} \arrow{l}{\psi_{m-1, j, i}}
	\end{tikzcd}
\end{equation}
is commutative and satisfies the following properties:
\begin{itemize}\itemsep0.5em
\item For any $i\in I, j\in J_m$, $\psi_{m-1, j, i}$ is a composition $c\circ j \circ a$, where $a$ is a projective $K$-modification, $j$ an open immersion, and $c$ a finite, finitely presented, faithfully flat $H_i$-invariant morphism, whose $K$-restriction becomes an $H_i$-torsor.  \smallskip

Moreover, the map  $\psi_{m-1, j}\colon \bigsqcup_{i\in I} X_{m-1, j, i} \to X'_{m-1, j}$ is a quasi-projective $K$-\'etale covering, so the natural morphism \[
\psi_{m-1} \colon \bigsqcup_{j\in J_m, i\in I} X_{m-1, j, i} \to X_{m-1}
\] is also a quasi-projective $K$-\'etale covering.

\item The $K$-restriction $\left(\psi_{m-1, j, i}\right)_K \colon  \left(X_{m-1, j, i}\right)_K \to \left(X'_{m-1, j}\right)_K$ is an $H_{i}$-torsor over its (open) image. 
\end{itemize} 
We note that all schemes in this diagram are flat and finite type over $S$, so they are finitely presented by Lemma~\ref{admissible}.  \smallskip

Now observe that $X'_{m, j, i}$ admits an action of $G_j^{(m)}\times H_i$ via base change (since $\a'_{m, j}$ is $G_j^{(m)}$-invariant and $\psi_{m-1, j, i}$ is $H_i$-invariant), so the morphism $\psi_{m, j, i}$ is $G_j^{(m)}$-equivariant and the morphism $X'_{m, j,i} \to X_{m-1, j, i}$ is $H_i$-equivariant (see the Cartesian square in (\ref{eqn:1})). We now apply Lemma~\ref{U-torsor} to the morphisms 
\[
\begin{tikzcd}[column sep=13ex]
(X_m)_K & \left(X'_{m, j}\right)_K \arrow{l}{\left(g'_{m, j}\right)_K} \arrow{d}{\left(\a'_{m, j}\right)_K} & \left(X'_{m, j, i}\right)_K\arrow{l}{\left(\psi_{m ,j, i}\right)_K} \arrow{d} \\
& \left(X'_{m-1, j}\right)_K & \left(X_{k-1,j_{k-1}}\right)_K \arrow{l}{\left(\psi_{m-1, j, i}\right)_K}
\end{tikzcd}
\]
to obtain that $\left(X'_{m, j, i}\right)_K \to \left(X_m\right)_K$ is a $G^{(m)}_j\times H_i$-torsor over its open image $V_{m, j, i}$. Also, note that $\bigsqcup_{i\in I} X'_{m, j, i} \to X^{(m)}_{m, j}$ is a quasi-projective $K$-\'etale covering as a ``base change over $K$'' of a quasi-projective $K$-\'etale covering. Similarly, projectivity of the morphisms $\beta_{m, j, i}$ implies that the map $\bigsqcup_{i\in I} X_{m, j, i} \to X^{(m)}_{m, j}$ is a quasi-projective $K$-\'etale covering for any $j\in J_m$.\smallskip


{\it Now we do the key step: we want to lift the action of $G^{(m)}_j\times H_i$ on $X'_{m,j, i}$ to an action on $X_{m, j, i}$}. A priori, there is no reason we should be able to lift this action. We solve this problem by {\it changing} $X_{m, j,i}$.  Temkin proved the remarkable uniqueness statement \cite[Theorem 1.2]{T1} that over an irreducible {\it normal} base with generic point $\eta$, a {\it stable} $k(\eta)$-modification of a relative curve is unique. So the idea is to replace $X_{m-1, j, i}$ with the normalization $\widetilde{X}_{m, j, i}$ inside its generic fiber. However, there are two possible issues: the normalization may not be finitely presented over $S$, and it may not be integral. We resolve the first problem by an approximation argument, and we resolve the second problem by showing that $\widetilde{X}_{m-1, j, i}$ is a disjoint union of open integral, normal subschemes.\smallskip

We note that the morphism $\widetilde{X}_{m-1, j, i} \to X_{m-1, j, i}$ an isomorphism over $K$, $\widetilde{X}_{m-1, j, i}$ is $S$-flat, and the action of $H_i$ on $X_{m-1, j, i}$ lifts to an action on $\widetilde{X}_{m-1, j, i}$. Therefore, the going-down lemma \cite[Theorem 9.5]{M1} applied to an affine open covering of $\widetilde{X}_{m-1, j, i}$ implies that all generic points of $\widetilde{X}_{m-1, j, i}$ lie inside 
\[
\left(\widetilde{X}_{m-1, j, i}\right)_K \simeq \left(X_{m-1, j, i}\right)_K
\]
that is $K$-smooth and noetherian. In particular, $\widetilde{X}_{m-1, j, i}$ has finite number of irreducible components and normal. Therefore,  \cite[\href{https://stacks.math.columbia.edu/tag/0357}{Tag 0357}]{stacks-project} implies that $\widetilde{X}_{m-1, j, i}$ is a disjoint union of a finite number of normal integral (open) subschemes. Thus, we can apply \cite[Corollary 1.3]{T1} to all irreducible components of $\widetilde{X}_{m-1, j, i}$ to get that 
\[
\widetilde{X}_{m, j, i}\coloneqq X_{m, j, i}\times_{X_{m-1, j, i}} \widetilde{X}_{m-1, j, i}
\]
is the unique {\it stable} modification of the relative curve
\[
\widetilde{X'}_{m, j, i}\coloneqq X'_{m, j, i}\times_{X_{m-1, j, i}} \widetilde{X}_{m-1, j, i} \  (= X'_{m, j}\times_{X'_{m-1, j}} \widetilde{X}_{m-1, j, i}). 
\]
This allows us to lift the action of $G^{(m)}_j\times H_i$ on $\widetilde{X'}_{m, j, i}$ to an action on $\widetilde{X}_{m, j, i}$ for free\footnote{We define action of $G^{(m)}_j\times H_i$ on $\widetilde{X'}_{m, j, i}$ via base change.}. Namely, the uniqueness result allows us to lift the action of each $g\in G^{(m)}_j\times H_i$ separately, and $S$-flatness of $\widetilde{X}_{m, j, i}$ and separatedness of $\widetilde{\beta}_{m, j, i}\colon \widetilde{X}_{m, j, i} \to \widetilde{X'}_{m, j, i}$ guarantee that this procedure defines an action of $G^{(m)}_j\times H_i$ and that $\widetilde{\beta}_{m, j, i}$ is automatically $G^{(m)}_j\times H_i$-equivariant. Clearly, its $K$-fiber is a torsor over the same 
\[
V_{m, j, i} \subset (X_{m-1, j, i})_K \simeq (\widetilde{X}_{m-1, j, i})_K.
\]

Now we point out that in this non-noetherian situation, the normalization maps may not be finite, so $\widetilde{X}_{m-1 ,j, i}$ (and hence $\widetilde{X}_{m, j, i}$) may not be finite type over $S$. But we note that the morphism $\widetilde{X}_{m, j, i} \to X_{m, j, i}$ is at least integral, so it corresponds to a quasi-coherent $\O_{X_{m, j, i}}$-algebra $\mathcal{A}$ integral over $\O_{X_{m, j, i}}$. \cite[I, 6.9.15]{EGA} gives that $\mathcal{A}$ is a filtered colimit of {\it finite} quasi-coherent $\O_{X_{m, j, i}}$-subalgebras $\mathcal{A}_i$. Note that all these $\O_{X_{m, j, i}}$-algebras are $S$-flat as they are torsion-free, so Lemma~\ref{admissible} implies that they are finite, finitely presented. In other words, $\widetilde{X}_{m, j, i}$ is a filtered colimit of finite, finitely presented, $K$-modifications $X_{m ,j, i}^{\lambda}$ of $X_{m, j, i}$ with affine transition morphisms. Therefore, a standard approximation argument shows that the action $G^{(m)}_j\times H_i$ on $X'_{m ,j,i}$  over $S$ lifts to an action on some $X_{m ,j, i}^{\lambda}$ that we may rename as $X_{m, j, i}$ for our purposes. \smallskip

{\it Step 3: The induction step. We extend our tower one layer lower.} We roughly perform the base change of the initial tower $X^{(m)}_{k, j}$ along the maps $X_{m, j, i} \to X^{(m)}_{m, j}$ to get the desired result. Let us do everything carefully. We consider the following commutative diagram:

\[
\begin{tikzcd}[column sep=12ex, row sep=8ex]
X_n \arrow{d}{f_n}    & X^{(m)}_{n,j}    \arrow{l}{} \arrow{d}{h^{(m)}_{n,j}} &   
 X'_{n, j} \arrow{l}{r_{n,j}} \arrow{d}{h'_{n,j}} &    X'_{n,j, i} \arrow{l}{\psi_{n, j, i}} \arrow{d}{h'_{n, j, i}}  & X_{n,j, i} \arrow{l}{} \arrow{d}{h_{n, j, i}}\\ 
 X_{n-1} \arrow{d}{f_{n-1}}    &    X^{(m)}_{n-1, j} \arrow{l}{} \arrow{d}{h^{(m)}_{n-1, j}} &    
 X'_{n-1, j} \arrow{l}{r_{n-1, j}} \arrow{d}{h'_{n-1, j}}&     X'_{n-1,j, i} \arrow{l}{\psi_{n-1, j, i}} \arrow{d}{h'_{{n-1}, j, i}}  &  X_{n-1, j, i} \arrow{l}{} \arrow{d}{h_{{n-1},j, i}}\\
\vdots \arrow{d}{f_{m+1}}   & \vdots \arrow{l}{} \arrow{d}{h^{(m)}_{m+1, j}}    & 
\vdots \arrow{l}{r_{k, j}} \arrow{d}{h'_{m+1, j}}    & \vdots \arrow{l}{\psi_{k, j, i}}\arrow{d}{h'_{m+1, j, i}}   & \vdots \arrow{l}{}\arrow{d}{h_{m+1, j, i}}\\
X_m \arrow{rd}{f_{m}}&   X^{(m)}_{m, j} \drar[phantom, "\star"] \arrow{l} \arrow{d}{\a_{m, j}}  &    X'_{m, j}  \arrow{l}{r_{m, j}} \arrow{d}{\a'_{m,j}} & X'_{m, j, i} \arrow{l}{\psi_{m, j, i}}\arrow{d} & X_{m, j, i} \arrow{l}{\beta_{m, j, i}} \arrow[ld, swap, "h_{m, j,i }"] \arrow{ld}{\text{semi-stable curve}}\\
& X_{m-1} & X'_{m-1, j} \arrow{l}{r_j} & X_{m-1, j, i} \arrow{l}{\psi_{m-1, j, i}}
\end{tikzcd}
\]
where all the squares except for the left column and the ``starred'' square are Cartesian\footnote{The maps $h_{k, j, i}$ along the right side are defined as the base change of $h'_{k, j ,i}$ for $k>m$.}. We note all the schemes $X'_{k, j, i}$ and $X_{k, j, i}$ are flat, finitely presented schemes over $S$ by construction (we use here that all $f_k$ are flat). Moreover, 
we recall that all the morphisms $h'_{k,j}$ are $G^{(m)}_j$-equivariant, and the morphism $\psi_{m, j, i}$ is $G^{(m)}_j$-equivariant and $H_i$-invariant. This allows us to define via base change an action of $G^{(m)}_j\times H_i$ over $S$ on each $X'_{k, j, i}$ such that all the maps $h'_{k,j, i}$ are $G^{(m)}_j\times H_i$-equivariant and the maps $\psi_{k, j, i}$ are $G^{(m)}_j$-equivariant and $H_i$-invariant. \smallskip

The same construction defines an action of $G^{(m)}_j\times H_i$ over $S$ on each $X_{k, j, i}$ (using that we already have an action on $X_{m, j, i}$) so that all the maps $\beta_{k, j, i}$ and $h_{k, j, i}$ are $G^{(m)}_j\times H_i$-equivariant. As the map $X_{k, j, i} \to X'_{k, j}$ is the base change of the quasi-projective $K$-\'etale map $X_{m, j, i} \to X'_{m, j}$, we conclude that the morphisms $X_{k, j, i} \to X'_{k, j}$ are quasi-projective $K$-\'etale morphisms. And the map
\[
\bigsqcup_{j \in J_m, i\in I} X_{k, j, i} \to X_k
\]
is a quasi-projective $K$-\'etale covering for any $k\geq m-1$. We also note that $h_{k, j, i}$ is a semi-stable, $K$-smooth curve fibration for any $k\geq m$. Finally, we observe that, for $k\geq m+1$, the $K$-restriction $(X_{k, j, i})_K \to (X_{k})_K$ is a $G^{(m)}_j\times H_i$-torsor over its open image $V_{k, j, i}$ as it is a base change of the map $X_{m, j, i} \to X_m$ that shares the same property. \\

We are now ready to define the set $J_{m-1}$ and all other corresponding objects that we need in the inductive step for $m-1$. Namely, we define $J_{m-1}$ to be the direct product $J_m\times I$, and the set of finite groups $\left\{G^{(m-1)}_{j'}\right\}_{j'\in J_{m-1}}$ to be the set of products $\left\{G^{(m)}_j\times H_i\right\}_{j\in J_m, i\in I}$. We also define a flat, finitely presented $S$-scheme $X^{(m-1)}_{k, j'}\coloneqq X_{k, j, i}$ where $j'=\left(j, i\right)$ for any $k\geq m-1$ and $j'\in J_{m-1}$. The maps $g^{(m-1)}_{k, j'}\colon X^{(m-1)}_{k, j'} \to X_k$ are the maps $X_{k, j, i} \to X_k$ in the diagram above, and the maps $h^{(m-1)}_{k, j'}\colon X^{(m-1)}_{k, j'} \to X^{(m-1)}_{k-1, j'}$ are defined as corresponding maps $h_{k, j, i}$ for $j'=\left(j, i\right)$ and any $k\geq m$. The discussion above shows that the diagrams (indexed by $j'\in J_{m-1}$):
\[
\begin{tikzcd}[column sep=13ex]
X_{n} \arrow{d}{f_n}  & X_{n,j'}^{(m-1)}  \arrow{l}{g_{n,j'}^{(m-1)}} \arrow{d}{h_{n,j'}^{(m-1)}}\\
X_{n-1} \arrow{d}{f_{n-1}} & X_{n-1, j'}^{(m-1)} \arrow{l}{g_{n-1,j'}^{(m-1)}} \arrow{d}{h_{n-1, j'}^{(m-1)}}\\
\vdots \arrow{d}{f_{m+1}} & \vdots \arrow{l}{g_{k+1,j'}^{(m-1)}} \arrow{d}{h_{m+1, j'}^{(m-1)}}\\
X_m \arrow{d}{f_{m}}&  \arrow{l}{g_{m,j'}^{(m-1)}} X_{m, j'}^{(m-1)} \arrow{d}{h^{(m-1)}_{m, j'}} \\
X_{m-1} & \arrow{l}{g_{m-1,j'}^{(m-1)}} X_{m-1, j'}^{(m-1)}
\end{tikzcd}
\]
satisfy all the requirements. This finishes the induction argument. 
\end{proof}


\begin{cor}\label{main-lemma-3} Under the notation of Proposition~\ref{main-lemma-2}, suppose that each $f_i$ is quasi-projective (so each $X_{n, j}$ is quasi-projective over $R$, too). Then 
\begin{enumerate}
    \item the geometric quotient $X_{n,j}/G_j$ exists as a flat, finitely presented $S$-scheme for any $j\in J$.
    \item If we define $\ov{g_{n,j}}\colon X_{n, j}/G_j \to X_n$  as the map induced by $g_{n,j}$ for any $j\in J$. Then we can choose $J, G_j, X_{n,j}, \dots$ in a way that the set $\left(X_{n,j}/G_{j}, \ \ov{g_{n,j}}\right)_{j\in J}$ can be obtained from $X_{n}$ as a composition of Zariski open coverings and $K$-modifications (in the sense of Definition~\ref{comp-alg}).
\end{enumerate}
\end{cor}
\begin{proof}
Since $X_{n, j}$ is quasi-projective over $R$, we see by \cite[Theorem 2.2.6 and Proposition 5.1.1]{Z2}  that $X_{n ,j}/G_{j}$ exists as a flat and finitely presented $S$-scheme for any $j\in J$. \smallskip 

Now we show that the schemes $X_{n,j}$ constructed in the {\it proof} of Proposition~\ref{main-lemma-2} satisfy the property that  $\left(X_{n,j}/G_{j}, \ov{g_{n,j}}\right)_{j\in J}$ can be obtained from $X_{n}$ as a composition of Zariski open coverings and $K$-modifications. Throughout the proof we use the same notation as in the proof of Proposition~\ref{main-lemma-2}; we refer to Steps $2$ and $3$ therein for the construction of $X^{(m)}_{k, j}$. \smallskip 

We show the claim by descending induction on $n \geq m\geq 1$. Namely, we start the induction by declaring $J_n\coloneqq \{\varnothing \}, G_{\varnothing}\coloneqq \{e\}$ and the map $X^{(n)}_{n, \varnothing} \to X_n$ to be the identity map $X_n \to X_n$. This clearly satisfies the condition that $X^{(n)}_{n, \varnothing}/G_{\varnothing}$ can be obtained from $X_n$ as a composition of Zariski open coverings and $K$-modifications. The claim now is that the induction argument from the proof of Proposition~\ref{main-lemma-2} that builds up $\left\{X^{(m-1)}_{n, j'}\right\}$ out of $\left\{X^{(m)}_{n, j}\right\}$ preserves this property. We suppose that this property holds for some $n \geq m>1$, and show it for $m-1$. \\

The induction argument constructs the schemes $X^{(m-1)}_{n, j'}$ for $j'\in J_{m-1}=J_m\times I$ in a rather specific way. We are going to use this explicit construction. First of all, we note that each $X^{(m-1)}_{n, j'}$ is a quasi-projective, finitely presented, flat $S$-scheme. Thus, \cite[Theorem 2.2.6 and Proposition 5.1.1]{Z2} imply that $X^{(m-1)}_{n ,j'}/G^{(m-1)}_{j'}$ exists as a flat and finitely presented $S$-scheme. Moreover, the $G^{(m-1)}_{j'}$-invariant map $g_{n,j'}^{(m-1)}\colon X_{n,j'}^{(m-1)} \to X_n$ induces a map 
\[
\ov{g_{n,j'}^{(m-1)}}\colon X_{n,j'}^{(m-1)}/G^{(m-1)}_{j'} \to X_n
\]
We claim that that the set $\left(X_{n,j'}^{(m-1)}/G^{(m-1)}_{j'}, \ov{g_{n,j'}^{(m-1)}}\right)_{j'\in J_{m-1}}$ can be obtained from $X_n$ as a composition of $K$-modifications and open Zariski coverings (in the sense of Definition~\ref{comp-alg}). We take a closer look at the construction of this morphism in Steps $2$ and $3$. \smallskip

 Firstly, we recall that $G^{(m-1)}_j=G^{(m)}_j\times H_i$, and the morphism $\psi_{m-1, j, i}\colon X_{m-1, j, i} \to X'_{m-1, j}$ was obtained as a composition of a (projective) $K$-modification followed by an open immersion, followed by a finite, finitely presented $H_i$-invariant morphism that becomes an $H_i$-torsor on generic fibers\footnote{We note that last map was obtained as a composition of a faithfully flat, finitely presented morphism and a finitely presented approximation of the normalization in the generic fiber.}. So we decompose the morphism $\psi_{m-1, j, i}$ as follows:
\[
X'_{m-1, j} \xleftarrow{a} X''_{m-1, j} \xleftarrow{j} U_{m-1, j, i} \xleftarrow{c} X_{m-1, j, i}
\] 
where $a$ is a projective $K$-modification, $j$ is an open immersion, and $c$ is a finite, finitely presented $H_i$-invariant morphism $c$ that becomes an $H_i$-torsor over the generic fiber $(U_{m-1, j, i})_K$. Now we draw the diagram that defines the scheme $X_{n,j'}^{(m-1)}=X_{n, j, i}$ for $j'=\left(j, i \right)$:

\[
\begin{tikzcd}[column sep=8ex, row sep=8ex]
 X^{(m)}_{n,j}   \arrow{d}{} &
 X'_{n,j}  \arrow{l}{r_{n,j}} \arrow{d}{}  & X''_{n, j, i} \arrow{l}{a_n} \arrow{d} & U_{n, j, i} \arrow{l}{j_n} \arrow{d} & X'_{n,j, i} \arrow{l}{c_n} \arrow{d}{} & X_{n,j, i} \arrow{l}{\beta_{n,j,i}} \arrow{d}{}\\ 
X^{(m)}_{m, j} \drar[phantom, "\star"]  \arrow{d}{\a_{m, j}}  & X'_{m, j}  \arrow{l}{r_{m, j}} \arrow{d}{\a'_{m, j}} & X''_{m, j, i} \arrow{l}{a_m} \arrow{d} & U_{m, j, i} \arrow{l}{j_m} \arrow{d} & 
X'_{m, j, i} \arrow{l}{c_{m}}\arrow{d} & X_{m, j, i} \arrow{l}{\beta_{m, j, i}} \arrow{ld}{\text{semi-stable curve}}\\
X_{m-1} & X'_{m-1, j} \arrow{l}{r} &  X''_{m-1, j} \arrow{l}{a}  & U_{m-1, j, i} \arrow{l}{j} &  X_{m-1, j, i} \arrow{l}{c}
\end{tikzcd}
\]
where all the squares except for the ``starred'' one are Cartesian. All the schemes in this diagram are quasi-projective, flat and finitely presented over $S$. Thus, \cite[Theorem 2.2.6 and Proposition 5.1.1]{Z2} guarantees that the geometric quotient of any of those schemes for the action over $S$ by a finite group exists as a flat, finitely presented $S$-scheme. The strategy now is to use this diagram to decompose the map $X_{n, j, i}/\left(G^{(m)}_{j}\times H_i\right) \to X_n$ into a composition of certain maps that we can understand. \smallskip

We recall that $\a'_{m, j}$ is $G^{(m)}_j$-invariant, so $X'_{m, j} \to X'_{m-1, j}$ is $G^{(m)}_j$-invariant as well. Thus,  $X''_{n ,j, i}$ and $U_{n ,j, i}$ admit an action of $G^{(m)}_j$ over $S$ via base change so that $a_n$ and $j_n$ are $G^{(m)}_j$-equivariant. Now $X_{m-1, j, i}$ admits an action of $H_i$ so that $c$ is $H_i$-invariant (and an $H_i$-torsor over $K$). Hence, $X'_{n, j, i}$ similarly admits an action of $G^{(m)}_j\times H_i$ over $S$ via base change. The morphism $c_n$ is $G^{(m)}_j$-equivariant and $H_i$-invariant. The action of $G^{(m)}_j\times H_i$ on $X'_{n, j,i}$ lifts to an action on $X_{n, j,i}$ over $S$ so that $\beta_{n, j, i}$ is equivariant. This is shown in the last paragraph of the proof of Step~$2$ of Proposition~\ref{main-lemma-2}\footnote{It is shown there for $m=n$, one can then define the action on $X_{n, j, i}$ via base change.}. \smallskip

The map $j_n$ is a $G^{(m)}_{j}$-stable open immersion; moreover,  $\left\{U_{n,j,i}\right\}_{i\in I}$ forms a covering of $X_{n,j}''$ for any $j\in J_m$. Finally, the maps $c_n$ and $\beta_{n, j, i}$ are both $G^{(m)}_{j}\times H_i$-equivariant $K$-modifications. The induction hypothesis implies that the set 
\begin{equation}\label{eqn:2}
    \left(X^{(m)}_{n,j}/G^{(m)}_j, \ov{g^{(m)}_{j}} \right)_{j\in J_m}
\end{equation}
is obtained from $X_n$ as a composition of $K$-modification and Zariski open coverings. \smallskip

We note that \cite[Proposition 5.1.2]{Z2} shows that $X''_{n,j}/G^{(m)}_{j} \to X^{(m)}_{n,j}/G^{(m)}_{j}$ is a $K$-modification. The construction of the geometric quotients in \cite[Definition 2.1.1]{Z2} implies that $U_{n,j, i}/G^{(m)}_{j} \subset X''_{n,j}/G^{(m)}_{j}$ is a Zariski open subscheme, and the union (indexed by $I$) of those open subsets is a covering of $X^{(m)}_{n,j}/G^{(m)}_{j}$. Now we note that 
\[
X'_{n,j,i}/\left(G^{(m)}_{j}\times H_i\right)= \left(X'_{n,j,i}/H_i\right)/G^{(m)}_{j}
\]
Therefore, \cite[Proposition 5.1.2]{Z2} shows that the map $X'_{n,j,i}/H_i \to U_{n,j,i}$ is a $K$-modification (as finite morphisms are proper), and {\it loc. cit} shows that the map 
\[
X'_{n,j,i}/\left(G^{(m)}_{j}\times H_i\right)= \left(X'_{n,j,i}/H_i\right)/G^{(m)}_{j} \to U_{n,j, i}/G^{(m)}_{j}
\]
is a $K$-modification. Finally, we use  {\it loc. cit.} once again to conclude that the map $X_{n,j,i}/\left(G^{(m)}_{j}\times H_i\right) \to X'_{n,j,i}/\left(G^{(m)}_{j}\times H_i\right)$ is a $K$-modification. Combining all these observations and the descending inductive hypothesis with (\ref{eqn:2}), we conclude that the set
\[
\left(X_{n,j'}^{(m-1)}/G^{(m-1)}_{j'}, \ov{g_{n,j'}^{(m-1)}}\right)_{j'\in J_{m-1}}=\left(X_{n,j, i}/\left(G^{(m)}_{j}\times H_i\right), \ov{g_{n,j, i}}\right)_{j\in J_{m}, i\in I} 
\]
can be obtained from $X_n$ as a composition of open Zariski coverings and $K$-modifications.
\end{proof}

\begin{lemma}\label{main-lemma-4} Let $R$ be a valuation ring with fraction field $K$. Let $S$ be the affine scheme $\Spec R$. Suppose that $f\colon X \to S$ is a morphism that can be written as a composition 
\[
X=X_n \xr{f_n} X_{n-1} \xr{f_{n-1}}X_{n-2} \xr{f_{n-2}} \dots \xr{f_{2}}X_1 \xr{f_1}X_0=S
\]
such that each morphism $f_i$ is a relative curve (in the sense of Definition~\ref{def-curve}) that is $K$-smooth. Then there is a finite Galois extension $K\subset K'$ with $R'$ the normalization of $R$ in $K'$, a finite non-empty set $J$, a set of finite groups $\left(G_{j}\right)_{j\in J}$ indexed by $J$, and commutative diagrams: 

\[
\begin{tikzcd}[column sep=13ex]
X_{n, R'} \arrow{d}{f_{n, R'}}  & X'_{n,j}  \arrow{l}{g_{n,j}} \arrow{d}{h_{n,j}}\\
X_{n-1, R'} \arrow{d}{f_{n-1, R'}} & X'_{n-1, j} \arrow{l}{g_{n-1,j}} \arrow{d}{h_{n-1, j}}\\
\vdots \arrow{d}{f_{2, R'}} & \vdots \arrow{l}{g_{k,j}} \arrow{d}{h_{2, j}}\\
X_{1, R'}\arrow{d}{f_{1, R'}} & X'_{1, j} \arrow{ld}{h_{1, j}} \arrow{l}{g_{1,j}} \\
\Spec {R'} &  
\end{tikzcd}
\]
of flat, finitely presented $S'\coloneqq \Spec R'$-schemes such that 
\begin{enumerate}
\item\label{main-lemma-4-1} $X'_{k, j}$ admits an $R'$-action of the group $G_j$ for any $k\geq 1$ and $j\in J$. Moreover, $g_{k,j}$ is $G_j$-invariant and $h_{k,j}$ is $G_j$-equivariant for any $j\in J, k=1, \dots, n$.
\item\label{main-lemma-4-2} The $K'$-restriction $(g_{k, j})_{K'}\colon (X'_{k, j})_{K'} \to (X_{k})_{K'}$ is a $G_j$-torsor over its (open) image $V_{k,j}$ for $k\geq 1$.
\item\label{main-lemma-4-3} $h_{k,j}$ is a semi-stable, $K'$-smooth relative curve for any $j\in J$ and any $k\geq 1$ 
\item\label{main-lemma-4-4} $g_{k,j}$ is quasi-projective for any $k\geq 1$ and $j\in J$. Moreover, the map $g_k\colon \sqcup_{j\in J}X'_{k,j} \to X_{k, R'}$ is a $K'$-\'etale covering for any $k\geq 1$.
\item\label{main-lemma-4-5} Assume that $K$ is a henselian rank-$1$ valued field\footnote{The results holds true if one only assumes that $R$ is a henselian valuation ring. The rank-$1$ assumption is only used to show that the integral closure of $R$ in an algebraic field extension $L/K$ is a valuation ring. This can be proven without the rank-$1$ assumption using \cite[Ch. 6, \textsection 8, n.6, Proposition 6]{Bou} and \cite[\href{https://stacks.math.columbia.edu/tag/09XI}{Tag 09XI}]{stacks-project}.} (in the sense of Definition~\ref{defn-henselian-field}) and $f_i$ is quasi-projective for every $i\geq 1$. Then the geometric quotient $X'_{n,j}/G_j$ exists as a flat, finitely presented $S'$-scheme for any $j\in J$. Moreover, we can choose $X'_{n,j}$ so that the set $\left(X'_{n,j}/G_{j}, \ \ov{g_{n,j}}\right)_{j\in J}$ can be obtained from $X_{n, R'}$ as a composition of Zariski open coverings and $K$-modifications (in the sense of Definition~\ref{comp-alg}). 
\end{enumerate}
\end{lemma}
\begin{proof}
The proof is similar to Proposition~\ref{main-lemma-2}, but one needs to be careful as the extension $R \to R'$ may not be finite. Instead we pick $J, \ G_j, \ X_{j,k} \xr{g_{j, k}} X_{j}$ as in Proposition~\ref{main-lemma-2} (or as in Corollary~\ref{main-lemma-3} in the case of henselian $K$ and quasi-projective $f_i$). In particular, $X_{k, j} \to X_{k-1, j}$ is a semi-stable $K$-smooth relative curve for $k\geq 2$, $j\in J$.  \smallskip

We note that $X_{1, j}$ is automatically a $K$-smooth $S$-curve as it is $S$-flat and finitely presented by construction. Thus, we can use Theorem~\ref{stable-modification-valuation} to get a finite Galois extension $K\subset K'$ with $R'$ the integral closure of $R$ in $K'$ such that the $R'$-scheme $X_{1, j, R'}$ admits a $R'$-stable $K$-modification $\beta_{1, j}\colon X'_{1,j} \to X_{1, j, R'}$ for any $j\in J$. We note that Theorem~\ref{stable-modification-valuation} a priori only gives us such an extension $K'_j$ separately for each $j\in J$, but then we can find some finite Galois extension $K \subset K'$ that dominates all $K'_j$. So this field can be chosen independently of $j\in J$. \smallskip

So we have commutative diagrams (indexed by $J$):
\[
\begin{tikzcd}[column sep=13ex]
X_{n, R'} \arrow{d}{f_{n, R'}}  &  \arrow{l}{g_{n, j, R'}} X_{n, j, R'} \arrow{d} & X'_{n,j}  \arrow{l}{\beta_{n, j}} \arrow{d}{h_{n,j}}\\
X_{n-1, R'} \arrow{d}{f_{n-1, R'}} & \arrow{l}{g_{n-1, j, R'}} X_{n-1, j, R'} \arrow{d} & X'_{n-1, j} \arrow{l}{\beta_{n-1, j}} \arrow{d}{h_{n-1, j}}\\
\vdots \arrow{d}{f_{2, R'}} & \vdots \arrow{d} \arrow{l} & \vdots \arrow{l}{\beta_{k, j}} \arrow{d}{h_{2, j}}\\
X_{1, R'}\arrow{d}{f_{1, R'}} &\arrow{l}{g_{1, j, R'}} X_{1, j, R'} \arrow{ld}{}  &  X'_{1, j} \arrow{l}{\beta_{1, j}} \arrow{lld}{h_{1, j}}  \\
\Spec {R'} &  &
\end{tikzcd}
\]
where each right square is Cartesian\footnote{We define $h_{k, j}$ as the base change of $X_{k, j, R'} \to X_{k-1, j, R'}$ for $k\geq 2$. Similarly, for $k\geq 2$, $\beta_{k, j}$ are defined as the base change of $\beta_{1, j}$.}. We note that $h_{k,j}$ is a relative semi-stable curve for each $k\geq 1$. Moreover, \cite[Theorem 1.1]{T1} implies that $\beta_{1, j}$ is a projective $K'$-modification, so the same holds for $\beta_{k,j}$ for any $k\geq 1$. So if we define $g_{k, j} \colon X'_{k, j}\to X_{k, R'}$ as the composition $g_{k ,j, R'}\circ \beta_{k, j}$, then we see that each of those $g_{k, j}$ is quasi-projective and the total map
\[
g_k\colon \bigsqcup_{j\in J} X'_{k, j} \to X_{k, R'}
\]
is a quasi-projective $K'$-\'etale covering for any $k \geq 1$. As for the group action, we use the uniqueness result \cite[Theorem 1.2]{T1} to lift the $R'$-action of $G_j$ on $X_{1, j, R'}$ to an $R'$-action of $G_j$ on $X'_{1, j}$. Since $\beta_{1,j}$ is $G_j$-equivariant and all the vertical maps in the middle column are $G_j$-equivariant, we obtain a canonical $R'$-action of $G_j$ on each $X'_{k,j}$ making morphisms $\beta_{k,j}$ and $h_{k,j}$ be equivariant for any $k\geq 1$. Then it is easy to see that each $g_{k, j}$ is $G_j$-invariant and its $K'$-restriction $(g_{k,j})_{K'}$ becomes a $G_j$ torsor over its open image. \smallskip

Now we go to the situation of a henselian rank-$1$ valued field $K$ and quasi-projective $f_i$. We note that Lemma~\ref{prop-hens} implies that the normalization $R'$ is a valuation ring. So we can use \cite[Theorem 2.2.7 and Proposition 5.1.1]{Z2} to conclude that $X'_{n,j}/G_j$ exists as a flat, finitely presented $R'$-scheme. Since the formation of the geometric quotient commutes with flat base change \cite[Theorem 2.1.16]{Z2}, we use Corollary~\ref{main-lemma-3} to see that the set $\left(X_{n,j, R'}/G_j, \ov{g_{n, j, R'}}\right)$ can be obtained from $X_{n, R'}$ as a composition of Zariski open coverings and $K'$-modifications. \smallskip

Now we note that \cite[Proposition 5.1.2]{Z2} gives that the map $X'_{n,j}/G_j \to X_{n, j, R'}/G_j$ is a $K'$-modification for any $j\in J$. This, in turn, implies that the set $\left(X'_{n,j}/G_j, \ov{g_{n, j}}\right)$ can be obtained from $X_{n, R'}$ as a composition of Zariski open coverings and $K'$-modifications. This finishes the proof.
\end{proof}



\appendix

\section*{Appendix}
\section{Approximation Techniques}

\begin{lemma}\label{spread-group} Let $R$ be a ring, let $G$ be a finite group, and let $A$ be an $R$-algebra with an $R$-algebra action of $G$. 
Suppose that an $R$-algebra $A = \colim_{i\in I} A_i$ if a filtered colimit of $R$-algebras $A_i$ with the following properties
\begin{itemize}\itemsep0.5em
\item Each $A_i$ is a finite type $R$-subalgebra of $A$.
\item Each finite type $R$-subalgebra $B\subset A$, which contains $A_i$ for some $i$, is equal to $A_j$ for some other $j$.
\end{itemize}
The subsystem $J=\{i\in I \ | \ A_i \text{ is } G \text{-stable subalgebra of } A\}$ is filtered and $A=\colim_{j\in J} A_j$.
\end{lemma}
\begin{proof}

It suffices to show that for any $i\in I$ there is $j\in I$ such that $A_i\subset A_j$, and $A_j$ is $G$-stable subalgebra of $A$. Pick some $R$-algebra generators $\{h_{\a}\}_{\a \in T}$ of $A_i$ for some finite set $T$. And now define $B$ as an $R$-subalgebra of $A$ generated by $g(h_{\a})$ for all $g\in G$ and $\a \in T$. Since $G$ acts on $A$ by $R$-algebra automorphisms, we conclude that $B$ is a $G$-stable $R$-algebra containing $A_i$. Moreover, this is still a finite type $R$-algebra because the group $G$ is finite. Thus our assumption implies that $B=A_j$ for some $j\in J$. This proves the claim.
\end{proof}
\smallskip

\begin{lemma}\label{val-approximation} Let $(\O, \pi)$ be a pair of a rank-$1$ valuation ring $\O$ with algebraically closed fraction field $K$ and a pseudo-uniformizer $\pi$. Then it can be written as a filtered colimit $(\O, \pi) \simeq \colim (A_i, t_i)$, where each $A_i$ is a regular noetherian subring of $\O$ and $\mathrm{V}(t_i)_{\red}$ is an snc divisor in $\Spec A_i$ (see Definition~\ref{snc}). 
\end{lemma}
\begin{proof}
We consider $3$ cases separately: $\O$ is of equal characteristic $0$, $\O$ is of equal characteristic $p>0$, or $\O$ of mixed characteristic $(0, p)$. Depending on the case, we can 
we can write $\O$ as a filtered colimit of finite type $\Q$ (resp. $\mathbf{F}_p$, resp. $\Z_{(p)}$) sub-algebras $A_i$ containing $\pi$. Now it suffices to show that for any $A_i \to \O_C$ there is a factorization $A_i \to A_j \to \O$ such that $A_j$ is regular, finite type over $\Q$ (resp. $\mathbf{F}_p$, resp. $\Z_{(p)}$), and $\rm{V}(\pi)_{\red}$ is snc in $\Spec A_j$. \smallskip

We use de Jong alternation theorem \cite[Theorem 4.1]{DJ1} in the equal characteristic case, and \cite[Theorem 6.5]{DJ1} in the mixed characteristic case, to say that there is an alteration $f\colon X \to \Spec A_i$ such that $X$ is regular and $f^{-1}(\rm{V}(\pi))_{\red}$ is snc. Since alterations are generically finite and $K$ is algebraically closed, we can lift a map $\Frac(A_i) \to K$ to the map $K(X) \to K$. Then since $X$ is proper over $A_i$, we can lift a map $\Spec \O \to \Spec A_i$ to the map $\Spec \O \to X$. We now choose some open affine neighborhood of its image and denote it by $\Spec A_j$. Then $A_j$ is regular by construction, and there is a factorization $A_i \to A_j \to \O$. The second map is injective as it is compatible with the injection $\Frac(A_j)=K(X) \to K$. Finally, consider $g\colon \Spec A_j \to \Spec A_i$. The construction said that $g^{-1}(\rm{V}(\pi)_{\red})$ is snc, but it is the same as $\rm{V}(g^*(\pi))_{\red}=\rm{V}(\pi)_{\red}$, where we consider $\pi$ as an element of $A_j$. This finishes the proof.
\end{proof}

\section{Formal And Rigid-Analytic Geometry}\label{adic}

The two main goals of this section are to spell out precisely some notions related to formal algebraic and rigid-analytic geometry that we use in the paper, and to compare these notions to their algebraic counterparts. Since terminology in these areas slightly vary from reference to reference, we decide to make the definitions as precise as we can. But we do not really discuss terminology that seem to have only one interpretation.  

\subsection{Formal Geometry}
We fix a complete rank-$1$ valuation field $\O_K$ with fraction field $K$. We recall that a formal $\O_K$-scheme always means for us an $I$-adic formal $\O_K$-scheme for a(ny) ideal of definition $I\subset \O_K$. And by a completion $\wdh{X}$ of finitely type $\O_K$-scheme $X$, we always mean the $I$-adic completion.

We refer to \cite[\textsection 7.4]{B} for the discussion of the notions of morphisms of topologically finite type and of topologically finite presentation in the context of formal $\O_K$-schemes. We now recall the notion of admissible formal $\O_K$-schemes from \cite[\textsection 7.4]{B} and relate it to its algebraic counterpart.

\begin{defn} A formal $\O_K$-scheme $\X$ is {\it admissible}, if it is $\O_K$-flat and locally of topologically finite type.
\end{defn}
\begin{rmk} We note that admissible formal $\O_K$-schemes are always topologically finitely presented by \cite[Corollary 7.3/5]{B}. Thus our definition coincides with \cite[Definition 7.4/1]{B}
\end{rmk}

Sometimes it is required that an admissible formal scheme should be also quasi-compact and quasi-separated. However, we follow the terminology of \cite{B} and \cite{BL1} and do not put those conditions into the definition of admissible formal schemes. We also note that in our situation quasi-compactness of $\X$ automatically implies quasi-separatedness of $\X$. Indeed, $|\X|=|\ov{\X}|$ and the latter is a noetherian topological space.

\begin{lemma}\label{different-adm} Let $X$ be an admissible $\O_K$-scheme (in the sense of Definition \ref{U-adm}), then its completion $\wdh{X}$ is an admissible formal $\O_K$-scheme.
\end{lemma}
\begin{proof}
We note that $X_K$ is schematically dense in $X$ if and only if $\O_X$ is a $\O_K$-torsion free sheaf. We claim that this implies $\O_K$-flatness of $\X$. The statement is Zariski local on $\X$, so we may and do assume that $X$ is affine. Now the completion of a flat $\O_K$-module is $\O_K$-flat by \cite[Proposition 7.3/11 and Remark 7.1/6]{B}. Thus $\wdh{X}$ is admissible as a formal $\O_K$-scheme.
\end{proof}
\medskip

We introduce the notion of a $K$-modification in the set-up of formal geometry and compare it to its algebraic analogue. 

\begin{defn}\label{C-mod-formal} We say that a morphism $f\colon \X \to \Y$ of admissible, quasi-compact, quasi-separated $\O_K$-formal schemes is a {\it rig-isomorphism} (resp. {\it rig-\'etale}, resp. {\it rig-smooth}, resp. {\it rig-surjective}), if the generic fibre $f_K\colon  \X_K \to \Y_K$ is an isomorphism (resp. \'etale, resp. smooth, resp. surjective morphism) of adic spaces.
\end{defn}

\begin{defn}\label{rig-etale-covering} We say that a morphism $f\colon \X \to \Y$ of admissible, quasi-compact, quasi-separated $\O_K$-formal schemes is a {\it rig-\'etale covering} if it is rig-\'etale and rig-surjective.
\end{defn}

\begin{rmk} We note that Definition~\ref{C-mod-formal} implies that any rig-isomorphism $f$ is proper (in the sense of \cite[Definition 4.7.1]{FujKato}). This follows from \cite[Theorem 3.1]{Lutke-proper} in the case $\O_K$ is discretely valued and from \cite[Corollary 4.4. and Corollary 4.5]{Temkin-proper} in general.
\end{rmk}

Now we need the following lemma that was proven in \cite[Corollaire 5.7.12]{RG}. However, the statement there is formulated for algebraic spaces, so to make things more accessible, we provide a short argument that extracts the result in the scheme case.

\begin{lemma}\label{RG-5-7-12} Let $U\subset S$ be a schematically dense quasi-compact open subscheme of a quasi-compact, quasi-separated scheme $S$. Let $f\colon X\to S$ be a proper, finitely presented morphism that is an isomorphism over $U$. Then there is an $f^{-1}(U)$-admissible blow-up $g \colon X' \to X$ such that the composition $h\colon X' \to S$ is also a $U$-admissible blow-up.
\end{lemma}
\begin{proof}
We note that \cite[\href{https://stacks.math.columbia.edu/tag/081R}{Tag 081R}]{stacks-project} implies that there is a $U$-admissible blow-up $S' \to S$ such that the strict transform $X'\to S'$ is an open immersion. It is also proper, so it must be an isomorphism by schematic density of $U$. Now \cite[\href{https://stacks.math.columbia.edu/tag/080E}{Tag 080E}]{stacks-project} implies that $X' \to X$ is an $f^{-1}(U)$-admissible blow-up. Thus we get a factorization $X' \xr{g} X \xr{f} S$ such that $g$ and $f \circ g$ are admissible blow-ups. 
\end{proof}

\begin{lemma}\label{different-C-mod} Let $f\colon X\to Y$ be a $K$-modification (in the sense of Definition \ref{U-mod}) of flat, finitely presented $\O_K$-schemes. Then its completion $\wdh{f}\colon \wdh{X} \to \wdh{Y}$ is a rig-isomorphism (in the sense of Definition \ref{C-mod-formal}).
\end{lemma}
\begin{proof}
Lemma~\ref{different-adm} says that $\wdh{X}$ is an admissible formal $\O_K$-scheme. Moreover, the same argument (using \cite[Proposition 7.3/10]{B} and \cite[Proposition 4.7.3]{FujKato} instead \cite[Proposition 7.3/11]{B}) shows that $\wdh{f}$ is proper since $f$ is proper. 
\smallskip

Now we use Lemma~\ref{RG-5-7-12} to find a morphism $g\colon X' \to X$ such that $g$ and $f\circ g$ are both admissible $K$-blow ups. Therefore, the completions $\wdh{g}\colon \wdh{X'} \to \wdh{X}$ and $\wdh{f\circ g}\colon \wdh{X'} \to \wdh{Y}$ are both rig-isomorphisms by \cite[Proposition 8.4/2]{B}. This implies that $\wdh{f}$ is a $K$-modification as well.
\end{proof}

Now we are ready to give a definition of a semi-stable formal curve fibration and relate it to the algebraic version of it.

\begin{defn}\label{def-ss-formal} We say that a morphism of locally topologically finitely presented formal $\O_K$-schemes $f\colon  \X \to \mathcal \Y$ is a {\it semi-stable formal relative curve} (or {\it semi-stable formal curve fibration}), if for any point $x\in \X$ either $f$ is smooth at $x$, or there is an open affine neighborhood $\Spf B \subset \X$ containing $x$ and an open affine neighborhood $\Spf A \subset \Y$ containing $y=f(x)$, such that there exists a diagram of pointed schemes

\[
\begin{tikzcd}
& (\Spf C, s)  \arrow[swap]{ld}{g} \arrow{rd}{h}\\
(\Spf B,x) &  & \left(\Spf \frac{A\langle U,V\rangle}{(UV-a)}, \{y,0,0\}\right),
\end{tikzcd}
\]
where $g$ and $h$ are \'etale and $a$ lies in the ideal of $y$.
\end{defn}
\smallskip

\begin{rmk} We note that any semi-stable formal curve fibration is flat. Indeed, \cite[Corollary I.4.8.2]{FujKato} implies that it suffices to check the claim modulo any power of a pseudo-uniformizer $\pi$. So the claim boils down to the fact that a (schematic) semi-stable curve fibration is flat. 
\end{rmk}

\begin{lemma}\label{completion-ss} Let $f\colon X \to Y$ be a morphism of flat, finitely presented $\O_K$-schemes that is a relative semi-stable curve (in the sense of Definition~\ref{def-ss}). Then its completion $\wdh{f}\colon \wdh{X} \to \wdh{Y}$ is a semi-stable formal relative curve (in the sense of Definition~\ref{def-ss-formal}).
\end{lemma}
\begin{proof}
Using Lemma~\ref{ss-alternative} the question boils down to the fact that a completion of an \'etale morphism of algebras is formally \'etale. This condition can be checked modulo all powers of a pseudo-uniformizer, but since $\wdh{A}/I^n \cong A/I^n$ holds for any completion along a {\it finitely generated} ideal, we conclude the statement. 
\end{proof}

\begin{defn}\label{semi-def-formal} Let $\O_K$ be a complete discretely valued ring. An admissible formal $\O_K$-scheme $\X$ is called {\it strictly semi-stable} if Zariski locally it admits an \'etale morphism 
\[
\mathfrak{U} \to \Spf \frac{\O_K\langle t_0, \dots, t_l\rangle}{(t_{0}\cdots t_{m}-\pi)}
\]
for some integers $m\leq l$, and a uniformizer $\pi\in \m_K \setminus \m_K^2$. 
\end{defn}

\begin{defn}\label{poly-def-formal} Let $\O_K$ be a complete rank-$1$ valuation ring. An admissible rig-smooth formal $\O_K$-scheme $\X$ is called {\it polystable} if \'etale locally it admits an \'etale morphism 
\[
\mathfrak{U} \to \Spf  \frac{\O_K\langle t_{1,0},\dots, t_{1, n_0}, \dots, t_{l,0}, \dots, t_{l, n_l}\rangle}{(t_{1,0}\cdots t_{1, n_1}-\pi_1, \dots, t_{l,0}\cdots t_{l, n_l}-\pi_l)}
\]
for some $\pi_i\in \O_K \setminus \{0\}$. 
\end{defn}

\begin{lemma}\label{completion-polyss} Let $X$ be a morphism of flat, finitely presented $\O_K$-schemes that is $K$-smooth and polystable in the sense of Definition~\ref{poly-def} (resp. strictly semistable in the sense of Definition~\ref{defn:strictly-semistable}). Then its completion $\wdh{X}$ is a rig-smooth polystable formal $\O_K$-scheme in the sense of Definition~\ref{poly-def-formal} (resp. strictly semistable in the sense of Definition~\ref{semi-def-formal}).
\end{lemma}
\begin{proof}
The proof is identical to that of Lemma~\ref{completion-ss}.
\end{proof}

We need to introduce another definition that is not standard: 

\begin{defn}\label{quasiproj-formal} We say that a morphism $f\colon \X \to \Y$ of topologically finite type formal $\O_K$-schemes is {\it formally quasi-projective}, if its special fibre 
\[
\ov{f}\colon  \ov{\X} \to \ov{\Y}
\]
is a quasi-projective morphism of $\O_K/\m_K$-schemes (in the sense of definition \cite[II, 5.3.1]{EGA}. It is clear that a composition of formally quasi-projective morphisms is formally quasi-projective.
\end{defn}
\medskip

\subsection{Rigid-Analytic Geometry}

We remind the reader that we use Huber's foundation for rigid-analytic geometry in this paper. So, a rigid-analytic space over a complete rank-$1$ valuation field $K$ always means here an adic space locally topologically finite type over $\Spa(K, \O_K)$. When we need to use classical rigid-analytic spaces, we refer to them as Tate rigid-analytic spaces.

We recall that there is a fully faithfull functor 
\begin{equation*}
r_K\colon\left\{\text{Tate Rigid-Analytic Spaces over } K\right\} \to \left\{\begin{array}{l} \text {Adic Spaces locally of topologically} \\ \text{finite type over}\  \Spa(K, \O_K)\end{array}\right\}
\end{equation*}
that becomes an equivalence on categories when it is restricted to quasi-compact and quasi-separated objects on both sides. We do not define this functor here, instead we only mention that it sends an affinoid rigid-analytic space $\operatorname{Sp}(A)$ to the adic space $\Spa(A, A^{\circ})$. The main difficulty then is to show that we can ``glue'' to define $r_K$ on non-affinoids. We refer to \cite[\textsection
4]{H1} and \cite[Lecture 16]{Seminar} for the full construction and discussion of some properties of this functor. \\

We also need the functor of adic generic fiber: 
\begin{equation*}
(-)_K\colon \left\{\begin{array}{l}\text{locally topologically finite} \\ \text{type formal } \O_K\text{-schemes}\end{array}\right\} \to \left\{\begin{array}{l} \text {Adic Spaces locally of topologically} \\ \text{finite type over}\  \Spa(K, \O_K)\end{array}\right\}
\end{equation*}
that is defined in \cite[\textsection 1.9]{H3} (it is denoted by $d$ there). Given an affine formal topologically finite type $\O_K$-scheme $\Spf(A)$, this functor assigns the affinoid adic space $\Spa(A\otimes_{\O_K} K, A^+)$ where $A^+$ is the integral closure of $A$ in $A\otimes_{\O_K} K$. One of the main features of this functor is that there is a natural isomorphism of functors: 
\[
(-)_K \cong r_K\circ (-)_{\text{rig}}
\]
where $(-)_{\text{rig}}$ is the ``classical'' Raynaud generic fiber (as defined in \cite[\textsection 7.4]{B}). \\

The last functor of interest is the analytification functor: 
\begin{equation*}
(-)^{\operatorname{an}}\colon \left\{\begin{array}{l}\text{locally finite type} \\  K\text{-schemes}\end{array}\right\} \to \left\{\begin{array}{l} \text {Adic Spaces locally of topologically} \\ \text{finite type over}\  \Spa(K, \O_K)\end{array}\right\}
\end{equation*}
that is defined as a composition 
\[
(-)^{\an}=r_K \circ (-)^{\text{rig}}
\]
of the classical analytification functor $(-)^{\text{rig}}$ (as it is defined in \cite[\textsection 5.4]{B}\footnote{We want to emphasize that the functors $(-)^{\rm{rig}}$ and $(-)_{\rm{rig}}$ are different. In particular, they have different source categories. The target category for both functors is the category of Tate Rigid-Analytic Spaces over $K$.}) and the functor $r_K$. 

The last thing we want to remind the reader is that given any morphism $\varphi\colon X \to Y$ of finite type $\O_K$-schemes, there are two ways to take the ``generic fibre'' of $\varphi$. The first one is to form the map between the usual generic fibres in the category of schemes and then consider its analytification. We denote the resulting morphism by $\varphi_K^{\an}\colon X_K^{\an} \to Y_K^{\an}$. Another way is first to complete $\varphi$ and then take its adic generic fibre in the sense of Raynaud; the result of this approach we denote by $\widehat{\varphi}_K\colon \widehat{X}_K \to \widehat{Y}_K$. In general, these two approaches are very different, i.e. $\A^{1, \an}_{K}$ is an analytic affine line, but $(\widehat{\A^1_{\O_K}})_K \cong \mathbf D_K$ is the closed unit disc. 

\begin{lemma}\label{open}\cite[Theorem 5.3.1]{C} Let $K$ be a complete rank-$1$ valuation field, and let $X$ be a locally finitely presented $\O_K$-scheme. Then there is a functorial morphism 
\[
j_X\colon \wdh{X}_K \to X^{\an}_K
\]
such that $j_X$ is a quasi-compact open immersion if $X$ is separated and admits a locally finite affine covering.
\end{lemma}

\begin{lemma}\label{Galois-analytification} Let $\varphi\colon X \to Y$ be a morphism of flat finitely presented separated $\O_K$-schemes, and suppose that $\varphi_{K}\colon X_K \to Y_K$ factors as a $G$-torsor over its open quasi-compact image $U\subset Y_K$ for a finite group $G$. Then the morphism $\hat{\varphi}_K\colon  \widehat{X}_K \to \widehat{Y}_K$ factors as a $G$-torsor over its open quasi-compact image $V\subset (\widehat{Y})_K$.
\end{lemma}
\begin{proof}
We start by observing that  $\varphi_K^{\an}\colon X_K^{\an} \to Y_K^{\an}$ factors as a $G$-torsor over an open $U^{\an}$. We want to use it to deduce the result for $\wdh{\varphi}_K$ via the commutative square

\[
\begin{tikzcd}
\wdh{X}_K \arrow{r}{j_X}  \arrow{d}{\wdh{\varphi}_K}& X_K^{\an} \arrow{d}{\varphi_K^{\an}} \\
\wdh{Y}_K \arrow{r}{j_Y} & Y_K^{\an}\\
\end{tikzcd}
\]
where $j_X$ the $G$-equivariant from Lemma~\ref{open}. Clearly, both morphisms $\wdh{\varphi}_K$ and $\varphi_K^{\an}$ are $G$-invariant. The morphism $\varphi_K^{\an}$ is \'etale as the analytification of the \'etale morphism $\varphi_K$. Lemma~\ref{open} gurantees that $j_X$ and $j_Y$ are both open immersions, so $\wdh{\varphi}_K$ is also \'etale. In particular, its image $V\coloneqq \wdh{\varphi}_K(\wdh{X}_K)$ is open and quasi-compact (as $\wdh{X}_K$ is quasi-compact). 

We claim that $\wdh{\varphi}_K$ is a $G$-torsor over $V$. It suffices to prove that the natural map $\wdh{X}_K \to X_K^{\an}\times_{Y_K^{\an}} V$ is an isomorphism. Since $j_X$ and $j_Y$ are quasi-compact open immersions by Lemma~\ref{open}, we conclude that it is equivalent to show that $(\varphi_K^{\an})^{-1}(V)=\wdh{X}_K$.

We know that $\wdh{X}_K \subset (\varphi_K^{\an})^{-1}(V)$ by the construction of $V$ as an image of $\wdh{X}_K$. So it is enough to show the reverse inclusion. Note that since $V\subset U^{\an}$, we see that $(\varphi_K^{\an})^{-1}(V) \to V$ is a $G$-torsor as a pullback of a $G$-torsor. Moreover, note that a fibre over any point $v\in V$ contains at least one point from $\wdh{X}_K$. This implies that $(\varphi_K^{\an})^{-1}(V)=G.\wdh{X}_K$ (the orbit of $\wdh{X}_K$ under the action of $G$), thus $G.\wdh{X}_K=\wdh{X}_K$ as $\wdh{X}_K$ is $G$-stable. We finally deduce that $\wdh{X}_K = (\varphi_K^{\an})^{-1}(V)$, so $\wdh{\varphi}_K$ is a $G$-torsor over $V$.
\end{proof} 



\section{Log Geometry}\label{log}

The main goal of this section is to construct a structure of log-smooth, log-variety on a successive semi-stable $C$-smooth fibration over a rank-$1$ valuation ring $\O_C$ with algebraically closed fraction field $C$ (see Theorem~\ref{log-str-semist}); this is used at the end of the proof of Theorem~\ref{schematic-uniformization}. The natural log structure on $\Spec \O_C$ is not fine, so one needs to be careful while working with log schemes over $\Spec \O_C$. We handle this issue by using Lemma~\ref{val-approximation}. 

For the readers convenience, we recall the main definitions of log geometry in this appendix. In particular, we recall the definition of a log $\O_C$-variety from \cite{AdipLiuPakTemkin}. We found \cite{Kato}, \cite{ogus-log} and \cite{Niziol} to be especially useful sources on this subject; we refer there for a more comprehensive treatment of the theory. \medskip

Throughout this appendix, all monoids are meant to be commutative monoids. Also, we work only with \'etale log structures. In particular, all sheaves in this sections are considered as sheaves on the small \'etale site of $X$.

\subsection{Basic Definitions}

\begin{defn1} A {\it pre-logarithmic structure} (or just {\it pre-log structure}) on a scheme $X$ is a homomorphism of sheaves of monoids $\alpha\colon \mathcal{M}_X \to \O_X$ on $X_{\et}$. A {\it logarithmic structure} (or just {\it log structure}) is a pre-logarithmic structure such that the induced homomorphism $\alpha^{-1}(\O_X^\times) \to \O_X^\times$ is an isomorphism. Morphisms are defined in the evident way.

We denote the category of prelog structures on $X$ as $\mathbf{Plog}_X$ and the category of log structures on $X$ by $\mathbf{log}_X$. There is the natural forgetful functor $r_X\colon \mathbf{log}_X \to \mathbf{Plog}_X$ denoted also as $r$ if it does not cause any confusion. 
\end{defn1}

\begin{defn1} A {\it log scheme} is a scheme X endowed with a log structure $\a \colon \mathcal{M} \to \O_X$.
\end{defn1}

\begin{exmpl1}\label{associated} Let $Z\subset X$ be a closed subscheme with the complement $j\colon U \to X$. We define the log structure $(\mathcal M_X, \a)$ {\it associated to the pair} $(X, Z)$ as follows. We define the sheaf of monoids $\mathcal M_X$ as $\mathcal M_X = \O_X \times_{j_*\O_U} j_*\O_U^\times$ that is sometimes informally denoted by $\O_X\cap j_*\O_U^\times$. We define the map $\a\colon \mathcal M_X \to \O_X$ to be the natural inclusion map. Then it is straightforward to check that this defines a logarithmic structure on $X$, or that $(X, \mathcal M_X, \a)$ is a log scheme. 
\end{exmpl1}

\begin{exmpl1} Let $\O_K$ be a rank-$1$ valuation ring with a uniformizer $\pi$ and the residue field $k$. We call the log structure associated with the pair $(\Spec \O_K, \Spec k)$ the {\it standard} log structure on $\Spec \O_K$. Explicitly, $\M(\Spec A) = A \cap A[\frac{1}{\pi}]^\times$ for any affine \'etale map $\Spec A \to \Spec \O_K$. 
\end{exmpl1}

The notion of a log scheme is too general. The sheaf of monoids $\mathcal{M}_X$ can be pretty much anything. We need to define the notion of quasi-coherent log structures. Before doing this, we need to discuss the way to associate log structures to monoids or, more generally, to any pre-log structure. 

\begin{defn1} Let $\a\colon \mathcal{P}_X \to \O_X$ be a pre-log structure on $X$. Then we define the {\it associated log structure} $(\mathcal{P}_X^{\log}, \a^{\log})$ as the pushout in the following diagram:
\[
\begin{tikzcd}
\a^{-1}(\O_X^\times) \arrow{d}{\a} \arrow{r} & \mathcal{P}_X \arrow{d} \arrow{rdd}{\a} & \\
\O_X^\times \arrow{r} \arrow{rrd}& \mathcal{P}_X^{\log} \arrow{rd}{\a^{\log}} & \\
& & \O_X
\end{tikzcd}
\]
\end{defn1}

\begin{lemma1} The functor $(-)^{\log}\colon \mathbf{Plog}_X \to \mathbf{log}_X$ defines a left adjoint to the forgetful functor $r_X\colon \mathbf{log}_X \to \mathbf{Plog}_X$. More precisely, the natural map $\mathcal{P}_X \to \mathcal{P}^{\log}_X$ induces an isomorphism $\Hom_{\bf{log}_X}(\mathcal{P}_X^{\log}, \mathcal M) \simeq \Hom_{\bf{Plog}_X}(\mathcal{P}_X, r_X(\mathcal M))$ for any log-structure $\mathcal M$.
\end{lemma1}
\begin{proof}
The proof is essentially trivial and left as an exercise.
\end{proof}

\begin{defn1} Let $X$ be a scheme and $P$ a monoid. A sheaf of monoids $\ud{P}_X$ is defined as the constant sheaf associated to $P$.
\end{defn1}

\begin{exmpl1}\label{monoid-log} Let $R$ be a ring, $P$ be a monoid, and $R[P]$ be the monoid ring of $P$. Then $\Spec R[P]$ has a canonical log structure. Namely, it is defined as the log structure associated to the pre-log structure given by a morphism of sheaves of monoid $\ud{P}_{X}\to \O_X$ on $X=\Spec R[P]$. That morphism, in turn, comes from the natural morphism of monoids $P \to R[P]$.
\end{exmpl1}

We recall that the category of monoids admits all small colimits \cite[page 2]{ogus-log}. For example, if $\begin{tikzcd}
            P\arrow[r, swap, shift right=.75ex, "u"] \arrow[r, shift right=-.75ex, "v"] & Q\end{tikzcd}$ are homomorphisms of monoids, the {\it coequalizer} $\mathrm{coeq}(\begin{tikzcd}
            P\arrow[r, swap, shift right=.75ex, "u"] \arrow[r, shift right=-.75ex, "v"] & Q\end{tikzcd})$ is constructed as the quotient of $Q$ by the minimal congruence relation\footnote{Equivalence relation that is also a monoid} $R\subset Q\oplus Q$ containing elements $(u(p), v(p)) \subset Q\oplus Q$ for all $p\in P$. 
One can construct the {\it cokernel} of $u\colon P \to Q$ as $\coker (u) = \mathrm{coeq}(\begin{tikzcd}
            P\arrow[r, swap, shift right=.75ex, "u"] \arrow[r, shift right=-.75ex, "1"] & Q\end{tikzcd})$

\begin{defn1}\label{gp-units} The {\it Grothendieck group $(P^{\rm{gp}}, p)$} of a monoid $P$ is defined as a group $P^{\rm{gp}}$ with a morphism of monoids $p\colon P \to P^{\rm{gp}}$ that is universal among maps to commutative groups.  \medskip 

The {\it group of units} $P^\times$ of a monoid $P$ is a group of invertible elements in $P$. \medskip 

The {\it sharpening} of a monoid $P$ is the monoid $\overline{P}\coloneqq P/P^\times$.
\end{defn1}

\begin{rmk1} Existence of $(P^{\rm{gp}}, p)$ is discussed at \cite[Beginning of Section~I.1.3]{ogus-log}. Namely, $P^{\rm{gp}}$ can be identified with the cokernel of the diagonal map $P \to P\oplus P$. 
\end{rmk1}

We use the terminology introduced in \cite{AdipLiuPakTemkin} and define the monoid $M[t_1, \dots, t_n]$ to be the monoid $M\oplus \N^{\oplus n}$. We define the monoid $M[t_1, \dots, t_n]/(f_i=g_i)$ for $f_i, g_i\in M[t_1, \dots, t_n]$ ($i=1, \dots, m$) to be the coequalizer 
\[
\frac{M[t_1, \dots, t_n]}{(f_i=g_i)} \coloneqq  \mathrm{coeq}(\begin{tikzcd}
            \N e_1 \oplus \dots \oplus \N e_m \arrow[r, swap, shift right=.75ex, "u"] \arrow[r, shift right=-.75ex, "v"] & M[t_1, \dots, t_n]\end{tikzcd})
\]
where $u(e_i)=f_i$ and $v(e_i)=g_i$.

\begin{defn1} A monoid $P$ is called {\it integral} if the cancellation low holds in $P$, i.e. $a\cdot b=a\cdot c$ implies $b=c$ in $P$. \medskip 


A monoid $P$ is called {\it saturated} if it is integral and if whenever $p\in P^{\rm{gp}}$ is such that $np \in P$
 for some $n\in \Z_{\geq 1}$, then $p \in P$.  


A map of monoids $\phi\colon N \to M$ is called {\it finitely generated} if it can be extended to a surjective map of monoids $\psi\colon N[t_1, \dots, t_n] \to M$. \medskip

If, in addition, one can choose $\phi$ defined by an equivalence relation generated by finitely many relations $f_i = g_i$ with $f_i, g_i \in N[t_1, \dots, t_n]$ (i.e. $M=N[t_1, \dots, t_n]/(f_1=g_1, \dots, g_m=f_m)$), then we say that $\phi$ is {\it finitely presented}. \medskip

A monoid $P$ is called {\it finitely generated} if the natural map ${e} \to P$ from the trivial monoid is finitely generated. \medskip

A monoid $P$ is called {\it fine} if it is finitely generated and integral. \medskip

\end{defn1}

\begin{exmpl1} Let $\O_K$ be a valuation ring of rank-$1$, and let $\Gamma_{\leq 1} \subset \Gamma \subset (\mathbf{R}_{>0}, \times)$ be the monoid of elements of norm less or equal than $1$. Then $\Gamma_{\leq 1}$ is integral and saturated, but it is usually not finitely generated.  
\end{exmpl1}

\begin{defn1} A (global) {\it chart}\footnote{It is called by a (global) affine chart in \cite{AdipLiuPakTemkin}.} for the log structure $\M_X$ consists of a monoid $P$ and a homomorphism $P \to \Gamma(X, \O_X)$, equivalently a morphism $\ud{P}_X \to \O_X$, such that the associated log structure $\ud{P}_X^{\rm{log}}$ is isomorphic to $\M_X$. We will usually denote charts by $(P \to \M_X)$.
\end{defn1} 

\begin{rmk1}\label{chart-etale} Suppose $(X, \M_X)$ is a log scheme with a chart $(P \to \Gamma(X, \O_X))$ and $Y\to X$ is an \'etale morphism with the log structure $(Y, \M_X|_Y)$. Then the associated morphism $(P \to \Gamma(Y,\O_Y))$ defines a chart for $(Y, \M_X|_Y)$.
\end{rmk1}

\begin{lemma1}\label{finite-stalks} Let $(P \to \M_X)$ be a chart for a log scheme $(X, \M_X)$. Then the natural map $\ov{P} \to \ov{\M}_{X, \ov{x}}$ is surjective for any $x\in X$.
\end{lemma1}
\begin{proof}
The proof is essentially trivial and follows from the definitions. 
\end{proof}

\begin{defn1}\label{defn:fine-log-str}  A log scheme $X$ is called {\it quasi-coherent} if its log structure possesses charts \'etale locally. \medskip 

A quasi-coherent log scheme $X$ is called {\it integral} (resp. saturated, resp. coherent) if every point $x\in X$ \'etale locally admits a chart with an integral (resp. saturated, resp. finitely generated) monoid $P$. \medskip

A quasi-coherent log scheme $X$ is called {\it fine} if it is integral and coherent.
\end{defn1}

Note that it is not, a priori, clear if a fine log structure admits charts with fine monoids $P$. However, it turns out to always be possible due to the following lemma:

\begin{lemma1}\label{fs-log} Let $(X, \M_X)$ be a fine (resp. fine and saturated) log scheme. Then, \'etale locally on $X$, $(X, \M_X)$ admits a chart $(P \to \M_X)$ with a fine (resp. fine and saturated) monoid $P$.
\end{lemma1}
\begin{proof}
Suppose $(X, \M_X)$ is fine (resp. fine and saturated), then \cite[Proposition II.1.1.8(b)]{ogus-log} implies that $\M_X$ is fine (resp. fine and saturated) sheaf of monoids in \cite[Definition II.2.1.5 and discussion before Proposition I.1.1.3]{ogus-log}. Then \cite[Corollary II.2.3.6]{ogus-log} implies that, \'etale locally, it is possible to find charts $(P_x \to \M_X)$ with fine (resp. fine and saturated) monoids $P_x$.
\end{proof}

\begin{rmk1}\label{depend} We warn the reader that the notion of a fine log scheme is not independent of a choice of a chart $P$. For instance, let $\O_K$ be a dvr and $(X=\Spec \O_K, \mathcal{M}_X)$ the associated log scheme with the standard log structure. Then $(\O_K \setminus \{ 0 \} \to \mathcal{M}_X)$ is a chart and the monoid $\O_K \setminus \{0\}$ is usually not finitely generated. However, if we pick a uniformizer $\pi$ then the induced map $(\mathbf{N}\pi \to \mathcal{M}_X)$ is also a chart with finitely generated, integral monoid $\mathbf{N}\pi$. 
\end{rmk1}

\begin{rmk1} We also warn the reader that the notion of quasi-coherent log scheme is quite restrictive. Unlike the case of quasi-coherent sheaves in algebraic geometry, many natural constructions of log schemes are not quasi-coherent. For instance, the log structures associated with pairs $(X, Z)$ as in the Example~\ref{associated} are often not quasi-coherent unless $Z$ is a strict normal crossing divisor in a regular scheme $X$ (see Definition~\ref{snc}). \medskip

In the case of a strict normal crossing divisor $Z\subset X$ and the associated log structure $(X, \mathcal{M}_X)$, the key observation is that any point $x\in X$ admits a regular sequence $(t_1, \dots, t_m)$ generating the maximal ideal of $\O_{X,x}$ and a natural number $r$ such that the product $t_1\cdot \dots \cdot t_r$ generates the ideal of $Z$ in $\O_{X,x}$. One shows that the natural map $\oplus_{i=1}^r \mathbf{N}t_i \to \mathcal{M}_U$ is a fine chart in some neighborhood $U$ of a point $x$. See \cite[Proposition III.1.8.2 and III.1.7.3]{ogus-log} for a more detailed discussion. 
\end{rmk1}

\begin{exmpl1}\label{not-fine} The standard log structure on $\Spec \O_K$ for a rank-$1$ valuation ring $\O_K$ is integral and quasi-coherent similarly to example in Remark~\ref{depend}. However, it is not fine unless the associated valuation monoid $\Gamma_{\leq 1}$ is finitely generated. For example, it never happens if $K$ is algebraically closed. 
\end{exmpl1}

\begin{rmk1} We will be mostly concerned with integral quasi-coherent log schemes. However, the fineness assumption is too strong for our purposes since $\Spec \O_K$ with its standard log structure is usually not fine (Example~\ref{not-fine}). 
\end{rmk1}

\begin{defn1} Let $f\colon X' \to X$ be a morphism of schemes. 
\begin{itemize}
\item Let $(\mathcal{M}_{X'}, \a)$ be a pre-log structure on $X'$. We define its {pushforward} $f_*^{\mathrm{plog}}(\mathcal{M}_{X'}, \a)$, or just $f_*^{\mathrm{plog}}(\mathcal{M}_{X'})$, as the fiber product
\[
\begin{tikzcd}
f_*^{\mathrm{plog}}(\mathcal{M}_{X'})\arrow{d} \arrow{r}{f_*^{\mathrm{plog}}(\a)} & \O_X \arrow{d} \\
f_*(\mathcal{M}_{X'}) \arrow{r}{f_*(\a)} & f_*(\O_{X'})
\end{tikzcd}
\]
\item  Let $(\mathcal{M}_X, \a)$ be a pre-log structure on $X$. We define its {pullback} $f^*_{\mathrm{plog}}(\mathcal{M}_X, \a)$, or just $f^*_{\mathrm{plog}}(\mathcal{M}_X)$, as the pre-log structure given by $f^{-1}(\mathcal{M}_X) \to f^{-1}(\O_X) \to \O_{X'}$.
\end{itemize}
\end{defn1}

\begin{defn1} Let $f\colon X' \to X$ be a morphism of schemes. 
\begin{itemize}
\item Let $(\mathcal{M}_{X'}, \a)$ be a log structure on $X'$. We define its {pushforward} $f_*^{\log}(\mathcal{M}_{X'}, \a)$, or just $f_*^{\log}(\mathcal{M}_{X'})$, as the fiber product
\[
\begin{tikzcd}
f_*^{\log}(\mathcal{M}_{X'})\arrow{d} \arrow{r}{f_*^{\log}(\a)} & \O_X \arrow{d}{f^\#} \\
f_*(\mathcal{M}_{X'}) \arrow{r}{f_*(\a)} & f_*(\O_{X'})
\end{tikzcd}
\]
This is already a log structure as can be easily checked. 
\item  Let $(\mathcal{M}_X, \a)$ be a log structure on $X$. We define its {pullback} $f^*_{\log}(\mathcal{M}_X, \a)$, or just $f^*_{\log}(\mathcal{M}_X)$, as the log structure associated to the pre-log structure given by \[f^{-1}(\mathcal{M}_X) \xr{f^{-1}(\a)} f^{-1}(\O_X) \xr{f^\#} \O_{X'} \ . \]
\end{itemize}
\end{defn1}

\begin{lemma1} Let $f\colon X' \to X$ be a morphism of schemes. The functor $f^*_{\mathrm{plog}}\colon \mathbf{Plog}_{X'} \to \mathbf{Plog}_X$ is a left adjoint to the functor $f_*^{\mathrm{plog}}\colon \mathbf{Plog}_{X} \to \mathbf{Plog}_{X'}$. Similarly, $f_{\log}^*$ is left adjoint to $f_*^{\log}$.
\end{lemma1}
\begin{proof}
Straightforward and left as an exercise. 
\end{proof}

\begin{cor1}\label{plog-log} Let $f\colon X' \to X$ be a morphism of schemes and $\M_X$ a pre-log structure on $X$. Then $\left(f^*_{\mathrm{plog}}\left(\M_X\right)\right)^{\log}\simeq f^*_{\mathrm{log}}\left(\M_X^{\log}\right)$.
\end{cor1}
\begin{proof}
Note that the functors $f_*^{\mathrm{plog}}\circ r_{X'}\colon \mathbf{log}_{X'} \to \mathbf{Plog}_X$ and $r_X\circ f_*^{\mathrm{log}}$ are canonically identified simply by the definitions of $f_*^{\mathrm{plog}}$ and $f_*^{\mathrm{log}}$. This means that their left adjoints are canonically isomorphic. The first functor has a left adjoint $(-)^{\log}\circ f^*_{\mathrm{plog}}$ and the second functor has a left adjoint $f^*_{\mathrm{log}} \circ (-)^{\log}$. This says that $(f^*_{\mathrm{plog}}(\M_X))^{\log}$ and $f^*_{\mathrm{log}}(\M_X^{\log})$ are canonically isomorphic for any pre-log structure $\M_X$.
\end{proof}

\begin{defn1} A morphism of log schemes $f\colon (X, \M_X) \to (X', \M_{X'})$ is called {\it strict} if the natural map $f^*_{\log}(\M_{X'}) \to \M_X$ is an isomorphism.
\end{defn1}

\begin{rmk1} A chart for a log structure on an $R$-scheme $X$ can be understood as a strict morphism $X \to \Spec R[P]$ where the target is endowed with the natural log structure from Example~\ref{monoid-log}.
\end{rmk1}

\begin{exmpl1} Let $\O_K$ be a rank-$1$ valuation ring with the value group $\Gamma$ and the associated monoid $V \coloneqq \O_K \setminus \{0\}$\footnote{It is denoted as $R$ is \cite{AdipLiuPakTemkin}. However, we prefer this notation since $R$ can be easily confused with the base ring.}. Then its sharpening $\ov{V}$ is isomorphic to $\Gamma_{\leq 1}$. 
\end{exmpl1}

In what follows, $V$ denotes the monoid $\O_K \setminus \{0\}$ associated to the rank-$1$ valuation ring $\O_K$.

\begin{defn1}\label{log-variety-temkin} A {\it log variety} over a rank-$1$ valuation ring $\O_K$ is an integral quasi-coherent log scheme $(X, \M_X)$ such that the underlying morphism of schemes $X \to \Spec \O_K$ is flat of finite presentation and each homomorphism of monoids $\ov{V}=\Gamma_{\leq 1} \to \ov{\mathcal{M}}_{X, \ov{x}}$ is finitely generated. 
\end{defn1}

\begin{rmk1}\label{fpr-aut} The last condition is automatic if an $\O_K$-log scheme $(X, \M_X)$ admits charts that are finitely generated over $V$. Indeed, if $x\in X$ \'etale locally admits a chart given by a finitely generated $R$-monoid $M$, then the associated log structure is given by $\ud{M}^{\log}$. Therefore, $\ov{M} \to \ov{\ud{M}^{\log}_{X, \ov{x}}}\simeq \ov{\M}_{X,\ov{x}}$ is surjective by Lemma~\ref{finite-stalks}. Therefore, $\ov{\M}_{X,\ov{x}}$ is finitely generated over $\ov{V}\simeq \Gamma_{\leq 1}$.
\end{rmk1}

\subsection{Log Smoothness}

\begin{defn1}\label{defn-smoothness} Let $f\colon (X, \M_X) \to (Y, \M_Y)$ be a morphism of fine quasi-coherent log schemes with a choice of a fine chart $(Q\to \M_Y)$. Then $f$ is called {\it log smooth} if, for each $x\in X$,
\begin{itemize}
\item There is a finitely presented morphism of integral monoids $Q \to P_x$ such that the kernel and the torsion part of the cokernel $Q^{\rm{gp}} \to P_x^{\rm{gp}}$ are finite groups of orders invertible on $X$
\item There is an \'etale neighborhood $U \to X$ of $x$, such that the natural map $U \to Y$ factors as 
$U \xr{g} Y\times_{\Spec \Z[Q]} \Spec \Z[P_x] \to Y$ and the map $g$ is strict and \'etale (in the usual sense).
\end{itemize} 
\end{defn1}

\begin{rmk1} We note that \cite[Theorem 3.5]{Kato} shows that this definition does recover the standard notion of log smooth morphism \cite[Definition at the beginning of (3.3)]{Kato} defined in terms of the log analog of the infinitesimal criterion of smoothness. In particular, this definition does not depend on a choice of a chart $Q$. \end{rmk1}

\begin{exmpl1} A non-trivial example of a log smooth morphism of log varieties is given by a semi-stable degeneration over a dvr $\O_K$. More precisely, let $X\to \Spec \O_K$ be a strictly semi-stable $\O_K$-scheme. Consider the log scheme $(X, \mathcal{M}_X)$ with the log structure associated with the closed fiber $X_s$, and the log scheme $(\Spec \O_K, \mathcal{M}_K)$ with the classical log structure. Then the morphism $(X, \mathcal{M}_X) \to (\Spec \O_K, \mathcal{M}_K)$ is log smooth by \cite[Corollary IV.3.1.18]{ogus-log}. 
\end{exmpl1}

Following \cite{AdipLiuPakTemkin}, we give the following ad-hoc definition of log smooth log variety:

\begin{defn1}\label{log-smooth-temkin} Let $\O_K$ be a rank-$1$ valuation ring with a divisible value group $\Gamma$. Then we say that a log $\O_K$-variety $X$ is {\it log smooth} if for every $x\in X$ there is an \'etale neighborhood $U$ of $x$ with a strict \'etale morphism \[U \to \Spec \O_K \times_{\Spec \Z[V]} \Spec \Z[P]\] for some integral monoid $P$ with a finitely presented, injective morphism $V \hookrightarrow P$.
\end{defn1}

The main goal of this section is to show a successive semi-stable $C$-smooth curve fibration over a rank-$1$ valuation ring with algebraically closed fraction field is indeed a log smooth log variety. Moreover, we will actually show that the desired log structure is the log structure associated to its special fiber. However, before proving this claim we need to deal with a noetherian situation. In order to do so, we need to use the notion of log regular log structures. 

Unfortunately, it is quite difficult to define log regular log structure, so instead we only mention the main properties and refer to \cite[Definition 4.4.5]{Tsuji} or \cite[Theorem III.1.11.1]{ogus-log} for a precise definition.

\begin{facts1}\label{log-reg}
\begin{enumerate}
	\item\label{log-reg-1}\cite[Ex. III.1.11.9]{ogus-log} Suppose that $X$ is a regular noetherian scheme with a strict normal crossing divisor $D$. Then the log structure associated to $D$ is log regular. 
	\item\label{log-reg-2} Any log regular log scheme $(X, \M_X)$ is fine and saturated (see Lemma~\ref{fs-log}). 
	\item\label{log-reg-3}\cite[Prop. 2.6]{Niziol} Suppose that $(X, \M_X)$ is a log regular log scheme. Then the locus of triviality $X_{\rm{tr}}\coloneqq \{x\in X \ | \  \M_{X, \ov{X}}=\O_{X, \ov{x}}^\times \}$ is a dense open subset of $X$, and $\M_X \simeq \O_X \cap 	j_*\O_{X_{\rm{tr}}}^\times$, where $j\colon X_{\rm{tr}} \to X$ is the natural open immersion.
	\item\label{log-reg-4}\cite[Prop. IV.3.5.3]{ogus-log} Let $(X', \M_{X'}) \to (X, \M_X)$ be a log smooth morphism of fine 	saturated log schemes whose underlying schemes are locally noetherian. If $(X, \M_X)$ is log 	regular then so is $(X', \M_{X'})$.
\end{enumerate}
\end{facts1}

\begin{defn1} We say that a pair $(X, Z)$ of a locally noetherian scheme and closed subscheme $Z$ is {\it log regular} if the log structure associated to the pair $(X, Z)$ is log regular.
\end{defn1}

\begin{prop1}\label{noeth-case}\cite[Exp.\ VI, Prop.\ 1.9]{deGabber} Let $(Y, T)$ be a log regular pair. And let $f\colon X \to Y$ be a semi-stable curve fibration that is smooth over $Y\setminus T$. Then the pair $(X, f^{-1}(T))$ is also log regular and the canonical morphism of log schemes $X \to Y$ is log smooth.
\end{prop1}

\begin{proof}
We denote the complement of $T$ in $Y$ as $j\colon U\to Y$ and its pre-image in $X$ as $V$. Pick a point $x\in X$ with $y\coloneqq f(x)\in Y$. The statement is \'etale local on $X$ and $Y$, so we can replace $X$ and $Y$ with \'etale neighborhoods of $x$ and $y$ to assume that $Y=\Spec A$ is affine, admits a global fine and saturated chart $P \to \M_Y$ (Lemma~\ref{fs-log}), and $f \colon X\to Y$ factors as 
\[
\begin{tikzcd}
X\arrow{r}{h} \arrow{d}{f} & \Spec A[x,y]/(xy-a) \arrow{dl}{g}\\
Y &
\end{tikzcd}
\]
where $h$ is \'etale and $a\in A\cap \O(U)^\times$\footnote{If $f$ is smooth at $x\in X$, we can choose $a=1$, and otherwise we use Lemma~\ref{ss-alternative}.} ($U$ is dense in $Y$ by Fact~\ref{log-reg}(\ref{log-reg-3})). The condition that $a$ is invertible in $\O(U)$ comes from the fact that $X \to Y$ is assumed to be smooth over $U$. \medskip

Since $P$ is a chart for $\M_Y=\O_Y \cap j_*\O_U^\times$, we can pass to some \'etale neighborhood of $y$ to assume that $a=um$ for some $m\in P$ and $u\in \O_Y(Y)^\times$. Then we use the isomorphism $\Spec A[x,y]/(xy-a) \cong \Spec A[x,y]/(xy-m)$ to assume that $a$ lies in (the image of) $P$. \medskip

Now we define a log structure on $\Spec A[x,y]/(xy-a)$ associated to the monoid $Q\coloneqq P[x,y]/(xy=a)$ with the obvious map $\beta\colon Q \to A[x,y]/(xy-a)$. This is clearly a fine saturated monoid, and the homomorphism $\phi\colon P \to Q$ defines the homomorphism $\phi^{\rm{gp}}\colon P^{\rm{gp}} \to Q^{\rm{gp}}$ 
that can be identified with the natural inclusion $P^{\rm{gp}} \to P^{\rm{gp}}\oplus \Z$. Therefore, $\ker (\phi^{\rm{gp}})=\{e\}$ and $\coker(\phi^{\rm{gp}})\simeq \Z$, in particular, it is torsion-free. Finally, we note that the natural map 
\[
\Spec A[x,y]/(xy-a) \to \Spec A \times_{\Spec \Z[P]} \Spec \Z[Q]
\] 
is an isomorphism since $\Z[Q]\simeq \Z[P][x, y]/(xy-a)$\footnote{There is a natural map $\Z[Q] \to \Z[P][x, y]/(xy-a)$ that can be checked to be an isomorphism ``by hands''. Alternatively, see Remark~\ref{good-prop-long}.}. This implies that the map 
\[
(\Spec A[x,y]/(xy-a), Q^{\log}) \to (\Spec A, \M_Y)
\] is a log smooth morphism of fine, saturated log schemes. In particular, Fact~\ref{log-reg}(\ref{log-reg-4}) ensures that $(\Spec A[x,y]/(xy-a), Q^{\log})$ is log regular. And now Fact~\ref{log-reg}(\ref{log-reg-3}), in turn, guarantees that the log structure on $\Spec A[x,y]/(xy-a)$ coincides with the log structure associated to $g^{-1}(T)$. \smallskip

Finally, we come back to showing that $f\colon X \to Y$ is log smooth. We reduced the situation to the case that $X$ admits an \'etale map to $\Spec A[x,y]/(xy-a)$. Given all the work above, it is suffices to show that this map is strict when $X$ is provided with its log structure associated with $f^{-1}(T)$ and $\Spec A[x,y]/(xy-a)$ is provided with the log structure $Q^{\log}$ constructed above. However, we have already shown the latter log structure coincides with the log structure associated to $g^{-1}(T)$. And it is clear that the restriction of this log structure on $X$ exactly coincides with the log structure associated with $f^{-1}(T)$ as $h$ is \'etale. In other words, the map $h\colon X \to \Spec A[x,y]/(xy-a)$ is a strict \'etale morphism. This ensures that $f$ is log smooth according to Definition~\ref{defn-smoothness}. Finally, $(X, f^{-1}(T))$ is log regular by Fact~\ref{log-reg}(\ref{log-reg-4}).
\end{proof}

\begin{rmk1}\label{more} The proof actually shows more. If $Y$ has a global chart $P$, then \'etale locally $X$ has a chart of the form $P[x,y]/(xy=a)$. Strictly speaking, we constructed the global chart only in the model example $Y=\Spec A$ and $X=\Spec A[x,y]/(xy-a)$. However, it defines a chart for $X$ \'etale over $\Spec A[x,y]/(xy-a)$ by Remark~\ref{chart-etale}.
\end{rmk1}

\begin{cor1}\label{log-commutes} Let $g\colon (Y', T') \to (Y, T)$ be a ``strict'' morphism of log regular pairs, i.e. a morphism of schemes $g\colon Y' \to Y$ such that $g^{-1}(|T|)=|T'|$. And let $f\colon X \to Y$ be a semi-stable curve fibration that is smooth over $Y\setminus T$, and let $f'\colon X' \to Y'$ be its base-change. Then the natural morphism of log schemes
\[
(X', \O_{X'} \cap j'_*\O^\times_{f'^{-1}(Y' \setminus T')}) \to (X, \O_{X} \cap j_*\O^\times_{f^{-1}(Y \setminus T)}) \times_{(Y, \O_Y\cap i_*\O^\times_{Y\setminus T})} (Y', \O_{Y'}\cap i'_*\O^\times_{Y'\setminus T'})
\]
is an isomorphism.
\end{cor1}
\begin{proof}
The claim is \'etale local on $X$, $Y$ and $Y'$. So we can reduce to the case $X$, $Y=\Spec A$, and $Y'=\Spec A'$ are affine  with $(Y,  \O_Y\cap i_*\O^\times_{Y\setminus T})$ and $(Y', \O_{Y'}\cap i'_*\O^\times_{Y'\setminus T'})$ having global compatible charts $P$ and $P'$, and $X$ having a strict \'etale morphism to 
\[
\Spec A[x,y]/(xy-a)=\Spec A\times_{\Spec \Z[P]}\Spec \Z\left[\frac{P[x,y]}{(xy=a)}\right]
\] 
for some $a\in P$. This follows from Remark~\ref{more}. \smallskip 

So it suffices to prove the claim in the case $X=\Spec A[x,y]/(xy-a)$ with the chart $P[x,y]/(xy=a) \to \M_X$. But then its pullback along the log map $Y' \to Y$ is given by $\Spec A'[x,y]/(xy-a)$ with the chart $P'[x,y]/(xy=a)$. The proof of Proposition~\ref{noeth-case} (and Remark~\ref{more}) ensures that this is exactly the log structure associated with the closed subscheme $(f')^{-1}(T')$. 
\end{proof}

We are almost ready to show that any successive semi-stable $C$-smooth curve fibration has a structure of a log smooth log $\O_C$-variety. We introduce the following definition that will be useful in the proof. 

\begin{defn1}\label{long} If $M$ is a monoid and $a\in M$ an element, we define the {\it semi-stable model} monoid over $M$ as $F_a(M)\coloneqq M[x,y]/(xy=a)$. \smallskip

A {\it good sequence for a monoid $M$} $(a_n, \dots, a_1)$ is a sequence of elements $a_1 \in M$, $a_2\in F_{a_1}(M)$, $a_3\in F_{a_2}(F_{a_1}(M)), \dots$. \smallskip

If $\ud{a}=(a_n, \dots, a_1)$ is a good sequence for a monoid $M$, we define {\it successive semi-stable} monoid over it as $F_{\underline{a}}(M)\coloneqq F_{a_n}(F_{a_{n-1}}(\dots F_{a_1}(M)))$. \smallskip

Similarly, if $A$ is an algebra with an element $a\in A$, we define the {\it semi-stable model} algebra over $A$ as $G_a(A)\coloneqq A[x,y]/(xy-a)$. \smallskip

A {\it good sequence for an algebra $A$} $(a_n, \dots, a_1)$ is a sequence of elements $a_1 \in A$, $a_2\in G_{a_1}(A)$, $a_3\in G_{a_2}(G_{a_1}(A)), \dots$. \smallskip

And, if $\ud{a}=(a_n, \dots, a_1)$ is a good sequence for an algebra $A$, we define {\it successive semi-stable} algebra over it as $G_{\underline{a}}(A)\coloneqq G_{a_n}(G_{a_{n-1}}(\dots G_{a_1}(A)))$
\end{defn1}

\begin{rmk1}\label{good-prop-long} There is a natural isomorphism $\Z[F_{\ud{a}}(M)]\simeq G_{\ud{a}}(\Z[M])$. Indeed, it suffices to prove the claim for one element $a\in M$. Then we have a natural map 
\[
\Z\left[F_a\left(M\right)\right]=\Z\left[\frac{M\left[x,y\right]}{\left(xy=a\right)}\right] \to G_a\left(\Z\left[M\right]\right)=\frac{\Z\left[M\right]\left[x,y\right]}{\left(xy-a\right)}.
\] Firstly, we note that $\Z[M[x,y]] \simeq \Z[M][x,y]$, so it suffices to show that $\Z[-]$ commutes with coequalizers. However, it is a left adjoint functor to the forget functor from (commutative) algebras to (commutative) monoids. Therefore, it does commute with coequalizers. 
\end{rmk1}

\begin{rmk1}\label{good-prop-long-2} One can check ``by hand'' that $F_{\ud{a}}(M)$ is always finitely presented over $M$, and it is integral and saturated provided that so is $M$.
\end{rmk1}

\subsection{Main Result}

\begin{thm1}\label{log-str-semist} Let $\O_C$ be a rank-$1$ valuation ring with the algebraically closed fraction field $C$ and the residue field $k$. Let $f\colon X\to \Spec \O_C$ be a successive semi-stable $C$-smooth fibration. Then the log structure associated with the closed fiber $X_k$ defines the structure of a log smooth log $\O_C$-variety.
\end{thm1}
\begin{proof}
We start by using Lemma~\ref{val-approximation} to write $(\O_C, \pi)$ as a filtered colimit of noetherian, regular subrings $(A_i, t_i)_{i\in I}$ with $\rm{V}(t_i)_{\red}$ an snc divisor. We may replace each $A_i$ with the localization $(A_i)_{\m_C\cap A_i}$ to assume that all $A_i$ are local and the morphisms $A_i \to \O_C$ are local as well. In what follows, we denote the standard log structure on $\Spec \O_C$ by $\M_C$. \medskip

{\it Step 1. Spread $X$ over some $A_i$}: We present $X$ as a successive semi-stable curve fibration 
\[
X=X_n \xr{f_n} X_{n-1} \xr{f_{n-1}} \dots \to X_1 \xr{f_1} X_0=\Spec \O_C
\]
such that each $f_i$ is a relative semi-stable curve smooth over $C$-fibers. Now we use \cite[IV\textsubscript{3}, 8.8.2]{EGA} to successively spread this tower over some $A_i$ to get a tower 
\[
X^{(i)}\coloneqq X_{n, i} \xr{f_{n, i}} X_{n-1, i} \xr{f_{n-1, i}} \dots \to X_{1, i} \xr{f_{1, i}} X_{0, i}=\Spec A_i
\]
of $A_i$-schemes. A standard approximation argument ensures that, after possibly enlarging $i$, we can assume that each $f_{n,i}$ is a relative semi-stable curve whose restriction over $\Spec A_i[1/t_i]$ is smooth. Now we replace the filtered set $I$ with $I_{\geq i}$, and define $X_{k,j}\coloneqq X_{i,j}\times_{\Spec A_i} \Spec A_{j}$  for any $j\geq i$. Clearly, the tower 
\[
X^{(j)}\coloneqq X_{n, j} \xr{f_{n, j}} X_{n-1, j} \xr{f_{n-1, j}} \dots \to X_{1, j} \xr{f_{1, j}} X_{0, j}=\Spec A_j
\]
is a successive semi-stable curve fibration whose restriction over $\Spec A_j[1/t_j]$ is smooth.  
\medskip

{\it Step 2. Construct the log structure over the finite layer $A_i$}: Now we note that $\Spec A_i$ has a canonical log structure associated to the closed subset $\rm{V}(t_i)_{\red}$. Moreover, this log structure is log regular by Facts~\ref{log-reg}(\ref{log-reg-1}), and it admits a global chart with $P_i=\N^d$  where $d$ is the number of irreducible components of $\rm{V}(t_i)_{\red}$. Only the latter claim requires a justification, in fact, it follows from the discussion before \cite[Proposition III.1.7.3]{ogus-log} and the fact that, for each irreducible component $D_i\subset \rm{V}(t_i)_{\red}$, the ideal sheaf $\O(-D_i)$ has a global generator as $A_i$ is a local regular ring. We denote this log structure on $\Spec A_i$ by $\M_i$. 

Now we endow each $X_{k, i}$ with the log structure $\M_{X_{k,i}}$ associated with the closed subscheme $X_{k,i}\times_{\Spec A_i} \Spec (A_i/t_i)$. We use (the proof of) Proposition~\ref{noeth-case} and Remark~\ref{more} successively to show that $X^{(i)}=X_{n, i}$ \'etale locally admits a chart of the form $F_{\ud{a}}(P_i)$ for some good sequence $\ud{a}$ for the monoid $P_i$ (see Definition~\ref{long}). And, moreover, it admits a strict \'etale morphism 
\[
X_{i} \to \Spec A_i \times_{\Spec \Z[P_i]} \Spec \Z[F_{\ud{a}}(P_i)] \simeq \Spec A_i \times_{\Spec \Z[P_i]} \Spec G_{\ud{a}}(\Z[P_i]) \simeq \Spec G_{\ud{a}}(A_i)  
\]
where the target has the log structure associated with the chart $F_{\ud{a}}(P_i) \to G_{\ud{a}}(A_i) \cap \left(G_{\ud{a}}(A_i)[\frac{1}{t_i}]\right)^\times$. \medskip

{\it Step 3. Construct some structure of a log smooth log $\O_C$-variety on $X$}: Now we simply define the log structure on $X$ to be the fiber product $(X^{(i)}, \M_{X^{(i)}}) \times_{(\Spec A_i, \M_i)} (\Spec \O_C, \M_C)$ in the category of log schemes. Recall that \cite[Proposition III.2.1.2]{ogus-log} ensures that the functor sending a log scheme to underlying scheme commutes with fiber products. Thus the underlying scheme of the fiber product $(X^{(i)}, \M_{X^{(i)}}) \times_{(\Spec A_i, \M_i)} (\Spec \O_C, \M_C)$ is exactly $X^{(i)}\times_{\Spec A_i} \Spec \O_C = X$. Thus, this does define some log structure $(X, \M_{X, i})$. We prefer to use this notation to emphasize that we do not know at this point if this structure is independent of a choice of $i$. \smallskip

Moreover, $X$ with this log structure \'etale locally admits a strict \'etale morphism\footnote{In the formula below, we slightly abuse notation and consider $\ud{a}$ as a sequence for $V$ using the natural morphism $P_i \to A_i\cap A_i[\frac{1}{t_i}]^\times \to \O_C\setminus \{0\}=V$.} 
\begin{align*}
X \to & \Spec \O_C \times_{\Spec A_i} \Spec G_{\ov{a}}(A_i) \\
&\simeq \Spec G_{\ov{a}}(\O_C)  \\
&\simeq \Spec \O_C \times_{\Spec \Z[V]} \Spec G_{\ov{a}}(\Z[V])  \\
&\simeq \Spec \O_C \times_{\Spec \Z[V]} \Spec (\Z[F_{\ov{a}}(V)]) \\
\end{align*}
and the target has a chart given by $F_{\ov{a}}(V) \to G_{\ov{a}}(\O_C) \cap \left(G_{\ov{a}}(\O_C)[\frac{1}{\pi}]\right)^\times$. In particular, the log scheme $(X, \M_{X, i})$ admits charts \'etale locally with the associated monoids $F_{\ov{a}}(V)$. This already implies that the constructed log structure is quasi-coherent. Now we need to study properties of these monoids to make sure that $(X, \M_{X, i})$ is actually a log smooth log $\O_C$-variety. 

We recall that Remark~\ref{good-prop-long-2} guarantees that $F_{\ov{a}}(V)$ is $V$-finitely presented, integral and the natural map $V \to F_{\ov{a}}(V)$ is injective. In particular, $X$ becomes an integral quasi-coherent log scheme with $V$-finitely presented charts. This implies $X$ is a log $\O_C$-variety by Remark~\ref{fpr-aut}. Furthermore, $X$ is log smooth as it admits strict \'etale morphisms $X \to \Spec \O_C \times_{\Spec \Z[V]} \Spec (\Z[F_{\ov{a}}(V)])$ \'etale locally on $X$. 

{\it Step 4. Show that the log structure $\M_{X, i}$ is independent on $i$ and coincides with $\O_X \cap j_* \O_{X_C}^\times$}: Similarly to what we did in Step~$3$, we can define the log structure $\M_{X, j}$ on $X$ as the pullback $(X^{(j)}, \M_{X^{(j)}}) \times_{(\Spec A_j, \M_j)} (\Spec \O_C, \M_C)$ for any $j\geq i$. Corollary~\ref{log-commutes} guarantees that $(X, \M_{X, j})\simeq (X, \M_{X, i})$ as $(X^{(j)}, \M_{X^{(j)}}) \simeq (X^{(i)}, \M_{X^{(i)}}) \times_{(\Spec A_i, \M_i)} (\Spec A_j, \M_j)$. 
Now we consider the fiber square 
\[
\begin{tikzcd}
X \arrow{d}{f} \arrow{r}{g'^{(j)}} & X^{(j)} \arrow{d}{f^{(j)}} \\
\Spec \O_C \arrow{r}{g^{(j)}} & \Spec A_j
\end{tikzcd} 
\]
and denote the composition as $h^{(j)}\coloneqq f^{(j)} \circ g'^{(j)} = g^{(j)}\circ f$. The proof of \cite[Proposition III.2.1.2]{ogus-log} shows that $\M_{X, j}$ is given by the coproduct of morphisms of log structures ${h^{(j), *}_{\log}}(\M_j) \to f^*_{\log}(\M_C)$ and ${h^{(j), *}_{\log}}(\M_j) \to {g'^{(j), *}_{\log}}(\M_{X^{(j)}})$. We denote this coproduct as $f^*_{\log}(\M_C) \oplus^{\log}_{{h^{(j), *}_{\log}}(\M_j)} {g'^{(j), *}_{\log}}(\M_{X^{(j)}})$, and write $\M_{X, i}\simeq \colim \M_{X, j}$ as this system is just constant. We claim there are isomorphisms
\begin{align*}
\colim \M_{X, j} & \simeq \colim \left(f^*_{\log}(\M_C) \bigoplus^{\log}_{{h^{(j), *}_{\log}}(\M_j)} {g'^{(j), *}_{\log}}\left(\M_{X^{(j)}}\right)  \right) \\
& \simeq \left(\colim f^*_{\log}\left(\M_C\right)\right) \bigoplus^{\log}_{\colim {h^{(j), *}_{\log}}(\M_j)} \left(\colim {g'^{(j), *}_{\log}}\left(\M_{X^{(j)}}\right)\right) \\
&\simeq  f^*_{\log}(\M_C) \bigoplus^{\log}_{\colim f^*_{\log}\left({g^{(j), *}_{\log}}\M_{i}\right)} \left(\colim {g'^{(j), *}_{\log}}\left(\M_{X^{(j)}}\right)\right) \\
& \simeq  f^*_{\log}(\M_C) \bigoplus^{\log}_{f^*_{\log}\left(\colim \left(\left({g^{(j), *}_{\rm{plog}}}\M_j\right)^{\log}\right)\right)} \colim \left(\left({g'^{(j), *}_{\rm{plog}}}\M_{X^{(j)}}\right)^{\log}\right) \\
& \simeq  f^*_{\log}(\M_C) \bigoplus^{\log}_{f^*_{\log}\left( \left(\colim {g^{(j), *}_{\rm{plog}}}\M_j\right)^{\log}\right)} \left(\colim{g'^{(j), *}_{\rm{plog}}}\M_{X^{(j)}}\right)^{\log} \\
& \simeq  f^*_{\log}(\M_C) \bigoplus^{\log}_{f^*_{\log}\left( \M_C\right)} \left(\O_X \cap j_*\O^\times_{X_C}\right)\\
& \simeq \O_X \cap j_*\O^\times_{X_C} \ . 
\end{align*}

Now we explain each isomorphism. The first is just the definition of $\M_{X, j}$. The second comes from the fact that filtered colimits commute with push-outs. The third just uses the fact that $h^{(j)}=g^{(j)}\circ f$ and that the colimit of a constant system is isomorphic to that constant term. The fourth uses that $f^*_{\log}$ is left adjoint (so it commutes with all colimits) and Corollary~\ref{plog-log}. The fifth uses that $(-)^{\log}$ is left adjoint, thus it commutes with arbitrary colimits, where colimits are understood in the category of {\it pre-log structures}. The sixth uses that $\colim {g^{(j)}_{\rm{plog}}}^{*}\M_{i}$ is already isomorphic to $\M_C$ (as we will soon justify), and so it stays the same after applying $(-)^{\log}$. Similarly, $\colim{g'^{(j), *}_{\rm{plog}}}\M_{X^{(j)}}$ is already isomorphic to $\O_X \cap j_*\O_{X_C}^\times$ (as we will soon justify), so so it stays the same after applying $(-)^{\log}$. The last isomorphism is trivial. \medskip

So, overall, the only thing we are left to show is that $\colim {g^{(j), *}_{\rm{plog}}}\M_j \simeq \M_C$ and $\colim{g'^{(j), *}_{\rm{plog}}}\M_{X^{(j)}} \simeq \O_X \cap j_*\O^\times_{X_C}$. We show the second, and the proof of the first is similar. The pre-log pullback $g'^{(j), *}_{\rm{plog}} \M_{X^{(j)}}$ is given by
\[
g'^{(j), -1}\M_{X^{(j)}}= g'^{(j), -1} (\O_{X^{(j)}}\cap \iota^{(j)}_* \O^\times_{U^{(j)}}) \to g'^{(j), -1}\O_{X^{(j)}} \to \O_{X} \ .
\]
where $\iota^{(j)}\colon U^{(j)} \to X^{(j)}$ is the complement of $X^{(j)}\times_{\Spec A_j} \Spec A_j/t_j$. So the question boils down to showing that the natural morphism
\[
\colim g'^{(j), -1} (\O_{X^{(j)}}\cap \iota^{(j)}_* \O^\times_{U^{(j)}})) \to \O_{X} \cap j_* \O^\times_{X_C}
\]
is an isomorphism, where $\colim$ is understood as the colimit in the category of sheaves of monoids. Since $g'^{(j), -1}$ is exact, it commutes with intersection (or, actually, any fiber product). Moreover, filtered colimits commute with finite limits, so we can rewrite the map as:
\[
\beta\colon \colim g'^{(j), -1} \O_{X^{(j)}} \cap \colim g^{(j), -1} \iota^{(j)}_* \O^\times_{U^{(j)}}  \to \O_{X} \cap j_* \O^\times_{X_C}
\] 
that we want to be an isomorphism. Now \cite[IV\textsubscript{3}, 8.2.12]{EGA} implies that $\colim g'^{(j), -1} \O_{X^{(j)}}  \to \O_{X}$ is an isomorphism. In order to establish that the morphism $\beta$ is an isomorphism, it suffices to show a local section $f$ of $\O_X$ is invertible on $X_C$ if and only if it comes from some local section of $\O_{X^{(j)}}$ invertible on $U^{(j)}$ for some large $j$. We have already shown that $f$ comes from some finite level $i$, and so we only have to show that it becomes invertible on $U^{(j)}$ for some large $j\geq i$. This again follows from \cite[IV\textsubscript{3}, 8.2.12]{EGA} as $X_C= \lim U^{(j)}$. This finishes the proof that the  structure of the log smooth log $\O_C$-variety on $X$ constructed above coincides with the log structure associated to the special fiber. 
\end{proof}

\bibliography{biblio}

\end{document}